\documentclass{amsart}
\usepackage{math}
\usepackage{tikz,mathabx,mathrsfs,mathtools,amsrefs,wasysym,comment,lineno,fancyvrb}
\usetikzlibrary{cd,quotes,angles,decorations,shapes}
\usetikzlibrary{positioning,arrows,automata,calc}

\pgfdeclaredecoration{single line}{initial}{
  \state{initial}[width=\pgfdecoratedpathlength-1sp]{\pgfpathmoveto{\pgfpointorigin}}
  \state{final}{\pgfpathlineto{\pgfpointorigin}}
}

\newtheorem{mainthm}{Theorem}

\newtheorem{remark}{Remark}

\newcommand\Hom{\operatorname{Hom}}
\newcommand{\ZO}{\mathbb{E}} 

\newcommand\Graph{{\mathsf{Graph}}}
\newcommand\gfA{{\mathscr A}}
\newcommand\gfB{{\mathscr B}}
\newcommand\gfC{{\mathscr C}}
\newcommand\gfD{{\mathscr D}}
\newcommand\gfF{{\mathscr F}}
\newcommand\gfG{{\mathscr G}}
\newcommand\gfH{{\mathscr H}}
\newcommand\gfK{{\mathscr K}}
\newcommand\gfL{{\mathscr L}}
\newcommand\gfS{{\mathscr S}}
\newcommand\gfX{{\mathscr X}}
\newcommand\DL[1][p,q]{{\mathsf{DL}(#1)}}

\def\markin{\tikz[baseline=0]{\draw[line width=0.3pt,{Stealth[reversed,scale=0.7]}-] (0,-0.6ex) -- (0,0.6ex);}}
\def\markout{\tikz[baseline=0]{\draw[line width=0.3pt,-{Stealth[scale=0.7]}] (0,0.4ex) -- (0,1.6ex);}}
\def\markinout{\tikz[baseline=0]{\draw[line width=0.3pt,{Stealth[reversed,scale=0.7]}-{Stealth[scale=0.7]}] (0,-0.6ex) -- (0,1.6ex);}}

\newcommand\ellipsesplitnode[6]{\node[draw,ellipse split] (#1) at #2 {\rule{#3}{0pt}\rule{0pt}{#4}};
  \node at ($0.5*(#1.north)+0.5*(#1.center)$) {#5};
  \node at ($0.5*(#1.south)+0.5*(#1.center)$) {#6};}

\newcommand\shadepath[4]{
  \coordinate (s) at (#3);
  \coordinate (t) at (#4);
  \path[shade,bottom color=#1,top color=#2] ($(s)!0.6pt!90:(t)$) -- ($(t)!0.6pt!-90:(s)$) -- ($(t)!0.6pt!90:(s)$) -- ($(s)!0.6pt!-90:(t)$) -- cycle;
}

\begin{document}
\title{Simulations and the Lamplighter group}
\author{Laurent Bartholdi}
\address{Mathematisches Institut, Georg-August Universit\"at zu G\"ottingen\and \'Ecole Normale Sup\'erieure, Lyon} 
\email{laurent.bartholdi@gmail.com}

\author{Ville Salo}
\address{University of Turku, Turku}
\email{vosalo@utu.fi}

\date{January 10th, 2021}
\begin{abstract}
  We introduce a notion of ``simulation'' for labelled graphs, in
  which edges of the simulated graph are realized by regular
  expressions in the simulating graph, and prove that the tiling
  problem (aka ``domino problem'') for the simulating graph is at
  least as difficult as that for the simulated graph.

  We apply this to the Cayley graph of the ``lamplighter group''
  $L=\Z/2\wr\Z$, and more generally to ``Diestel-Leader graphs''. We
  prove that these graphs simulate the plane, and thus deduce that the
  seeded tiling problem is unsolvable on the group $L$.

  We note that $L$ does not contain any plane in its Cayley graph, so
  our undecidability criterion by simulation covers cases not covered
  by Jeandel's criterion based on translation-like action of a product
  of finitely generated infinite groups.

  Our approach to tiling problems is strongly based on categorical
  constructions in graph theory.
\end{abstract}
\maketitle

\section{Introduction}

Let $G$ be a finitely generated group, with finite generating set
$S=S^{-1}$. For a given finite set $A$ and a subset
$\Pi\subseteq A\times S\times A$ called \emph{tileset}, a
\emph{tiling} of $G$ is the choice, for every $g\in G$, of a colour
$\phi_g\in A$ such that neighbouring group elements have matching
colours: for every $g\in G,s\in S$ we have
$(\phi_g,s,\phi_{g s})\in\Pi$.

The \emph{tiling problem} for $G$ asks for an algorithm that, given
$\Pi$, determines whether there exists such a tiling. A useful
variant, the \emph{seeded tiling problem}, asks for an algorithm that,
given $\Pi$ and a seed colour $a_0\in A$, determines whether there
exists a tiling with colour $a_0$ at the origin.

These problems have attracted much attention since
Wang's~\cite{wang:ptpr2} and Berger's~\cite{berger:undecidability}
results that the seeded tiling problem, respectively tiling problem
are unsolvable for $G=\Z^2$.

More generally, groups that are geometrically covered by sufficiently
regular planes also have unsolvable tiling problems; see the next
section for more precise statements. It is also easy to see that
groups with unsolvable word problem have unsolvable tiling
problem. To the best of our knowledge, all groups known up to now to
have unsolvable tiling problem fall in one of these two classes.

On the other hand, straightforward considerations show that the both
tiling problems are solvable for free groups, and more generally
virtually free groups (namely groups with a free subgroup of finite
index).

The main contribution of this article is a proof of undecidability for
a group not covered by these classes, the \emph{lamplighter
  group}. This is the wreath product $L=\Z/2\wr\Z$, and also the group
generated by the affine transformations $a(f)=t f$ and $b(f)=t f+1$ of
the ring $\mathbb F_2[t,t^{-1}]$; it admits as presentation
\[L=\langle a,b\mid (a^n b^{-n})^2\text{ for all }n\ge1\rangle.
\]
This group is not virtually free, and it does not contain
geometrically any plane or hyperbolic plane. We prove:
\begin{mainthm}\label{thm:main}
  The seeded tiling problem for the lamplighter group $L$ is unsolvable.
\end{mainthm}

We view the tiling problem as a question about marked graphs; in our
setting, the Cayley graph of the group $L$, namely the graph with
vertex set $L$ and an edge between $g$ and $g s$ for all
$s\in\{a^{\pm1},b^{\pm1}\}$.

We prove Theorem~\ref{thm:main} by introducing a notion of
``simulation'': even though $\Z^2$ is not contained geometrically in
$L$, it is contained ``automatically'', in that there is a function
$\Z^2\to L$ which is not Lipschitz but maps edges of $\Z^2$'s Cayley
graph to regular expressions in edges of $L$'s Cayley graph; these
regular expressions are driven by an auxiliary labelling provided to
the Cayley graph of $L$ by a subshift of finite type. This is
sufficient to reduce the seeded tiling problem of $\Z^2$ to that of
$L$, and therefore conclude with the latter's undecidability.

Our construction is fundamentally graph-theoretical. It applies in
particular, with trivial modifications, to all Diestel-Leader graphs
$\DL$; see~\S\ref{ss:DL}; up to subtleties regarding edge labellings,
$\DL[2,2]$ is just the Cayley graph of $L$. We prove in fact:
\begin{mainthm}\label{thm:main2}
  The seeded tiling problem for the Diestel-Leader graph $DL$ is
  unsolvable.
\end{mainthm}

\noindent We leave as an interesting open problem
\begin{conj}
  The tiling problem of the lamplighter group and Diestel-Leader
  graphs are undecidable.
\end{conj}

\subsection{Groups with unsolvable (seeded) tiling problem}
As we mentioned in the introduction, the tiling problem on $G$ is
solvable if $G$ is virtually free~\cite{muller-s:context-free-2}. In fact,
if a tileset $\Pi$ tiles at all, then it admits a ``rational'' tiling,
namely there are regular languages $(R_a\subseteq S^*)_{a\in A}$ such
that $g\in G$ is coloured $a$ if $g$ is in the image of $R_a$ under
the natural evaluation map $S^*\twoheadrightarrow G$. A tiling
algorithm therefore in parallel lists tilings of larger and larger
balls in $G$, and tests $A$-tuples of regular languages in $S^*$ for
valid tilings, and we are guaranteed that one of the processes will
stop. See~\cite{aubrun-barbieri-jeandel:domino}*{Theorem~9.3.37}.
No other groups are known to have solvable tiling problem:
\begin{conj}[Ballier \&\ Stein~\cite{ballier-stein:pgdomino}]\label{conj:bs}
  A finitely generated group has solvable domino problem if and only
  if it is virtually free.
\end{conj}

There has been, since 2013 (when~\cite{ballier-stein:pgdomino}
appeared as a preprint), continuous progress towards
Conjecture~\ref{conj:bs}; here is a brief list of elementary results,
in which all groups are assumed to be finitely generated.
\begin{itemize}
\item If $G$ has unsolvable word problem (namely there is no algorithm
  determining, with input a word $w$ in $G$'s generators whether $w=1$
  holds in $G$), then there is \emph{a fortiori} no algorithm
  determining whether a tileset tiles;
  see~\cite{aubrun-barbieri-jeandel:domino}*{Theorem~9.3.28}.
\item The fundamental results of Wang and Berger show that $\Z^2$ has
  unsolvable tiling problem.
\item If $G,H$ are commensurable (meaning they have finite-index
  isomorphic subgroups) then $G$ has solvable tiling problem if and
  only if $H$ does.
\item If $H$ is a finitely generated subgroup of $G$ with unsolvable
  tiling problem, then $G$'s tiling problem is also unsolvable;
  see~\cite{aubrun-barbieri-jeandel:domino}*{Proposition~9.3.30}.
\item If $N$ is a normal finitely generated subgroup of $G$ and $G/N$
  has unsolvable tiling problem, then so does $G$;
  see~\cite{aubrun-barbieri-jeandel:domino}*{Proposition~9.3.32}.
\end{itemize}

Recall that two finitely generated groups $G,H$ are
\emph{quasi-isometric} if there are Lipschitz maps $G\to H$ and
$H\to G$ whose compositions are at bounded sup-distance from the
identity. Cohen proves that the tiling problem is geometric in the
following sense:
\begin{thm}[\cite{cohen:geometrysubshifts}]\label{thm:qi}
  If $G,H$ are finitely presented and quasi-isometric, then $G$ has
  solvable tiling problem if and only if $H$ does.
\end{thm}

Recall next that a finitely generated group $H$ acts
\emph{translation-like} on $G$ if the action is free and by self-maps
of $G$ at bounded sup-distance from the identity. This notion was
introduced by Whyte~\cite{whyte:amenability};
Seward~\cite{seward:burnside} proved that a finitely generated group
is infinite if and only if it admits a translation-like action of
$\Z$.
\begin{thm}[\cite{jeandel:translation}*{Theorem~3}]\label{thm:trans}
  If $H$ is finitely presented, has unsolvable tiling problem, and
  acts translation-like on $G$, then $G$ has unsolvable tiling
  problem.
\end{thm}
Thus every group containing a subgroup of the form $H_1\times H_2$
with $H_1,H_2$ infinite and finitely generated has unsolvable tiling problem. This applies in
particular to ``branched groups'' such as the first Grigorchuk group.
The current ``state of the art'' also includes the following groups as
having unsolvable tiling problem:
\begin{itemize}
\item Non-virtually-cyclic virtually
  nilpotent~\cite{ballier-stein:pgdomino} and more generally virtually
  polycyclic groups~\cite{jeandel:polycyclic};
  see~\cite{aubrun-barbieri-jeandel:domino}*{Theorem~9.3.43}.
\item Baumslag-Solitar groups~\cite{aubrun-kari:bs};
  see~\cite{aubrun-barbieri-jeandel:domino}*{Theorem~9.3.47}.
\item Fundamental groups of closed
  surfaces~\cite{aubrun-barbieri-moutot:dominosurface}.
\end{itemize}
  
Note that all the above examples geometrically contain a plane, either
Euclidean or hyperbolic, in their Cayley graph. In fact, Ballier and
Stein prove more generally that if $G$ is not virtually cyclic but
contains an infinite cyclic central subgroup then its tiling problem
is unsolvable.

We have not yet addressed an important, related question: given a
group $G$, does there exist a finite tileset that admits tilings, but
only aperiodic ones? There are several variants to this question, in
particular a \emph{weakly aperiodic} tileset is such that all tilings
have infinite orbits (equivalently: infinite-index stabilizer) under
translation by $G$, a \emph{strongly aperiodic} tileset is such that
all tilings have trivial stabilizer.

These notions are equivalent for $\Z^2$, but differ in general: Cohen
constructs in~\cite{cohen:lamplighters} a weakly aperiodic tileset for
the lamplighter group $L$, but no strongly aperiodic tileset is known
for $L$.

\subsection{The general tiling problem of the lamplighter group}

It remains open whether the (unseeded) tiling problem of the
lamplighter group is solvable.  All the standard methods used to prove
the tiling problem undecidable on the plane and other finitely
generated groups seem plausible on the lamplighter group: the
lamplighter group is residually finite and one can imagine tiling it
by ``macro-tiles'' that form another copy of the lamplighter (or
related) group, and one could try a Robinson-like construction or the
fixed-point methods for this.  However, the group is too disconnected
to make a direct implementation possible.  To implement this method,
we suspect the obstacle to overcome is to find a strongly aperiodic
tileset.

Another approach is the transducer method, which applies to groups
equipped with a homomorphism onto $\Z$ (``indicable groups''). On each
fibre, we write numbers from a finite set, their average density
represents a real number, and a transducer verifies that consecutive
fibres have a density related by a piecewise-affine map without
periodic points. This strategy succeeds on Baumslag-Solitar
groups~\cite{aubrun-kari:bs}; on the lamplighter group, it cannot work
directly, because the fibres are disconnected as graphs (no matter
what the generating set is). It would be interesting to understand
which $\Z$-subshifts have a sofic pullback to the lamplighter group
along the natural homomorphism. We show in
Proposition~\ref{prop:sealevelSofic} that the sunny-side-up on $\Z$
has a sofic lift, while for the transducer method it seems one would
need to lift a full shift (though it is not clear this is sufficient).

\subsection{Simulations}
We consider the (seeded) tiling problem in the context of edge- and
vertex-labelled graphs $\gfG$; when applied to a group
$G=\langle S\rangle$, the graph we consider is the Cayley graph of
$G$, which has one vertex per group element and an edge labelled $s$
from $g$ to $g s$ for all $g\in G$ and all $s\in S$. In the case of a
seeded tiling problem, the vertex $1\in G$ furthermore has a special
marking. A tileset is then a collection of allowable colourings of
vertices and their abutting edges and neighbours.

Our strategy to prove that a graph $\gfG$ has unsolvable tiling
problem can be summarized as follows: define a preorder relation
``simulate'' on a class of graphs containing $\gfG$, and then
\begin{enumerate}
\item show that if a graph $\gfA$ simulates $\gfB$ and $\gfB$ has
  unsolvable tiling problem, then so does $\gfA$;
\item choose a graph $\gfH$ with unsolvable tiling problem;
\item show that $\gfG$ simulates $\gfH$.
\end{enumerate}

This is uninteresting if ``simulate'' is the equality relation; the
coarser ``simulate'' is, the harder (1) becomes and the easier (3)
becomes. Taking ``simulate'' to mean ``admits a translation-like
action by'' interprets Theorem~\ref{thm:trans} in this context.

We propose a more general notion of ``simulation'', in which the edges
of the simulated graph are not paths of bounded length in the
simulating graph (as would be the case in a translation-like action)
but rather are given by regular expressions, or more appropriately by
graph-walking finite state automata. These automata simulate edges by
following edge and vertex labellings on the simulating graph, which
typically is drawn by a legal tiling over an auxiliary fixed tile
set. An appropriate image, back to the context of groups, is that when
$G$ simulates $H$, there is for every generator $t$ of $H$ an ant that
reacts to pheromones deposited on the Cayley graph of $G$; when
started at a vertex representing an element $h\in H$, it will move
according to its finite state table and pheromone inputs until it
reaches a stop state, which will then represent $h t\in H$.

The ``automaton'' or ``ant'' picture is useful to describe actual
simulations, but is better expressed in the more formal setting of a
``bigraph'', namely a graph equipped with two graph morphisms towards
finite graphs. These should be thought of as generalizations of graph
morphisms (in case one of the morphisms is the identity).

\noindent We realize our plan, towards the proof of
Theorem~\ref{thm:main}, as follows:
\begin{enumerate}
\item is proven as Theorem~\ref{thm:simul} and Theorem~\ref{thm:simulfamily};
\item we use the quandrant $\N^2$ or the plane $\Z^2$;
\item is proven in two different manners: using a ``comb'', in
  Proposition~\ref{prop:combsimul}, or a ``sea level'', in
  Proposition~\ref{prop:seasimul}. 
\end{enumerate}
Both of the simulations require a marking of $L$ by tiles from given
``pheromone'' tileset (whose size we could bring down to $6$ tiles in
the case of the ``comb''). We have proven, therefore, that the Cayley
graph of $L$, when decorated by this tileset, has unsolvable tiling
problem. However, this immediately implies that $L$ itself also has
unsolvable seeded tiling problem, since the ``pheromone'' tileset can
be combined to the instance to be solved. We need a seed tile so as to
ensure that there is a non-trivial decoration from the ``pheromone''
tileset; indeed we have not been able to find a tile set on $L$ \emph{all}
of whose configurations simulate a quandrant or $\Z^2$.

\begin{proof}[Summary of the proof of Theorem~\ref{thm:main}]
  Let $\Pi$ be an instance of a tiling problem for the quadrant. The
  ``sea level'' is a tiling problem $\Pi_{\aquarius}$ for $L$, which
  (after specifying a seed tile) has a unique solution. The grid may
  be simulated using the labels on $L$ coming from this tiling,
  leading to a tiling problem $\Pi'$ for $L$. If $L$'s seeded tiling
  problem were solvable, it would in particular be solvable for the
  product of $\Pi_{\aquarius}$ and $\Pi'$, and therefore lead to a
  solution of the tiling problem $\Pi$ on the quadrant, a
  contradiction.

  We could, alternatively, use the ``comb'' tiling problem instead of
  the ``sea level''. Then the induced labelling on $L$ is not unique,
  and forms a family of tilings, some of which simulate grids. The
  proof would then proceed identically.
\end{proof}

The reason we exhibit two different constructions is that each has its
own additional benefits. The comb gives a better proof in terms of
conciseness: the proof is shorter, and indeed Section~\ref{ss:comb}
(skipping Proposition~\ref{prop:combsimul} and
Corollary~\ref{cor:firstproof}) give a self-contained proof of the
main theorem using the comb SFT.

On the other hand, the sea level SFT admits a unique seeded
configuration which simplifies the application of our simulation
theorem to it), and is more efficient (an $n\times n$ grid requires
roughly $\log(n)n^2$ vertices in the lamplighter graph). In the
process, we shall learn more about the ``SFT-accessible geometry'' of
the lamplighter group: recalling that the lamplighter group is a split
extension $\mathbb F_2[a,a^{-1}]\rtimes\langle a\rangle$, we show that
the one-point compactified actions on the coset spaces of the
subgroups $\mathbb F_2[a,a^{-1}]$ and $\langle a \rangle$ are sofic
shifts, and the sea level SFT is obtained by combining their covering
SFTs with a simple rule. The sea level construction also does not
require the rigidity of a vertex-transitive graph, and with minor
modifications applies to the Diestel-Leader graphs $\DL$. We feel
these ideas are more likely to be useful in further work on the group
than the ad hoc tricks used for the comb.

\subsection{Acknowledgments}
We are grateful to the UMPA and LIP at the \'Ecole Normale
Sup\'erieure, Lyon for its hospitality, and in particular to Nathalie
Aubrun for her interest and support.

\section{Graphs and their simulations}
We begin by some general notions on graphs. They rely heavily on
categorical ideas, but we have taken care to explain them in direct
terms, only hinting at their abstract origin. The reader who is
unfamiliar with categorical language is strongly encouraged to ignore
every sentence that matches `\texttt{.*[Cc]ategor[iy].*}'.

\subsection{Graphs}
\begin{defn}
  A \emph{graph} is a set $\gfG=V\sqcup E$ partitioned into
  \emph{vertices} and \emph{edges}, with two maps
  $e\mapsto e^\pm\colon E\to V$; we call $e^+$ the \emph{head} of $e$
  and call $e^-$ its \emph{tail}. We write $V(\gfG)$ and $E(\gfG)$ for
  the vertices and edges of a graph $\gfG$.

  Graphs are born oriented. An \emph{unoriented} graph is a graph
  endowed with an involution $e\mapsto e'\colon E\to E$ on its edges
  called the \emph{reversal}, satisfying $(e')^+=e^-$ for all
  $e\in E$.

  The \emph{closure} of an edge $e\in E$ is the subgraph
  $\overline e=\{e,e^+,e^-\}$ in the oriented case and
  $\overline e=\{e,e',e^+,e^-\}$ in the unoriented case, and the closure of a
  vertex $v\in V$ is $\overline v=\{v\}$.
  
  A \emph{morphism of graphs} is a pair of maps between their vertices
  and edges respectively, which interlace the head, tail and (in the
  unoriented case) reversal operations. We write
  $\phi\colon\gfG\to\gfH$ a morphism, using the same symbol for the
  map on vertices and on edges. A morphism $\phi\colon\gfG\to\gfH$ is
  \emph{\'etale} if it is locally injective: if two edges
  $e,f\in E(\gfG)$ satisfy $\phi(e)=\phi(f)$ and ($e^-=f^-$ or
  $e^+=f^+$) then $e=f$. A graph is \emph{weakly \'etale} if all
  edges which have the same label and share an endpoint actually
  share both endpoints.
\end{defn}

Graphs, with graph morphisms, form a category $\Graph$. In the next
sections, we shall display important properties of this category. We
write $\Graph(\gfG,\gfH)$ the set of morphisms $\gfG\to\gfH$, or just
$\Hom(\gfG,\gfH)$ if we don't want to precise the subcategory of
$\Graph$ under consideration.

\subsection{Labelled graphs}
Consider a graph $\gfG$, with labels on its vertices and edges. We
thus have maps $\alpha\colon V(\gfG)\to A$ and
$\alpha\colon E(\gfG)\to C$ to label sets $A,C$. Without loss of
generality, we may assume that every label $c\in C$ ``knows'' the
label of its extremities, so we have maps
$c\mapsto c^{\pm}\colon C\to A$. Indeed at worst replace $C$ by
$A\times C\times A$ and extend the edge labelling to a map
$E(\gfG)\to A\times C\times A$ by
$e\mapsto(\alpha(e^-),\alpha(e),\alpha(e^+))$.

Thus a labelling of $\gfG$ is nothing but a graph morphism
$\alpha\colon\gfG\to A\sqcup C$. From now on we shall write
$\alpha\colon\gfG\to\gfA$ for a labelling on $\gfG$, with
$\gfA=A\sqcup C$.

Again, in the unoriented case, we assume that there is an involution
on $C$ compatible with the edge involution on $\gfG$.

\begin{exple}\label{ex:cayley}
  Let $G$ be a group with generating set $S$. An example we shall
  return to in much more detail is the \emph{Cayley graph} of $G$: it
  is the graph $\gfC=G\sqcup(G\times S)$ with $(g,s)^-=g$ and
  $(g,s)^+=g s$. In case $S=S^{-1}$, we also have an edge reversal
  given by $(g,s)'=(g s,s^{-1})$.

  The Cayley graph is naturally an edge labelled graph, in which the
  edge $(g,s)$ has label $s$. We express this as a map
  $\gfC(G,S)\to 1\sqcup S$, where $1\sqcup S$ is the graph with one
  vertex and $S$ loops at it.
\end{exple}

\begin{exple}
  A bipartite graph is a graph given with a partition of its vertices
  into two parts $V_0\sqcup V_1$, and such that all edges cross
  between $V_0$ and $V_1$.

  A bipartite graph is then naturally a graph equipped with a
  labelling to the segment, namely to the graph with vertex set
  $\{0,1\}$ and two edges $e,e'$ with $e^+=(e')^-=1$ and
  $e^-=(e')^+=0$.
\end{exple}

\begin{exple}\label{ex:ssu}
  Let $G$ be a monoid with generating set $S$. The element $1\in G$
  may be distinguished, leading to a \emph{rooted Cayley graph}, or
  \emph{sunny side up}; this last terminology because one vertex is
  marked as \textsf{vitellus} (yolk) and all the others as
  \textsf{albumen} (white). Thus the rooted Cayley graph is expressed
  by a map
  $\mathscr
  C(G,S)\to\{\mathsf{vitellus},\mathsf{albumen}\}\sqcup(\{\mathsf{vitellus},\mathsf{albumen}\}\times S\times\{\mathsf{vitellus},\mathsf{albumen}\})$, with $(x,s,y)^-=x$ and $(x,s,y)^+=y$; namely the map defined on
  vertices by $1\mapsto\mathsf{vitellus}$ and
  $g\neq1\mapsto\mathsf{albumen}$, and likewise on edges.
\end{exple}

In categorical language, the $\gfA$-labelled graphs, namely the graphs
with labelling by a graph $\gfA$, form the \emph{slice} category
$\Graph_{/\gfA}$; this is the category whose objects are morphisms
$\alpha\colon\gfG\to\gfA$ and whose morphisms between
$\alpha\colon\gfG\to\gfA$ and $\beta\colon\gfH\to\gfA$ are usual graph
morphisms $\gamma\colon\gfG\to\gfH$ satisfying $\alpha=\beta\circ\gamma$.

The \emph{simplification} of a labelled graph $\gfG$ is the graph
where all edges $e$ with the same endpoints and label are merged into
a single edge, and the \emph{full simplification} is obtained from the
simplification by also removing all self-loops.  Note that a graph is
weakly \'etale if and only if its simplification is \'etale.

\subsection{Pullbacks and exponentials}
Consider $\gfG_1,\gfG_2$ two $\gfA$-labelled graphs, namely graphs
equipped with labellings $\alpha_i\colon\gfG_i\to\gfA$. Their
\emph{pullback} is the graph
\[\gfG_1\times_\gfA\gfG_2=\{(u_1,u_2)\in\gfG_1\times\gfG_2:\alpha_1(v_1)=\alpha_2(v_2)\},\]
labelled by $\alpha(u_1,u_2)=\alpha_1(u_1)$. The graph structure is
given by $(e_1,e_2)^\pm=(e_1^\pm,e_2^\pm)$, and similarly for the
reversal.

In some sense, $\gfG_1\times_\gfA\gfG_2$ is \emph{universal}: it is
the ``best'' way to construct an $\gfA$-labelled graph equipped with
morphisms to $\gfG_1$ and $\gfG_2$ (which of course are given by
projections to the first and second factor). The universal property
says that, for every graph $\gfH$ also equipped with morphisms
$\phi_i\colon\gfH\to\gfG_i$, there is a unique map
$\eta\colon\gfH\to\gfG_1\times_\gfA\gfG_2$ with
$\phi_i=\pi_i\circ\eta$.

\begin{exple}
  Consider $\gfG_2=\gfA\sqcup\gfA$, namely two disjoint copies of
  $\gfA$'s vertices and edges. Then
  $\gfG_1\times_\gfA\gfG_2\cong\gfG_1\sqcup\gfG_1$.

  Consider next $\gfG_2$ a subgraph of $\gfA$, for example the
  subgraph spanned by a subset $A'$ of $\gfA$'s vertices. Then
  $\gfG_1\times_\gfA\gfG_2$ is the subgraph of $\gfG_1$ whose
  labelling belongs to $\gfG_2$, for example the subgraph of $\gfG_1$
  spanned by vertices whose label is in $A'$.
\end{exple}

\begin{exple}
  Let $G_1,G_2$ be two groups with respective generating sets
  $S_1,S_2$. Then the Cayley graph $\gfC(G_1\times G_2,S_1\times S_2)$
  is the usual direct product of the graphs $\gfC(G_1,S_1)$ and
  $\gfC(G_2,S_2)$; it is a connected graph precisely when
  $S_1\times S_2$ generates $G_1\times G_2$. This direct product is
  the pullback $\gfC(G_1,S_1)\times_{1\sqcup 1}\gfC(G_2,S_2)$ over the
  trivial graph with one vertex and one edge. More generally, let
  there be maps $f_1\colon S_1\to S$ and $f_2\colon S_2\to S$; then
  the Cayley graph of $\gfC(G_1\times G_2,S_1\times_S S_2)$ is
  $\gfC(G_1,S_1)\times_{1\sqcup S}\gfC(G_2,S_2)$, naturally viewing
  $\gfC(G_i,S_i)\in\Graph_{/1\sqcup S}$ via $f_i$.
\end{exple}

Consider two $\gfB$-labelled graphs $\gfG_1,\gfG_2$, and let
$\alpha\colon\gfG_2\to\gfA$ be an $\gfA$-labelling on $\gfG_2$. The \emph{exponential}\footnote{This is not actually an exponential object in a category $\Graph_{/\gfB}$, because of the different categories $\Graph_{/\gfA}$, $\Graph_{/\gfB}$ involved. It does however produce exponential objects, in the sense of adjoints to fibre products, when $\gfA=\gfB$ and the $\gfA$- and $\gfB$-labellings coincide; see Lemma~\ref{lem:adjoint1}.} of these graphs is the graph
\[\gfG_1^{\gfG_2}=\{(f,a):a\in\gfA,f\in\Graph_{/\gfB}(\alpha^{-1}(\overline a),\gfG_1)\}.\]
There is a natural $\gfA$-labelling on $\gfG_1^{\gfG_2}$ by projection
to the second coordinate. The graph structure on $\gfG_1^{\gfG_2}$ is
given by $(f,a)^\pm=(f\restriction \alpha^{-1}(a^\pm),a^\pm)$. In the
unoriented case, we have $(f,a)'=(f',a')$ with $f'(e)=f(e')$ and
$f'(e^\pm)=f(e^\mp)$.

In words, an edge in $\gfG_1^{\gfG_2}$ with label $c\in E(\gfA)$ is a
function defined on the $\alpha$-preimages of $c$ and its endpoints,
with values in $\gfG_1$, which preserves the graph structure and the
$\gfB$-labelling.

\begin{exple}\label{ex:2sets}
  Consider the graph $\gfB$ with vertex set $\{0,1\}$ and no
  edge. Thus $\Graph_{/\gfB}$ is the category of sets $S$ partitioned
  in two parts $S_0\sqcup S_1$. Choose $\gfA=\gfB$ and let $\alpha$
  coincide with the $\gfB$-labelling. Unwrapping the definition, we see
  $(S_0\sqcup S_1)^{T_0\sqcup T_1}=S_0^{T_0}\sqcup S_1^{T_1}$, where
  in the last expression $S_i^{T_i}$ is as usual the set of maps
  $T_i\to S_i$.
\end{exple}

The classical bijection $A^{B\times C}=(A^C)^B$ is a particular case
of \emph{adjunction}: the right adjoint of the cartesian product is
the exponential. We have
\begin{lem}\label{lem:adjoint1}
  For any labelled graphs $\gfG_1\to\gfA$ and
  $\gfA\leftarrow\gfG_2\to\gfB$ and $\gfG_3\to\gfB$ we have a natural
  bijection
  \[\Graph_{/\gfB}(\gfG_1\times_\gfA\gfG_2,\gfG_3)\cong\Graph_{/\gfA}(\gfG_1,\gfG_3^{\gfG_2}).\]
\end{lem}

The proof is essentially ``currying'': a morphism
$\phi\colon\gfG_1\times_\gfA\gfG_2\to\gfG_3$ gives rise to a family of
morphisms $\phi_x\colon\gfG_2\to\gfG_3$, indexed by $x\in\gfG_1$, by
the formula $\phi_x(y)=\phi(x,y)$; and conversely.
\begin{proof}[Proof with the painful details]
  Write the labellings $\alpha_i\colon\gfG_i\to\gfA$ for $i=1,2$ and
  $\beta_i\colon\gfG_i\to\gfB$ for $i=2,3$. We give maps in both
  directions, and then show that they are inverses of each other.

  $(\exists\rightarrow)$ Let
  $\lambda\colon \gfG_1 \times_\gfA \gfG_2 \to \gfG_3$ be a graph
  morphism.  Define the map $\rho\colon \gfG_1 \to \gfG_3^{\gfG_2}$ by
  $\rho(u_1)=(f_{u_1},\alpha_1(u_1))$ with
  $f_{u_1}(u_2)=\lambda(u_1,u_2)$ for all
  $u_2\in\alpha_2^{-1}\alpha_1(\overline{u_1})$. Note
  $f_{u_1}\in\Graph_{/\gfB}$ since it is a restriction of
  $\lambda\in\Graph_{/\gfB}$. Clearly $\rho(u_1)^\pm=\rho(u_1^\pm)$
  and $\rho$ preserves the $\gfA$-labelling, so
  $\rho\in\Graph_{/\gfA}$.

  $(\exists\leftarrow)$ Let $\rho\colon \gfG_1 \to \gfG_3^{\gfG_2}$ be
  a graph morphism. Define the map
  $\lambda\colon \gfG_1 \times_\gfA \gfG_2 \to \gfG_3$ by
  $\lambda(u_1,u_2)=f_{u_1}(u_2)$ where $\rho(u_1)=(f_{u_1},\alpha_1(u_1))$. Note,
  since $f_{u_1}\in\Graph_{/\gfB}$, that
  $\lambda(u_1,u_2)^\pm=\lambda(u_1^\pm,u_2^\pm)$ and that $\lambda$
  preserves the $\gfB$-labelling; so $\lambda\in\Graph_{/\gfB}$.

  $(\leftarrow\circ\rightarrow)$ We next show that the constructions
  are inverses of each other. Let
  $\widehat\rho \in \Graph_{/\gfA}(\gfG_1, \gfG_3^{\gfG_2})$ be
  constructed from
  $\lambda \in \Graph_{/\gfB}(\gfG_1 \times_\gfA \gfG_2, \gfG_3)$,
  which in turn is constructed from
  $\rho \in \Graph_{/\gfA}(\gfG_1, \gfG_3^{\gfG_2})$. We show
  $\widehat\rho = \rho$: for $u_1 \in \gfG_1$ we have
  $\widehat\rho(u_1)=(\widehat{f_{u_1}},\alpha_1(u_1))$ and
  $\rho(u_1)=(f_{u_1},\alpha_1(u_1))$, with
  $\widehat{f_{u_1}}(u_2) = \lambda(u_1,u_2) = f_{u_1}(u_2)$ directly
  from the defining formulas, so $\widehat{f_{u_1}}=f_{u_1}$.

  $(\rightarrow\circ\leftarrow)$ Conversely, let
  $\widehat\lambda \in \Graph_{/\gfA}(\gfG_1 \times_\gfA \gfG_2,
  \gfG_3)$ be constructed from
  $\rho \in \Graph_{/\gfB}(\gfG_1, \gfG_3^{\gfG_2})$, which in turn is
  constructed from
  $\lambda \in \Graph_{/\gfA}(\gfG_1 \times_\gfA \gfG_2,
  \gfG_3)$. Writing $\rho(u_1)=(f_{u_1},\alpha_1(u_1))$ we have
  \[\widehat\lambda(u_1,u_2) = f_{u_1}(u_2) = \lambda(u_1,u_2).\qedhere\]
\end{proof}

\begin{exple}
  Continuing on Example~\ref{ex:2sets}, with $\gfG_1=S_0\sqcup S_1$
  and $\gfG_2=T_0\sqcup T_1$ and $\gfG_3=U_0\sqcup U_1$, we get
  \begin{align*}
    \Hom(\gfG_1\times_\gfA\gfG_2,\gfG_3)&=\Hom(S_0\times T_0\sqcup S_1\times T_1,U_0\sqcup U_1)=U_0^{S_0\times T_0}\times U_1^{S_1\times T_1}\\
    &=\big(U_0^{T_0}\big)^{S_0} \times \big( U_1^{T_1} \big)^{S_1} = \Hom(S_0\sqcup S_1,U_0^{T_0}\sqcup U_1^{T_1})=\Hom(\gfG_1,\gfG_3^{\gfG_2}).
  \end{align*}
\end{exple}

\subsection{Path subdivisions}
We introduce a geometric operation, that of subdividing edges by
replacing them by unbounded paths:
\begin{defn}
  Let $\gfB$ be a graph. Its \emph{path subdivision} is the graph
  $\gfB^*$ with vertex set $\gfB=V(\gfB)\sqcup E(\gfB)$ and edge set
  $\{0,1\}\times E(\gfB)\times\{0,1\}$, with extremities
  \begin{xalignat*}{3}
    (0,e,0)^\pm &=e^\pm,& (0,e,1)^- &=e^-, & (0,e,1)^+ &=e,\\
    (1,e,1)^\pm &=e,& (1,e,0)^- &=e, & (1,e,0)^+ &=e^+.\\
  \end{xalignat*}
  In the unoriented case the vertex set is
  $\gfB/\{e'{=}e\;\forall e\in E(\gfB)\}$, and $(i,e,j)'=(j,e',i)$.
\end{defn}

In pictures, we are replacing every edge $e$, from $v$ to $w$, by a small graph:
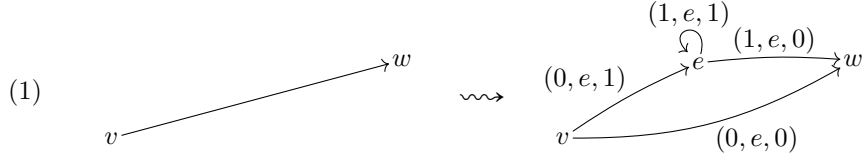
\begin{equation}\label{eq:subdiv}
  \begin{tikzpicture}[every node/.style={inner sep=1pt,minimum size=0pt},baseline=0]
    \begin{scope}[xshift=-3cm]
      \node (v0) at (195:2) {$v$};
      \node (w0) at (15:2) {$w$};
      \draw[->] (v0) edge (w0);
    \end{scope}
    \node at (0,0) {\Large $\rightsquigarrow$};
    \begin{scope}[xshift=3cm]
      \node (v) at (195:2) {$v$};
      \node (w) at (15:2) {$w$};
      \node (e) at (105:0.5) {$e$};
      \draw[->] (v) edge[bend right=15] node[below right] {$(0,e,0)$} (w);
      \draw[->] (v) edge[bend left=5] node[above left] {$(0,e,1)$} (e);
      \draw[->] (e) edge[out=75,in=135,loop] node[above] {$(1,e,1)$} (e);
      \draw[->] (e) edge[bend left=5] node[above] {$(1,e,0)$} (w);
    \end{scope}
  \end{tikzpicture}
\end{equation}

There is a natural map from paths in $\gfB^*$ to paths in $\gfB$;
namely, replace each path going through an $E(\gfB)$-vertex by the
straight edge going directly from $v$ to $w$
in~\eqref{eq:subdiv}. There is a natural map between the geometric
realization of $\gfB^*$ to that of $\gfB$, and a natural embedding of
$\gfB$ in $\gfB^*$ by mapping every edge $e$ to $(0,e,0)$. There
is no natural graph morphism $\gfB^*\to\gfB$.
The operation is evidently functorial, i.e.\ there is a natural way
to associate to a graph homomorphism (or labelled graph) $\alpha\colon\gfG\to\gfB$ a
graph homomorphism (or labelled graph) $\beta\colon\gfG^*\to\gfB^*$.

Consider a labelled graph $\alpha\colon\gfG\to\gfB^*$. We construct a
labelled graph $\beta\colon\gfG^\flat\to\gfB$ as follows: its vertex set is
$\alpha^{-1}(V(\gfB))$, with $u\in V(\gfG^\flat)$ labelled
$\alpha(u)$. For each path $(e_1,\dots,e_n)$ in $\gfG$, say from $u$
to $v$, such that $\alpha(u),\alpha(v)\in V(\gfB)$ and the labels
along $(e_1,\dots,e_n)$ are $(0,c,1),(1,c,1),\dots,(1,c,1),(1,c,0)$
for some $c\in E(\gfB)$, there is an edge from $u$ to $v$ labelled
$c$. We call paths $(e_1,\dots,e_n)$ as above \emph{coherent paths}.
The terminology ``$\flat$'' is justified by deFLATion of paths into
edges.

It may sometimes be necessary to add some extra, ``sink'' vertices to
$\gfG^*$: anticipating~\S\ref{ss:simulations}, the new vertices in
$\gfG^*$ represent intermediate steps in a calculation; and we want to
allow the possibility of calculations that start but do not terminate
--- no results are produced, but computational resources are
consumed. For a labelled graph $\alpha\colon\gfG\to\gfB$, define the
labelled graph $\beta\colon\gfG^\sharp\to\gfB^*$ as follows: its vertex  is
\[V(\gfG^\sharp)=V(\gfG^*)\sqcup \{\pm1\}\times E(\gfB).
\]
Its edge set is the union of $E(\gfG^*)$ and, for each edge $e$
labelled $c=\alpha(e)$ of $\gfG$, an edge labelled $(0,c,1)$ from
$e^-\in V(\gfG)$ to $(-,c)$, an edge labelled $(1,c,0)$ from
$(+,c)$ to $e^+$, and edges labelled $(1,c,1)$ from
$(+,c)$ to $(+,c)$, $(-,c)$ and $e$ and from
$e$ and $(-,c)$ to $(-,c)$. In pictures:
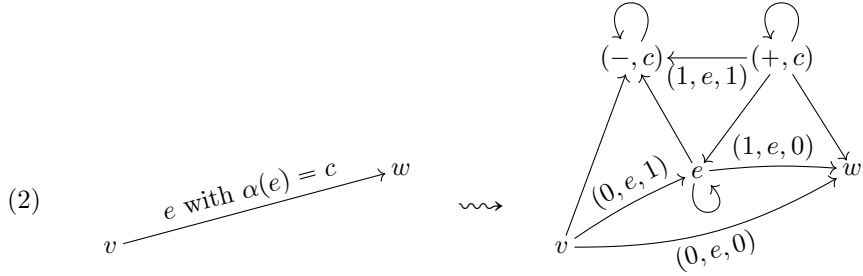
\begin{equation}
  \begin{tikzpicture}[every node/.style={inner sep=2pt,minimum size=0pt},baseline=0]
    \begin{scope}[xshift=-3cm]
      \node (v0) at (195:2) {$v$};
      \node (w0) at (15:2) {$w$};
      \draw[->] (v0) edge node[sloped,above] {$e$ with $\alpha(e)=c$} (w0);
    \end{scope}
    \node at (0,0) {\Large $\rightsquigarrow$};
    \begin{scope}[xshift=3cm]
      \node (v) at (195:2) {$v$};
      \node[inner sep=1pt] (-) at (-1,2) {$(-,c)$};
      \node[inner sep=1pt] (+) at (1,2) {$(+,c)$};
      \draw[->] (v) edge (-);
      \draw[->] (+) edge[out=60,in=120,loop] () edge node[below] {$(1,e,1)$} (-) edge (e) edge (w);
      \draw[->] (-) edge[out=60,in=120,loop] ();
      \draw[->] (e) edge (-);
      \node (w) at (15:2) {$w$};
      \node (e) at (105:0.5) {$e$};
      \draw[->] (v) edge[bend right=15] node[below,sloped] {$(0,e,0)$} (w);
      \draw[->] (v) edge[bend left=5] node[above,sloped,pos=0.6] {$(0,e,1)$} (e);
      \draw[->] (e) edge[out=255,in=315,loop] ();
      \draw[->] (e) edge[bend left=5] node[above] {$(1,e,0)$} (w);
    \end{scope}
  \end{tikzpicture}
\end{equation}

The operations
$\gfG\mapsto\gfG^\flat$ and $\gfG\mapsto\gfG^\sharp$ close to being adjoint to each other:

\begin{lem}\label{lem:notadjoint}
  Consider $\gfG\in\Graph_{/\gfB^*}$ and $\gfH\in\Graph_\gfB$. 
  Then there is a map
  \[ \Hom(\gfG,\gfH^\sharp) \rightarrow \Hom(\gfG^\flat,\gfH). \]
  If furthermore $\gfG^\flat$ is weakly \'etale, then
  there is also a map
  \[ \Hom(\gfG,\gfH^\sharp) \leftarrow \Hom(\gfG^\flat,\gfH), \]
  so $\Hom(\gfG^\flat,\gfH)$ is empty if and only if
  $\Hom(\gfG,\gfH^\sharp)$ is empty.
\end{lem}
\begin{proof}
  Consider first $f\colon\gfG\to\gfH^\sharp$, and define
  $f^\flat\colon\gfG^\flat\to\gfH$ as follows. Vertices of
  $\gfG^\flat$ are in particular vertices of $\gfG$, so we define
  $f^\flat$ on $V(\gfG^\flat)$ as the restriction of $f$. For an edge
  $(e_1,\dots,e_n)$ of $\gfG^\flat$, the path $(f(e_1),\dots,f(e_n))$
  is of the form
  $((\mathbf1_{i\neq1},e,\mathbf1_{i\neq n}))_{1\le i\le n}$, and we
  set $f^\flat(e_1,\dots,e_n)=e$. Note that in fact $f^\flat$ only
  depends on the restriction of $f$'s image to $\gfH^*$, since there
  is no coherent path in $\gfH^\flat$ going through the vertices
  $\{\pm1\}\times E(\gfB)$ and starting and ending at vertices of
  $\gfH$.

  Consider next $g\colon\gfG^\flat\to\gfH$. As a preprocessing step,
  we may assume without loss of generality that $g$ has the property
  that for all $p, q \in E(\gfG^\flat)$ with $p^- = q^-$ and
  $p^+ = q^+$ and $\alpha(p) = \alpha(q)$ we have $g(p) = g(q)$:
  simply choose one preferred path $p_{u,c,v}$ between any two
  $u, v \in V(\gfG^\flat)$ and for each $c \in E(\gfB)$ with
  $\alpha(u) = c^-$, $\alpha(v) = c^+$, and redefine
  $g(q) = g(p_{q^-,\alpha(q),q^+})$ for all $q \in E(\gfG^\flat)$.
  
  Now define $g^\sharp\colon\gfG\to\gfH^\sharp$ as follows. On
  vertices of $\gfG$ with label in $V(\gfB)$, called ``vertex
  vertices'', let $g^\sharp$ coincide with $g$. Consider now a vertex
  $p$ of $\gfG$ with label in $V(\gfB^*)\setminus V(\gfB)$, an ``edge
  vertex''. If $p$ lies on a coherent path $(e_1,\dots,e_n)$ in
  $\gfG$, set $g^\sharp(p)=g(e_1,\dots,e_n)$ for an arbitrary such
  path. Since $\gfG^\flat$ is weakly \'etale, all coherent paths have
  the same end points and $\alpha$-label $\alpha(p)$, and by our
  preprocessing of $g$ all such paths have the same image, so the
  definition of $g^\sharp(p)$ is unambiguous.
  
  Consider next an edge vertex $p$ that does not lie on a
  coherent path. If it admits a path labelled
  $(1,c,1),\dots,(1,c,1),(1,c,0)$ towards a vertex vertex, map it by
  $g^\sharp$ to $(+,c)$; if on the other hand there is path from a
  vertex vertex to $p$ labelled $(0,c,1),(1,c,1),\dots,(1,c,1)$ then
  map it by $g^\sharp$ to $(-,c)$; if neither path exists, map it (by
  will) to $(+,\alpha(p))$.

  There is a unique way of extending $g^\sharp$ to edges: edges in
  $\gfG$ between vertex vertices, namely with label $(0,c,0)$, are
  edges $e\in E(\gfG^\flat)$ so one may set
  $g^\sharp(e)=(0,g(e),0)$. Edges along coherent paths are mapped by
  $g^\sharp$ to the unique possible edges $(0,e,1),(1,e,1),\dots,(1,e,1),(1,e,0)$ in
  $\gfH^*\subset\gfH^\sharp$. Finally, if an edge of $\gfG$ starts or
  ends at a vertex that we already decided to map to
  $w \in \{+,-\}\times E(\gfB) \subseteq V(\gfH^\sharp)$, then map this
  edge by $g^\sharp$ to a loop at $w$, except if it respectively ends
  or starts at a vertex vertex or an edge vertex on a coherent path.
\end{proof}
Note that if $\gfG^\flat$ is actually \'etale, meaning that in $\gfG$
there is a unique coherent path with a given label between any two
vertices, then the preprocessing step is not needed.

Note also that, in the proof above, the compositions of $f\mapsto f^\flat$
and $g\mapsto g^\sharp$ are not quite the identity: $(g^\sharp)^\flat$ is the
identity on vertices, but when $\gfG^\flat$ is not \'etale this
identifies images of edges corresponding to coherent paths with the same
endpoints and labels. On the other hand, $(f^\flat)^\sharp$ projects
components without initial or final vertex to ``sinks''. On the other hand,
restricting to an appropriate subclass of graphs in which all edge vertices
lie on coherent paths, one obtains an adjunction between $\gfG\mapsto\gfG^\flat$
and $\gfH\mapsto\gfH^\sharp$.

\subsection{Simulations}\label{ss:simulations}
We now introduce a more general notion of morphism, in the same way
that bimodules generalise algebra morphisms:
\begin{defn}
  Let $\gfA,\gfB$ be two graphs, which we think of as labellings. A
  \emph{simulator} is a graph $\gfS$ equipped with two labellings
  $\alpha\colon\gfS\to\gfA$ and $\beta\colon\gfS\to\gfB^*$. We call
  $\gfS$, more precisely, an \emph{$(\gfA,\gfB)$-simulator}.
\end{defn}

Let $\gfG$ be an $\gfA$-labelled graph. Its \emph{image} under the
simulation $\gfS$ is the graph obtained from $\gfG\times_\gfA\gfS$ by
replacing each subgraph as on the right of~\eqref{eq:subdiv} by the
corresponding edge; namely, using the labelling $\beta$ the graph
$\gfG\times_\gfA\gfS$ is $\gfB^*$-labelled, and we define
$\gfG\rtimes\gfS$ as the $\gfB$-labelled graph
$(\gfG\times_\gfA\gfS)^\flat$. We say that $\gfG$ \emph{simulates}
$\gfH$ if $\gfH\cong\gfG\rtimes\gfS$ for some \emph{finite} simulator
$\gfS$.  We say $\gfG$ \emph{simulates} $\gfH$ \emph{up to
  simplification} if $\gfG$ simulates a graph $\gfK$ whose simplification is $\gfH$.

\newcommand\NESW{\tikz[baseline=-1ex]{\draw(0,-1ex) -- (0,1ex) (-1ex,0) -- (1ex,0) -- (0,0);}}
\newcommand\NEW{\tikz[baseline=0]{\draw(0,0) -- (0,1ex) (-1ex,0) -- (1ex,0) -- (0,0);}}
\newcommand\NES{\tikz[baseline=-1ex]{\draw(0,-1ex) -- (0,1ex) (0,0) -- (1ex,0) -- (0,0);}}
\newcommand\NE{\tikz[baseline=0]{\draw(0,0) -- (0,1ex) (0,0) -- (1ex,0) -- (0,0);}}

\begin{exple}\label{ex:quadrant}
  The \emph{plane} is the Cayley graph of
  $\Z^2=\langle\rightarrow,\leftarrow,\uparrow,\downarrow\rangle$, and
  will still be written $\Z^2$. The \emph{quadrant} is the full
  subgraph of $\Z^2$ with vertex set $\N^2$; its vertices have degree
  $4$, $3$ or $2$ depending on how many of their coordinates are
  $0$.

  The quadrant is naturally vertex-labelled, according to the available
  directions: $\NESW, \NES, \NEW, \NE$; thus the origin is labelled
  $\NE$ in the quadrant, and the line $\{x=0\}$ is labelled
  $\NES$. (For exhaustivity, the plane is also vertex-labelled, with
  $\NESW$ everywhere). The quadrant's labelling is bit finer than the
  sunny-side-up labelling from Example~\ref{ex:ssu} in that it
  remembers the axes too.
\end{exple}

\begin{exple}\label{ex:simquadrant}
    We claim that the quadrant simulates the plane. For this, it
    suffices to exhibit a finite simulator:
  
  
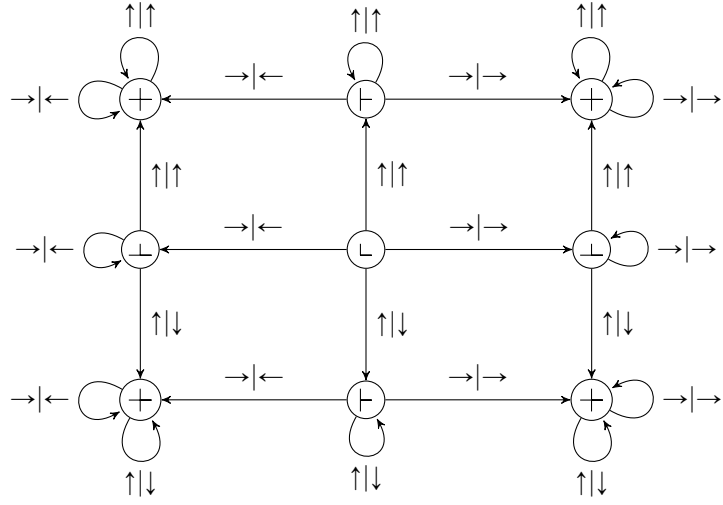
\centerline{\begin{tikzpicture}[every state/.style={inner sep=1pt,minimum size=5mm},>=stealth']
      \node[state] (00) at (0,0) {\NE};
      \foreach\i/\j/\x/\y/\sym in {+/0/3/0/\NEW,-/0/-3/0/\NEW,0/+/0/2/\NES,0/-/0/-2/\NES,+/+/3/2/\NESW,+/-/3/-2/\NESW,-/+/-3/2/\NESW,-/-/-3/-2/\NESW} {\node[state] (\i\j) at (\x,\y) {\sym}; }
      \foreach\j in {-,0,+} {
        \draw[->] (0\j) -- node[above] {${\rightarrow}|{\rightarrow}$} (+\j);
        \draw[->] (0\j) -- node[above] {${\rightarrow}|{\leftarrow}$} (-\j);
        \draw[->] (+\j) edge[in=30,out=-30,loop] node[right] {${\rightarrow}|{\rightarrow}$} ();
        \draw[->] (-\j) edge[in=210,out=150,loop] node[left] {${\rightarrow}|{\leftarrow}$} ();
      }
      \foreach\i in {-,0,+} {
        \draw[->] (\i0) -- node[right] {${\uparrow}|{\uparrow}$} (\i+);
        \draw[->] (\i0) -- node[right] {${\uparrow}|{\downarrow}$} (\i-);
        \draw[->] (\i+) edge[in=120,out=60,loop] node[above] {${\uparrow}|{\uparrow}$} ();
        \draw[->] (\i-) edge[in=300,out=240,loop] node[below] {${\uparrow}|{\downarrow}$} ();
      }
    \end{tikzpicture}}

  In this graph, edges are labelled as `$a|b$' to represent its labels
  under the left and right map, respectively, to
  $(1,\{\rightarrow,\leftarrow,\uparrow,\downarrow\})$. We have only
  drawn edges corresponding to directions $\uparrow, \rightarrow$ for
  the map to $\gfA$; the other edges are obtained by flipping the
  arrows. Since all edges under the right map are of the form
  $(0,e,0)$, we have simply written $e$ instead; and since there are
  no vertex labels in $\Z^2$ we have only indicated the vertex labels
  from $\N^2$ next to the nodes.

  In particular, the root $(0,0)$ in the quadrant (mapped to the
  central node of the simulator) corresponds to the root $(0,0)$ in
  the plane; the axes in the quadrant are each covered by two
  half-axes; and the remainder of the quadrant is covered by the four
  quadrants of the plane. Thus the plane is simulated, within the
  quadrant, by folding it like a handkerchief.
\end{exple}

\begin{exple}
  The previous example does not use paths of length $>1$ under the
  right map. It is, in a sense, possible to simulate the plane within the quadrant
  while never putting two simulated vertices at the same place, at the
  cost of simulating edges in the plane by paths of length $2$ in the
  quadrant.

  For simplicity, we will simulate $\Z$ in $\N$; naturally
  simulators for $\Z^2$ in $\N^2$ may be obtained by taking
  products. The simulation from Example~\ref{ex:simquadrant}, when
  restricted to $\Z$, was (now depicted folded, with \textsf{v.} and
  \textsf{a.} for \textsf{vitellus} and \textsf{albumen})
  
  
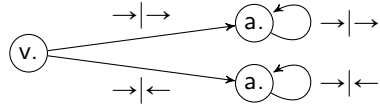
\centerline{\begin{tikzpicture}[every state/.style={inner sep=1pt,minimum size=5mm},>=stealth']
      \node[state] (0) at (0,0) {\textsf{v.}};
      \node[state] (+) at (3,0.4) {\textsf{a.}};
      \node[state] (-) at (3,-0.4) {\textsf{a.}};
      \draw[->] (0) -- node[above] {${\rightarrow}|{\rightarrow}$} (+);
      \draw[->] (0) -- node[below] {${\rightarrow}|{\leftarrow}$} (-);
      \draw[->] (+) edge[in=30,out=-30,loop] node[right] {${\rightarrow}|{\rightarrow}$} ();
      \draw[->] (-) edge[in=30,out=-30,loop] node[right] {${\rightarrow}|{\leftarrow}$} ();
    \end{tikzpicture}}

  The idea of the following simulator is that $n\in\Z$ is simulated by
  $2n\in\N$ and $-n\in\Z$ by $2n-1\in\N$, for all $n\ge0$. This is
  impossible to achieve exactly with our definition, so we produce
  $\Z\sqcup\N\sqcup{-\N}$ instead.  The vertices of the simulator are
  now of two kinds: split nodes indicate an actual vertex of the
  simulated graph (their label in the simulated graph being
  furthermore indicated by the lower symbol --- in this case,
  $-,\cdot,+$ giving the sign\footnote{The Cayley graph of $\Z$ does
    not have these markings, but having them makes the automaton
    easier to read, and identifying $-$ with $+$ yields a valid
    simulation of seeded-$\Z$.} in $\Z$), and unsplit nodes indicate
  intermediate nodes from $\gfB^*\setminus\gfB$:
  
  
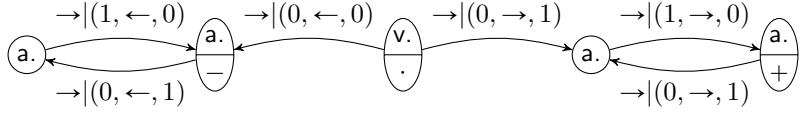
\centerline{\begin{tikzpicture}[every state/.style={inner sep=1pt,minimum size=5mm},>=stealth']
      \ellipsesplitnode{0}{(0,0)}{1.2mm}{1mm}{$\textsf{v.}$}{$\cdot$}
      \node[state] (o1) at (2.5,0) {\textsf{a.}};
      \ellipsesplitnode{e1}{(5,0)}{1.2mm}{1mm}{$\textsf{a.}$}{$+$}
      \ellipsesplitnode{o2}{(-2.5,0)}{1.2mm}{1mm}{$\textsf{a.}$}{$-$}
      \node[state] (e2) at (-5,0) {\textsf{a.}};
      \draw[->] (0) edge[bend left=15] node[above] {${\rightarrow}|(0,\rightarrow,1)$} (o1);
      \draw[->] (0) edge[bend right=15] node[above] {${\rightarrow}|(0,\leftarrow,0)$} (o2);
      \draw[->] (o1) edge[bend left=15] node[above] {${\rightarrow}|(1,\rightarrow,0)$} (e1);
      \draw[->] (e1) edge[bend left=15] node[below] {${\rightarrow}|(0,\rightarrow,1)$} (o1);
      \draw[->] (o2) edge[bend left=15] node[below] {${\rightarrow}|(0,\leftarrow,1)$} (e2);
      \draw[->] (e2) edge[bend left=15] node[above] {${\rightarrow}|(1,\leftarrow,0)$} (o2);
    \end{tikzpicture}}

  Indeed following $(\rightarrow)^{2n}$ in $\N$, starting from the
  origin, amounts to following
  $((0,\rightarrow,1)(1,\rightarrow,0))^n$ in the pullback, which
  amounts to following $(\rightarrow)^n$ in $\Z$, and following
  $(\rightarrow)^{2n+1}$ in $\N$, starting from the origin, amounts to
  following $(0,\leftarrow,0)((0,\leftarrow,1)(1,\leftarrow,0))^n$ in
  the pullback, which amounts to following $(\leftarrow)^{n+1}$ in
  $\Z$.
  
  The fibre product $\N\rtimes\gfS$ also contains vertices
  $(2n,\tikz[baseline=-1mm]{\ellipsesplitnode{0}{(0,0)}{0mm}{0mm}{\tiny\textsf{v.}}{\tiny$-$}})$
  and
  $(2n+1,\tikz[baseline=-1mm]{\ellipsesplitnode{0}{(0,0)}{0mm}{0mm}{\tiny\textsf{v.}}{\tiny$+$}})$ producing a copy of $-\N$
  and a copy of $\N$.
  If we add vertex labels for odd and even positions in $\N$, then it is possible to eliminate these
  additional copies of $\N$.
  Such a labelling can be introduced by an SFT marking (which is always a possibility in our applications,
  as our simulators run on graphs decorated by SFT configurations). See also Section~\ref{ss:QuadToPlane},
  where we give another interpretation.
\end{exple}

We adapt the definition of exponential objects for simulators. Let
$\gfH$ be a $\gfB$-labelled graph, labelled by
$\eta\colon\gfH\to\gfB$. We define $\gfH^\gfS$ as the $\gfA$-labelled
graph $(\gfH^\sharp)^\gfS$, with the morphisms in the exponential
belonging to $\Graph_{/\gfB^*}$.
\begin{lem}\label{lem:almostadjoint}
  Let $\gfG$ be an $\gfA$-labelled graph, let $\gfH$ be a
  $\gfB$-labelled graph, and let $\gfS$ be an
  $(\gfA,\gfB)$-simulator. Then 
  \[\Graph_{/\gfB}(\gfG\rtimes\gfS,\gfH)\neq\emptyset\text{ whenever }\Graph_{/\gfA}(\gfG,\gfH^\gfS)\neq\emptyset. \]
  If additionally $\gfG\rtimes\gfS$ is weakly \'etale, then 
  \[\Graph_{/\gfB}(\gfG\rtimes\gfS,\gfH)\neq\emptyset\text{ if and only if }\Graph_{/\gfA}(\gfG,\gfH^\gfS)\neq\emptyset. \]
\end{lem}
\begin{proof}
  This follows immediately from Lemma~\ref{lem:adjoint1} and
  Lemma~\ref{lem:notadjoint}. For the second claim,
  $\Graph_{/\gfB}(\gfG\rtimes\gfS,\gfH)=\Graph_{/\gfB}((\gfG\times_\gfA\gfS)^\flat,\gfH)$
  is empty if and only if
  $\Graph_{/\gfB^*}(\gfG\times_\gfA\gfS,\gfH^\sharp)\cong\Graph_{/\gfA}(\gfG,\gfH^\gfS)$
  is empty.
\end{proof}

\noindent As could be expected, simulation is transitive:
\begin{lem}\label{lem:simultrans}
  If $\gfG$ simulates a graph $\gfH$, and $\gfH$ simulates a graph $\mathscr K$, then $\gfG$ simulates $\mathscr K$.
\end{lem}

\begin{proof}
  The idea of the proof is straightforward: run the simulation of
  $\mathscr K$ as a subroutine within the simulation of $\mathscr
  H$. Here are the details.
  
  Let $\gfS$ be a $(\gfA,\gfB)$-simulator expressing the simulation
  $\gfH\cong\gfG\rtimes\gfS$, and let $\mathscr T$ be a
  $(\gfB,\gfC)$-simulator expressing the simulation
  $\gfH\rtimes\mathscr T\cong\mathscr K$. Then $\mathscr T^\sharp$ is
  a $\gfB^*$- and $\gfC^{**}$-labelled graph; its labels in
  $\gfC^{**}$ take the form $(i,(i',c,j'),j)$. Only for this proof,
  let $\natural$ be the operation that replaces every such label by
  $(\max(i,i'),c,\max(j,j'))$, changing the $\gfC^{**}$-labelling into
  a $\gfC^*$-labelling. Set
  $\mathscr U=(\gfS\times_{\gfB^*}\mathscr T^\sharp)^\natural$, and
  note that it is an $(\gfA,\gfC)$-simulator. We have, noting that the
  operations $()^\sharp$ and $()^\natural$ commute with products,
  \begin{align*}
    \gfG \rtimes\mathscr U
    &=(\gfG\times_\gfA\mathscr U)^\flat
      =(\gfG\times_\gfA(\gfS\times_{\gfB^*}\mathscr T^\sharp)^\natural)^\flat\\
    &\cong(\gfG\times_\gfA\gfS\times_{\gfB^*}\mathscr T^\sharp)^{\natural\flat}
      \cong((\gfG\times_\gfA\gfS)^{\flat\sharp}\times_{\gfB^*}\mathscr T^{\sharp})^{\natural\flat}\\
    &=(\gfH^\sharp\times_{\gfB^*}\mathscr T^\sharp)^{\natural\flat}
      \cong(\gfH\times_\gfB\mathscr T)^{\sharp\natural\flat}\\
    &\cong(\gfH\times_\gfB\mathscr T)^\flat
      =\gfH\rtimes\mathscr T
      \cong\mathscr K.\qedhere
  \end{align*}
\end{proof}

Note that the $\flat$-operation removes the sinks introduced by
$T \mapsto T^\sharp$, so without changing the previous proof, one can
shave off a small number of vertices by using
$\gfS \times \mathscr T^*$ as the simulator.

Robinson's tileset may be used to prove the undecidability of the
tiling problem for $\Z^2$ by simulating seeded $\Z^2$-tilings by
unseeded ones. The following example exhibits the simulator part of
the argument.  We shall return to it in Example~\ref{ex:Robinson} and
explain how the undecidability of the tiling problem of $\Z^2$ is
obtained from this.

\begin{exple}
\label{ex:CompressingRectangles}
Let $S = \{(1,0), (0,1), (-1,0), (0,-1)\}$ be the standard generating
set for $\Z^2$ and let $\gfB = 1 \sqcup S$ be the graph used for the
Cayley graph labelling.  Consider the induced subgraph $R$ with vertex
set $\{0,1,\dots,n-1\}^2 \subset \Z^2$. Consider a labelled graph
$\alpha\colon R \to \gfA$, for the product graph
\[\gfA = (1\sqcup S) \times (\{\textrm{good},\textrm{bad}\} \sqcup \{\textrm{good},\textrm{bad}\}^2).\]
Suppose that the label in $(1\sqcup S)$ is just the induced Cayley graph
labelling from $\Z^2$, and suppose that the subset $S' \subset V(S)$
of good vertices is a set-theoretic rectangle:
\[ \exists A, B \subset \{0,1,\dots,n-1\}: \alpha^{-1}(\{(1,\textrm{good})\}) = A \times B. \]

Then consider the following $(\gfA,\gfB)$-simulator on $S$. The vertex
labelled ``good'' has $\gfA$-label $(1,\mathrm{good})$ and
$\gfB^*$-label $1 \in V(\gfB)$, while the split node labelled
``bad/$s$'' has $\gfA$-label $(1,\mathrm{bad})$ and $\gfB^*$-label
$s \in S = E(\gfB)$; $s$ ranges over $S$ so the figure is short for a
$5$-vertex graph.

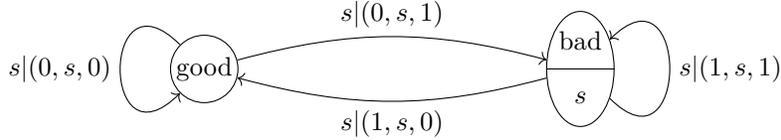
\begin{center}
  \begin{tikzpicture}[every state/.style={inner sep=1pt,minimum size=4mm}]
    \node[state] (g) at (0,0) {good};
    \ellipsesplitnode{b}{(5,0,0)}{4mm}{3mm}{bad}{$s$}
    \draw[->] (g) edge[bend left=15] node[above] {$s|(0,s,1)$} (b);
    \draw[->] (b) edge[bend left=15] node[below] {$s|(1,s,0)$} (g);
    \draw[->] (g) edge[in=225,out=135,distance=15mm] node[left] {$s|(0,s,0)$} (g);
    \draw[->] (b) edge[in=45,out=-45,distance=15mm] node[right] {$s|(1,s,1)$} (b);
  \end{tikzpicture}
\end{center}

In words, this simulator compresses runs over bad nodes into an edge,
and thus compresses the set-theoretic rectangle $A \times B$ into an
actual rectangle. The $\gfB$-labelling is the one induced from
$\{0,1,\dots,|A|-1\} \times \{0,1,\dots,|B|-1\} \subset \Z^2$.  In
terms of the graph-walking automata introduced in the following
section, the automaton ``skips over'' the bad nodes.
\end{exple}

\subsection{Graph-walking automata}
Simulators can be seen as a compact way of expressing two things at
once: vertex duplication (when we want to simulate multiple nodes in
one actual nodes), and coalescence of paths into single edges based on
the labels of connecting paths.

In this section, we re-express simulators explicitly in terms of vertex blow-ups,
graph-walking automata, and formal languages. We show that these points of view are equivalent to the
simulators introduced in the previous section; while simulators tend
to be convenient in proofs, graph-walking automata are preferable
for describing concrete simulations.

\begin{defn}[Vertex blow-up]
  Let $\gfA$ be a graph, and let $k\colon v(\gfA) \to \N$ be
  a function. The \emph{vertex blow-up of $\gfA$} by $k$ is the
  graph with nodes
  \[ V(\gfA^k) = \{(u, i)\in V(\gfA)\times\N : i<k(u)\}\]
  and edges
  \[ E(\gfA^k) = \{(i,e,j) \in \N\times E(\gfA) \times \N: i < k(e^-), j < k(e^+)\}\]
  with $(i,e,j)^- = (e^-,i)$ and $(i,e,j)^+ = (e^+,j)$. 
\end{defn}

\noindent We extend this to labelled graphs by blowing-up the graph
used for the labels:
\begin{defn}
  Let $\alpha\colon\gfG \to \gfA$ be an $\gfA$-labelled graph, and let $k\colon V(\gfA) \to \N$
  be a function.  The \emph{vertex blow-up of $\gfG$} by $k$ is defined
  as the graph with nodes
  \[ V(\gfG^k) = \{(u, i)\in V(\gfG)\times\N: i<k(\alpha(u))\} \]
  and edges
  \[ E(\gfG^k) = \{(i,e,j) \in \N\times E(\gfG) \times \N: i<k(e^-), j<k(e^+)\}
  \]
  with $(i,e,j)^- = (e^-,j)$ and $(i,e,j)^+ = (e^+,j)$ and labels in
  $\gfA^k$ defined by $\alpha((u,i)) = (\alpha(u),i)$.
\end{defn}

It is sometimes convenient to blow up vertices and give them more
memorable names than numbers, and use a function
$k\colon V(\gfA) \to\{\text{finite subsets of } U\}$ instead, for some
universe $U$. To interpret this, fix bijections between $k(v)$ and
$\{0,\dots,|k(v)|-1\}$ for all $v \in V(\gfA)$ and conjugate all
arguments through this bijection.

\begin{defn}[Graph-walking automata]
  Let $\alpha\colon \gfG \to \gfA$ be an $\gfA$-labelled graph. A \emph{graph-walking
    automaton} (short: \emph{GWA}) on $\gfG$ is a tuple
  $M = (Q,S,F,\Delta)$ with $Q$ a finite set of \emph{states},
  $S \subset Q$ a set of \emph{initial states}, $F \subset Q$ a set of
  \emph{final states}, and $\Delta \subset Q \times E(\gfA) \times Q$
  a \emph{transition relation}. We require $I \cap F = \emptyset$.
  For $u, v \in V(\gfG)$, an \emph{$M$-run from $u$ to $v$} is a
  sequence $(q_0, e_0, q_1,\cdots, e_{k-1}, q_k)$ with for all $i<k$
  \[ q_0 \in I, q_k \in F, q_i \in Q, e_i \in E(\gfG), e_0^- = u,
    e_{k-1}^+ = v, (q_i, \alpha(e_i), q_{i+1}) \in \Delta.
  \]
  For $u \in V(\gfG)$, the \emph{$M$-successors} of $u$ are the nodes
  $v \in V(\gfG)$ such that there exists an $M$-run from $u$ to $v$.
\end{defn}

The basic idea is that runs will be replaced by edges, just as paths
are replaced by edges in the definition of simulators. Note that
$I \cap F=\emptyset$ implies that all runs are of length at least one.

\begin{defn}[GWA simulator]
  For $\gfG$ an $\gfA$-labelled graph and $\gfB$ is a finite
  graph, a \emph{$\gfB$-labelled GWA simulator} $S$
  consists of the following data:
  \begin{enumerate}
  \item a function $\lambda\colon V(\gfA) \to V(\gfB) \cup \{\bot\}$;
  \item for each $e \in E(\gfB)$, a GWA $M_e$.
  \end{enumerate}
  The GWA $M_e$ is required to have the property that if
  $u \in V(\gfG)$ and $\lambda(\alpha(u)) \neq e^-$, then $u$ has no
  $M_e$-successors, and all $M_e$-successors $v$ of $u$ satisfy
  $\lambda(\alpha(v)) = e^+$.  The GWA simulator $S$ \emph{simulates}
  the $\gfB$-labelled graph $\gfH$ with vertices
  \[ V(\gfH) = \{v \in V(\gfG) : \lambda(v) \neq \bot\} \]
  and edges
  \[ E(\gfH) = \{(u,e,v)\in V(\gfG)\times E(\gfB)\times V(\gfG): \text{ there is an $M_e$-run from $u$ to $v$}\} \]
  with $(u,e,v)^- = u, (u,e,v)^+ = v$ and $\gfB$-labelling $\beta(v) = \lambda(\alpha(v))$, $\beta((u,e,v)) = e$.
\end{defn}
Observe that the extra requirement on the GWA $M_e$ is for convenience
only: we can easily modify any GWA so that it satisfies this property.

\begin{defn}
  Let $\gfG$ be an $\gfA$-labelled graph. We say $\gfG$
  \emph{GWA-simulates} a $\gfB$-labelled graph $\gfH$ if there exists a
  $\gfB$-labelled GWA simulator that simulates $\gfH$.
\end{defn}

\begin{defn}[Regular subdivision]
  Let $\gfG$ be an $\gfA$-labelled graph and let $\gfB$ be a finite
  graph. Suppose that we are given
  \begin{enumerate}
  \item a function $\lambda\colon V(\gfA) \to V(\gfB) \cup \{\bot\}$, and
  \item for each $e \in E(\gfB)$, a regular language
    $L_e \subset E(\gfA)^+$ of nontrivial words.
  \end{enumerate}
  Suppose further that every word $w=w_0\dots w_{\ell-1}\in L_e$ is a
  path in $\gfA$ satisfying $\lambda(w_0^-) = e^-$ and
  $\lambda(w_{\ell-1}^+) = e^+$.  Then the graph $\gfH$ with vertices
  \[ V(\gfH) = \{v \in V(\gfG): \lambda(v) \neq \bot\} \]
  and edges
  \[ E(\gfH) = \{(u,e,v)\in V(\gfG)\times E(\gfB)\times V(\gfG): \text{ there is an $L_e$-labelled path from $u$ to $v$} \} \]
  with $(u,e,v)^- = u, (u,e,v)^+ = v$ and $\gfB$-labelling
  $\beta(v) = \lambda(\alpha(v))$, $\beta((u,e,v)) = e$, is called a
  \emph{regular $\gfB$-subdivision} of $\gfG$.
\end{defn}

\noindent With these definitions in place, let us prove the
equivalence of these notions:
\begin{lem}\label{lem:gwa}
  Let $\gfG$ be an $\gfA$-labelled graph. Then the following are equivalent:
  \begin{enumerate}
  \item $\gfG$ simulates $\gfH$ up to full simplification,
  \item some vertex blow-up of $\gfG$ GWA-simulates $\gfH$,
  \item $\gfH$ is a regular $\gfB$-subdivision of a vertex blow-up of $\gfG$.
  \end{enumerate}
\end{lem}

Simplification is only a technicality: it is natural to require it in
an automata-theoretic setting, by including only one edge if a run
exists, while in the categorical setting of graphs we find it more
natural to include an edge for each run. Simplification translates
between these conventions. Full simplification (including self-loops)
is mainly for notational convenience; self-loops have little effect on
the tiling problems we are concerned with, and their presence
simplifies some proofs.

\begin{proof}
  $(1)\Rightarrow(2)$. Let $\gfG$ simulate $\gfH$ up to simplification by some simulator
  $\gfS$; we denote by $\alpha\colon \gfG \to \gfA$ and
  $\alpha_\gfS\colon\gfS\to \gfA$ and
  $\beta_\gfS\colon\gfS \to \gfB^*$ the respective labellings. We start
  by blowing up the vertices of $\gfG$ using the function
  $k\colon V(\gfA) \to\{\text{ subsets of }V(\gfS)\}$ defined by
  \[ k(u) = \{v \in V(\gfS) : \alpha_\gfS(v) = u \}.
  \]
  Then the graph $\gfG^k$ has vertices $(w,v)$ with $w \in V(\gfG)$
  and $v \in k(\alpha(w))$. For the function $\lambda$ we pick
  $\lambda(v) = \beta_\gfS(v)$ whenever $\beta_\gfS(v) \in V(\gfB)$,
  and $\lambda(v) = \bot$ whenever
  $v \in V(\gfB^*) \setminus V(\gfB)$.

  The vertices of both $\gfG\rtimes\gfS$ and $\gfH$ are by
  definition pairs $(w, v)$ where $\alpha(w) = \alpha(v)$ and
  $\beta(v) \in \gfB$. No matter how we pick the automata $M_e$, the
  vertices of the GWA-simulated graph will be
  $\{(w, v): \lambda(v) \neq \bot\}$. These sets are equal, so also in
  the simulated graph the vertex set will literally be
  $V(\gfG\rtimes\gfS)$.

  Now, for $e \in E(\gfB)$, we pick the automata $M_e$. By the
  definition of simulation up to simplification, in $\gfH$ there is at most one edge with
  label $e$ between $w \in V(\gfH)$ and $w' \in V(\gfH)$, and there is
  such an edge precisely when $w \neq w'$ and we are in one of the following cases:
  \begin{enumerate}
  \item there is an edge $(e_\gfG, e_\gfS) \in \gfG\rtimes\gfS$ with
    $e_\gfG^- = w$, $e_\gfG^+ = w'$ and
    $\beta_\gfS(e_\gfS) = (0,e,0)$, or
  \item there is a path of length at least two in $\gfG\rtimes\gfS$
    from $w$ to $w'$, such that the following three conditions are satisfied:
    \begin{enumerate}
    \item the first edge $(e_\gfG, e_\gfS)$ on the path satisfies
      $\beta_\gfS(e_\gfS) = (0,e,1)$,
    \item the last edge $(e_\gfG, e_\gfS)$ on the path satisfies
      $\beta_\gfS(e_\gfS) = (1,e,0)$,
    \item all other edges $(e_\gfG, e_\gfS)$ on the path satisfy
      $\beta_\gfS(e_\gfS) = (1,e,1)$.
    \end{enumerate}
  \end{enumerate}
  
  We are now ready to construct a GWA $M = M_e$ such that there is an
  $M$-run from $w$ to $w'$ if one of these cases occurs. Such a GWA is
  obtained from the $\gfB^*$-labelled graph $(\gfS, \beta_\gfS)$ as
  follows:
  \begin{enumerate}
  \item Make three disjoint copies of $V(\gfS)$, say
    $S_i = V(\gfS) \times \{i\}$ for $i=1, 2, 3$.
  \item In $S_1$ retain only vertices $(v,1)$ with
    $\beta_\gfS(v) = e^-$, remove others.
  \item In $S_2$ retain only vertices $(v,2)$ with
    $\beta_\gfS(v) = e$, remove others.
  \item In $S_3$ retain only vertices $(v,3)$ with
    $\beta_\gfS(v) = e^+$, remove others.
  \item Set $Q = S_1 \cup S_2 \cup S_3$ as the set of states.
  \item Construct then $\Delta\subset Q\times E(\gfA)\times Q$ as
    follows: for each edge $e_\gfS \in E(\gfS)$, say with
    $\beta_\gfS(e_\gfS) = (i,e,j)$, include
    $((e^-, 1+i), \alpha_\gfS(e_\gfS), (e^+, 3-j))$ in $\Delta$.
  \item As initial states choose $I = S_1$, and as final states
    $F = S_3$.
  \end{enumerate}

  Setting $M = (Q,I,F,\Delta)$, we see that $M$-runs are in bijective
  correspondence with the $\gfG$-paths and $\gfG$-edges defining edges
  of $\gfG\rtimes\gfS$ with label $e$, in particular there is an edge
  with label $e$ from $w$ to $w'$ if and only if there is an $M$-run
  from $w$ to $w'$. This concludes the proof that if $\gfG$ simulates
  $\gfH$, then a vertex blow-up of $\gfG$ GWA-simulates $\gfH$.

  $(2)\Rightarrow(1)$. Recall from Lemma~\ref{lem:simultrans} that
  simulation is transitive, so it suffices to show separately that if
  $\gfH$ is a vertex blow-up of $\gfG$ then $\gfG$ simulates $\gfH$,
  and that if $\gfG$ GWA-simulates $\gfH$, then it simulates $\gfH$.

  Vertex blow-ups are in fact a special case of simulation: if
  $k\colon V(\gfA) \to \N$ is a function, let $\gfS$ be the graph
  $\gfA^k$; with $\gfB=\gfA^k$ and $\alpha\colon\gfS\to\gfA$ the
  canonical labelling of $\gfA^k$ and
  $\beta\colon\gfS\to\gfB\to\gfB^*$ the identity and natural
  inclusion, we see that $\gfG\rtimes\gfS$ is precisely the vertex
  blow-up of $\gfG$ by $k$.

  We next show that if $\gfG$ GWA-simulates $\gfH$, then it simulates
  $\gfH$. For $e\in\gfB$, let $Q_e\supseteq I_e,F_e$ be the stateset,
  initial and final states of the automaton $M_e$, and assume all
  $Q_e$ are disjoint.  The nodes of the simulator $\gfS$ will be
  \[ V(\gfS) = V(\gfA) \sqcup \bigsqcup_{e \in \gfB} Q_e.
  \]
  The $\gfB^*$-labelling on vertices is given by
  $\beta(u) = \lambda(u)$ for $u \in V(\gfA)$, and $\beta(q) = e$ if
  $q \in Q_e$. In $E(\gfS)$ we include an edge with $\gfB^*$-label
  $(0,e,0)$ from $e^-$ to $e^+$ if there is a run of length one in
  $M_e$, with a suitable $\alpha$-label, namely if there exists
  $(s,\alpha(e'),t) \in \Delta$ with $s \in I_e$, $t \in F_e$, and
  $\beta(e') = e$.

  For each transition $(u, e', u')$ of $M_e$ (so $u, u' \in Q_e$,
  $e' \in E(\gfA)$), we include an edge in $\gfS$ with $\alpha$-label
  $e'$ and $\beta$-label $(1,e,1)$ from $u$ to $u'$. For each
  $u \in V(\gfA)$, we include an edge with $\alpha$-label $e'$ and
  $\beta$-label $(0,e,1)$ from $u$ to $q \in Q_e$ whenever there is
  $(s,e',q) \in \Delta$ with $s \in I$. Symmetrically, for each
  $u \in V(\gfA)$, we include an edge with $\alpha$-label $e'$ and
  $\beta$-label $(1,e,0)$ from $q \in Q_e$ to $u$ whenever there is
  $(q,e',s) \in \Delta$ with $s \in F$.

  Then the simulated graph $\gfG\rtimes\gfS$ can be seen to be
  isomorphic to the GWA-simulated graph $\gfH$. We conclude that
  GWA-simulation and simulation are equivalent concepts.

  $(2)\Leftrightarrow(3)$. It is enough to show that $\gfH$ is a
  regular $\gfB$-subdivision of $\gfG$ if and only if $\gfG$ simulates
  $\gfH$. This is clear from the definition of a regular language,
  because the possible paths corresponding to valid $M_e$-runs for
  $M_e$ satisfying $I \cap F = \emptyset$ form precisely a regular
  language of nonempty words.
\end{proof}


\section{Subshifts}

We are ready define subshifts using the language of graphs introduced
in the previous section.

\subsection{Subshifts of finite type}

\begin{defn}
  Let $\gfG$ be a graph, possibly labelled. A \emph{directed $\Hom$-shift} or
  \emph{DHS} 
    with carrier $\gfG$ is the space $\Hom(\gfG,\gfF)$, for some
  finite graph $\gfF$.

  If $\gfG$ is $\gfA$-labelled, then $\gfF$ should also be
  $\gfA$-labelled, and the space of homomorphisms should be taken in
  the appropriate category $\Graph_{/\gfA}$.
\end{defn}

DHS are nothing more than a formalism for subshifts of finite type,
and as explained in Section~\ref{ss:cayley} this definition is 
dynamically entirely equivalent to the more standard definition
of a subshift of finite type by finitely many allowed (or forbidden)
patterns.

The topology on $\Hom(\gfG,\gfF)$ is the usual function topology;
namely, $\Hom(\gfG,\gfF)$ is a closed subset of
$V(\gfF)^{V(\gfG)}\times E(\gfF)^{E(\gfG)}$, and therefore is compact.

\begin{exple}
  Consider $\gfG=\N\sqcup \N$, with $n^+=n+1$ and
  $n^-=n$. Geometrically, it is a one-sided ray. Consider the
  following finite graphs:
  
  
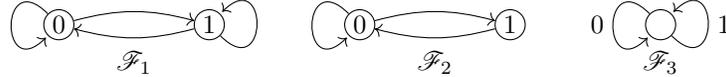
\centerline{\begin{tikzpicture}[every state/.style={inner sep=1pt,minimum size=4mm}]
      \node[state] (00) at (0,0) {$0$};
      \node[state] (01) at (2,0) {$1$};
      \draw[->] (00) edge[bend left=15] (01)
      (00) edge[in=225,out=135,loop] (00)
      (01) edge[bend left=15] (00)
      (01) edge[in=45,out=-45,loop] (01);
      \node at (1,-0.5) {$\gfF_1$};
      \begin{scope}[xshift=4cm]
        \node[state] (10) at (0,0) {$0$};
        \node[state] (11) at (2,0) {$1$};
        \draw[->] (10) edge[bend left=15] (11)
        (10) edge[in=225,out=135,loop] (10)
        (11) edge[bend left=15] (10);
        \node at (1,-0.5) {$\gfF_2$};
      \end{scope}
      \begin{scope}[xshift=7cm]
        \node[state] (20) at (1,0) {};
        \draw (20) edge[in=45,out=-45,loop] node[right] {$1$} (20)
        (20) edge[in=225,out=135,loop] node[left] {$0$} (20);
        \node at (1,-0.5) {$\gfF_3$};
      \end{scope}        
    \end{tikzpicture}}
  Then, in the standard symbolic dynamics terminology \cite{lind:symdyn},
  \begin{itemize}
  \item $\Hom(\gfG,\gfF_1)$ is the full vertex shift $\{0,1\}^\N$;
  \item $\Hom(\gfG,\gfF_2)$ is the ``golden mean'' shift
     $\{x\in\{0,1\}^\N: \forall n \in \N: x_nx_{n+1}=0\}$; 
  \item $\Hom(\gfG,\gfF_3)$ is the full edge shift $\{0,1\}^\N$.
  \end{itemize}
\end{exple}

\subsection{The tiling problem}

\begin{defn}\label{defn:tp}
  Let $\gfG$ be a graph, possibly labelled. The \emph{tiling problem}
  for $\gfG$ is the following decision problem:
  \begin{description}
  \item[given] a finite graph $\gfF$;
  \item[decide] is $\Hom(\gfG,\gfF)$ non-empty?
  \end{description}

  Let now $\Gamma$ be a family of graphs. The \emph{tiling problem}
  for $\Gamma$ is the problem of, given a finite graph $\gfF$,
  deciding whether $\Hom(\gfG,\gfF)$ is non-empty for at least
  one $\gfG\in\Gamma$.
\end{defn}

We call the tiling problem for $\gfG$ (respectively $\Gamma$)
\emph{solvable} if there exists an algorithm that truthfully answers 
its tiling problem.

\begin{thm}\label{thm:simul}
  Let $\gfG,\gfH$ be labelled graphs, and assume that $\gfG$ simulates
  $\gfH$ and $\gfH$ is weakly \'etale. If $\gfG$ has solvable tiling problem,
  then so does $\gfH$.
\end{thm}

\begin{proof}
  Let $\gfF$ be an instance of the tiling problem for $\gfH$, namely a
  finite graph with same labelling as $\gfH$. Since $\gfH$ is
  simulated by $\gfG$, there exists a finite simulator $\gfS$ with
  $\gfH\cong\gfG\rtimes\gfS$. We may algorithmically compute the
  finite graph $\gfF^\gfS$, and by hypothesis we may decide whether
  $\Hom(\gfG,\gfF^\gfS)$ is non-empty. Now by
  Lemma~\ref{lem:almostadjoint} we have
  $\Hom(\gfG,\gfF^\gfS)\neq\emptyset$ if and only if
  $\Hom(\gfH,\gfF)\neq\emptyset$, so solving the tiling problem for
  $\gfG$ on $\gfF^\gfS$ solves at the same time the tiling problem for
  $\gfH$ on $\gfF$.
\end{proof}

This result extends readily to families of graphs; this is the most
general result we obtain:
\begin{defn}
  \label{def:WeaklyMaps}
  For two graphs $\gfG,\gfH$, we say that $\gfG$ \emph{weakly maps} to
  $\gfH$ if every finite subgraph of $\gfG$ maps to $\gfH$; namely
  $\Hom(\gfG',\gfH)\neq\emptyset$ for all finite subgraphs $\gfG'$ of
  $\gfG$.
\end{defn}
In case $\gfH$ is finite, this is equivalent, by compactness, to
$\Hom(\gfG,\gfH)\neq\emptyset$, but in general it differs: for example
if $\gfG$ is an infinite ray and $\gfH$ is a disjoint union of
arbitrarily long finite rays.

\begin{thm}\label{thm:simulfamily}
  Let $\Gamma,\Delta$ be families of graphs, and assume that $\Gamma$
  ``weakly'' simulates $\Delta$ in the following sense: there is a
  finite simulator $\gfS$ such that
  \begin{enumerate}
  \item every graph in $\Delta$ may be simulated: for every $\gfH\in\Delta$ there is $\gfG\in\Gamma$ such that $\gfG\rtimes\gfS$ is weakly \'etale and weakly maps to $\gfH$;
  \item simulated graphs are images of $\Delta$: for every
    $\gfG\in\Gamma$ there is $\gfH\in\Delta$ that weakly maps to
    $\gfG\rtimes\gfS$.
  \end{enumerate}
  If $\Gamma$ has solvable tiling problem, then so does $\Delta$.
\end{thm}
\begin{proof}
  Let $\gfF$ be an instance of tiling problem for $\Delta$. As in the
  proof of Theorem~\ref{thm:simul}, we may solve the tiling problem
  for $\Gamma$ on instance $\gfF^\gfS$, so it suffices to prove
  \[ \exists \gfG \in \Gamma: \Hom(\gfG,\gfF^\gfS) = \emptyset \iff \exists \gfH \in \Delta: \Hom(\gfH,\gfF) = \emptyset \]
  
  If $\Hom(\gfG,\gfF^\gfS)$
  is non-empty for some $\gfG\in\Gamma$, then
  $\Hom(\gfG\rtimes\gfS,\gfF)\neq\emptyset$ by
  Lemma~\ref{lem:almostadjoint}, so by the second assumption there is
  a graph $\gfH\in\Delta$ such that
  $\Hom(\gfH',\gfF)\supseteq\Hom(\gfG\rtimes\gfS,\gfF)\circ\Hom(\gfH',\gfG\rtimes\gfS)\neq\emptyset$
  for all finite $\gfH'\subseteq\gfH$. Since $\gfF$ is finite,
  $\Hom(\gfH,\gfF)=\lim\Hom(\gfH',\gfF)\neq\emptyset$ by compactness.

  Conversely, if $\Hom(\gfH,\gfF)\neq\emptyset$ for some $\gfH\in\Delta$,
  then by the first assumption there is a graph $\gfG\in\Gamma$ such
  that $\gfG \rtimes \gfS$ is \'etale and $\Hom(\gfG'\rtimes\gfS,\gfH)\neq\emptyset$ for all finite
  $\gfG'\subseteq\gfG$, so $\Hom(\gfG'\rtimes\gfS,\gfF)\neq\emptyset$,
  so $\Hom(\gfG',\gfF^\gfS)\neq\emptyset$ by
  Lemma~\ref{lem:almostadjoint} (because subgraphs of \'etale graphs are \'etale). Now $\gfF^\gfS$ is finite so
  $\Hom(\gfG,\gfF^\gfS)=\lim\Hom(\gfG',\gfF^\gfS)\neq\emptyset$ by
  compactness.
\end{proof}

\begin{remark}
\label{rem:TPEquivalent}
If $\Gamma$ and $\Delta$ are singletons with the same labelling graph $\gfA = \gfB$, and $\gfS = \gfA$ is the trivial simulator (both labellings are the identity map), the theorem reduces to the fact that if two graphs weakly map to each other, then their tiling problems are equivalent.
\end{remark}

The literature often mentions the ``seeded tiling problem''; we shall
return to it in~\S\ref{ss:cayley}. It suffices, for now, to define it
as a tiling problem for a graph with the ``sunny side up'' labelling, see
Example~\ref{ex:ssu}.

As a side-note, it is an interesting, and not yet fully understood,
problem to determine which groups admit a ``sunny side up'' labelling
defined as a factor (shift-commuting continuous image) of a shift of
finite type (equivalently, factor of a DHS), see \cite{dahmani:symbolic}.
Such images are called \emph{sofic}. For us, the ``sunny side
up'' labelling is fixed once and for all on the graph, and exists
independently of the tiling problem. The soficity of the
sunny-side-up on the lamplighter group is a side-effect of our
constructions, see Proposition~\ref{prop:SSUSofic}.

Our undecidability result rests on the following result of Wang,
proven itself by a reduction to the halting problem of Turing
machines:
\begin{thm}[Kahr, Moore \&\ Wang~\cites{kahr-moore-wang:aea, wang:ptpr2}]
  The seeded tiling problem on $\Z^2$ (namely, on the Cayley graph of
  $\Z^2$ marked by the ``sunny side up'') is unsolvable.
\end{thm}

Since the quadrant and half-plane simulate the plane, by
Example~\ref{ex:simquadrant}, it follows that the tiling problem on
the quadrant (with the markings from Example~\ref{ex:quadrant})
is also unsolvable.


One can interpret Robinson's classical proof of undecidability of the
tiling problem \cite{robinson:undecidability} as simulating
seeded-$\Z^2$ on $\Z^2$. We give an informal explanation that
concentrates on the link to weak mappings, assuming the reader is
familiar with the proof. More specifically, we have in mind Kari's
presentation~\cite{kari:tp}*{Section~8.2.4}. This continues
Example~\ref{ex:CompressingRectangles}.

\begin{exple}\label{ex:Robinson}
  In Robinson's proof of the undecidability of the tiling problem, one
  builds a subshift of finite type whose configurations contain
  drawings of squares (containing squares containing squares\dots)
  around a square grid (possibly with some degenerate squares), so
  that
  \begin{itemize}
  \item each configuration contains arbitrarily large finite squares, and
  \item the larger a square is, the more ``free rows'' (resp. free
    columns) it contains; a free row is one that does not hit (a
    smaller square inside).
  \item in some configuration every cell of the grid is contained in a
    finite square.
\end{itemize}

One can use additional signals to mark the free rows and free columns
inside the rectangles, and the cells that are part of a free row and a
free column necessarily form a set-theoretic rectangle. Thus, every
square can be seen as being of the type in
Example~\ref{ex:CompressingRectangles}.

The bottom row and leftmost column can be made visible in each cell,
and thus we can modify the construction in
Example~\ref{ex:CompressingRectangles} slightly to obtain a simulator
$\gfS$ such that the simulated rectangles
$\{0,1,...,|A|-1\} \times \{0,1,...,|B|-1\}$ simulate a rectangle on
the left corner of $\N^2$ with its natural vertex labelling from
Example~\ref{ex:quadrant}.

We claim that the family $\Gamma$ of all tilings of $\Z^2$ by this
tile set then weakly simulates the singleton family
$\Delta = \{\N^2\}$ (with its natural labelling). To prove
\begin{quote}
  every graph in $\Delta$ may be simulated: for every $\gfH\in\Delta$
  there is $\gfG\in\Gamma$ such that $\gfG\rtimes\gfS$ is weakly
  \'etale and weakly maps to $\gfH$;
\end{quote}
take $\gfG \in \Gamma$ from the third item. Then every finite subgraph
of $\gfG \rtimes \gfS$ is contained in one that is a disjoint union of
simulated full finite squares. These are subgraphs of $\N^2$, so to
get a graph homomorphism, on each such square separately we can take
its graph embedding into $\N^2$.  It is clear that all simulated
graphs are weakly \'etale because of the form of the simulator, so in
particular $\gfG \rtimes \gfS$ is. To prove
\begin{quote}
  simulated graphs are images of $\Delta$: for every $\gfG\in\Gamma$
  there is $\gfH\in\Delta$ that weakly maps to $\gfG\rtimes\gfS$,
\end{quote}
for any finite subgraph of $\gfH' \subset \gfH$, taking a large enough
square in $\gfG$, we see that $\gfG \rtimes \gfS$ contains an actual
copy of $\gfH'$, and we can use the graph embedding as the graph
homomorphism.

Since $\N^2$ with the labelling from Example~\ref{ex:quadrant}
simulates seeded-$\Z^2$, $\Gamma$ also weakly simulates
$\{\mathrm{seeded-}\Z^2\}$, since weak simulation is easily seen to be
transitive.
\end{exple}

\subsection{General SFTs on Cayley and Schreier graphs}\label{ss:cayley}
We develop more the link between graphs and groups, sketched in
Example~\ref{ex:cayley}. Consider a group $G=\langle S\rangle$, and a
set $X$ on which $G$ acts on the right. We associate with $X$
the \emph{Schreier graph} $\gfX$ with vertex set $X$ and edge set
$X\times S$, with as usual $(x,s)^-=x$ and $(x,s)^+=x s$.  If
$S=S^{-1}$, then $\gfX$ is unoriented with $(x,s)'=(x s,s^{-1})$.

\begin{defn}\label{def:SFT}
  A \emph{subshift of finite type} or \emph{SFT} $\Omega$ on $X$ is given by
  a finite set $A$ called the \emph{alphabet}, an integer $n$ called
  the \emph{radius}, and a subset $\Pi$ of $A^{S^{\le n}}$ called the
  \emph{allowed patterns}. It is defined as
  \[\Omega=\{\alpha\in A^X: \forall x\in X:\exists P_x\in\Pi: P_x(w)=\alpha(x w)\text{ for all }w\in S^{\le n}\},\]
  namely the set of labellings of $X$ by elements of $A$ such that, in
  every neighbourhood of size $n$, the labels form an allowed pattern.
\end{defn}

The definition above is a generalization of the more 
classical notion, in which $X=G$ with action by translation. The
\emph{tiling problem} for $X$ asks for an algorithm that, given
$\Pi\subseteq A^{S^{\le n}}$, determines whether the corresponding
SFT $\Omega$ is non-empty.

An important variant is the \emph{seeded tiling problem}, which asks
for an algorithm that, given $\Pi\subseteq A^{S^{\le n}}$ and and
$x_0\in X$ and $a_0\in A$, determines whether the corresponding
$\Omega$ contains a configuration $\alpha\in A^X$ with
$\alpha(x_0)=a_0$.

As we shall now see, the tiling problem for $X$ is essentially
equivalent to the tiling problem on graphs from
Definition~\ref{defn:tp}, and the seeded tiling problem is essentially
equivalent to the tiling problem on a graph with a marked vertex (as
in Example~\ref{ex:ssu}).

We need a few technicalities. First, SFTs of course lose all
edge information, so we need the following definition.

\begin{defn}
  A DHS $\Hom(\gfG, \gfF)$ is \emph{weakly
    resolving} if $\gfF$ is weakly \'etale.
\end{defn}

Note that one can make any DHS weakly resolving at the cost of adding
a few more vertices to $\gfF$, without changing the system up to isomorphism
(in the sense of the following definition). Even when no information is lost,
the constructions between SFTs and DHSs are only inverses of each other
up to isomorphism.

\noindent We now give a suitable notion of isomorphism:
\begin{defn}\label{def:Isomorphism}
  Let $\Omega_1, \Omega_2$ be SFTs on a $G$-set $X$. A \emph{block
    map} from $\Omega_1$ to $\Omega_2$ is a map of the form
  \[ f(\eta)_x = f_{\mathrm{loc}}(\eta_{xg_1},\dots, \eta_{xg_k}) \]
  for some $f_{\mathrm{loc}}\colon A^{k} \to B$ and some fixed
  $g_1,\dots,g_k \in G$.  We say two SFTs are \emph{(block map)
    isomorphic} if there are block maps
  $f_i\colon \Omega_i \to \Omega_{3-i}$ which are inverses of each
  other. Similarly one can define block maps and isomorphisms between
  DHSs $\Hom(\gfX, \gfF_1)$ and $\Hom(\gfX, \gfF_2)$, as well as
  between SFTs and DHSs.
\end{defn}

In the classical situation $G = X$ with $G$ acting by
$g\eta_{h} = \eta_{g^{-1}h}$, morphisms between SFTs are just the
usual morphisms of topological $G$-systems, namely shift-commuting
continuous functions, and isomorphisms are just the \emph{topological
  conjugacies}, or shift-commuting homeomorphisms.

In the general situation, there are some subtleties.  If $X$ is a
$G$-set, then $G$ also acts on $A^X$ by $g\eta_x = \eta_{xg}$, but the
above block maps are not the continuous functions commuting with this
action.  Indeed, if $G$ acts $\infty$-transitively on $X$ then only
finitely many continuous functions commute with its natural action on
$A^X$ (but there are plenty of morphisms in the above sense); on the
other hand if $G = \Z$ and no element in $X$ has infinite orbit but
infinitely many elements have nontrivial orbit, then there are
uncountably many shift-commuting continuous functions on $A^X$.  We
also note that bijectivity of a block map $f\colon A^X \to A^X$ for a
transitive $G$-set $X$ is equivalent to having a block map inverse,
but this is no longer true if the action is not transitive.

\begin{prop}\label{prop:sft=vsft}
  Let $G$ be a monoid acting on a set $X$. SFTs on $X$ are equivalent
  by block map isomorphisms to weakly resolving DHSs.
\end{prop}

In essence, every SFT can be converted into a DHS, and
vice versa; the constructions are defined by local rules, and
involve no funny business.

The precise statement we are referring to is the following:
In the proof we construct a mapping $F$ that turns an
SFT $\Omega$ into a weakly resolving DHS
$\Hom(\gfX, \gfF)$, and give another construction $F'$ for the other
direction. These extend to functors between the appropriate categories,
when one takes the morphisms to be the block maps, and the functors
$F$ and $F'$ give an equivalence of categories.

This equivalence is also ``by block map isomorphisms'', in that the
object mappings of $F$ and $F'$ are themselves given by invertible block maps;
in the classical dynamical situation of $G$-subshifts, the object mappings
are topological conjugacies.

\begin{proof}
  We only give the object mappings and show that they are isomorphisms.
  the choices of mappings between morphisms are obvious,
  and verification that the resulting funtors are a categorical equivalence
  is routine.

  In the direction ``$\gfX$ to $X$'': let $\gfF$ be a finite graph
  with no vertex labels and edge labels $S$, and consider the DHS
  $\Hom(\gfX, \gfF)$ that it defines. We construct a SFT
  $\Omega$ on $X$ as follows: we set $A = V(\gfF)$ and $n = 1$, and
  define $\Pi \subset A^{S^{\leq 1}}$ by taking $P \in \Pi$ if for all
  $s \in S$, the graph $\gfF$ has an edge with label $s$ from $P(1)$
  to $P(s)$. The next two paragraphs describe maps
  $\Omega\leftrightarrow\Hom(\gfX, \gfF)$.

  Firstly, consider $\eta \in \Hom(\gfX, \gfF)$ and construct
  $\alpha \in A^X$ by $\alpha(x)\coloneqq\eta(x)$ for all $x \in X$;
  namely, $\alpha = \eta\restriction V(\gfX)$.  We claim
  $\alpha \in \Omega$. Indeed consider $x \in X$, and define
  $P_x\in A^{S^{\le1}}$ by $P_x(w)\coloneqq\alpha(x w)$ for all
  $w \in S^{\leq 1}$. Then in $\gfX$ the edge $(x, s)$ has label $s$
  from $x$ to $x s$, and we have $\eta(x) = \alpha(x) = P_x(1)$ and
  $\eta(x s) = \alpha(x s) = P_x(s)$, to $\eta$ map the edge $(x, s)$ to
  some edge in $\gfF$ with label $s$ from $P_x(1)$ to $P_x(s)$; in
  particular such an edge exists, and we have $P_x \in \Pi$.

  Secondly, consider $\alpha \in \Omega$, and construct a corresponding
  homomorphism $\eta \in \Hom(\gfX, \gfF)$. We define
  $\eta(x) \coloneqq \alpha(x)$ on vertices. Consider an edge $(x, s)$
  with label $s$ from $x$ to $x s$, and define $P_x\in A^{S^{\le1}}$
  by $P_x(w)\coloneqq\alpha(x)$ for all $w\in S^{\le1}$. Because
  $P_x \in \Pi$, there must be an edge from $P_x(1)$ to $P_x(s)$ with
  label $s$ in $\gfF$, and since $\gfF$ is weakly resolving there is a
  unique such edge, which we call $\eta((x, s))$. In this manner we
  defined a graph morphism $\eta\colon\gfX\to\gfF$.

  The constructions are clearly inverses of each other and are given by
  block maps.

  In the direction ``$X$ to $\gfX$'': let $\Omega$ be an SFT
  on $X$ for some alphabet $A$, some $n \geq 0$ and some
  $\Pi \subset A^{S^{\leq n}}$. We define a weakly resolving DHS
  via a graph $\gfF$, which is constructed as follows:
  $V(\gfF) = \Pi$, and for each $P, P' \in \Pi$ we include in $\gfF$
  an edge $e = (P, s, P')$ labelled $s \in S$ with $e^+ = P$ and
  $f^- = P'$ whenever $P'(w) = P(s w)$ for all $w \in S^{\leq
    n-1}$. This graph is obviously weakly resolving, since the
  labelling map and the head and tail maps are projections. The next
  two paragraphs describe maps
  $\Hom(\gfX, \gfF)\leftrightarrow\Omega$.

  Firstly, consider $\alpha \in \Omega$. We construct a homomorphism
  $\eta \in \Hom(\gfX, \gfF)$ as follows. On vertices $x \in X$, set
  $\eta(x) \coloneqq P$ with $P(w) \coloneqq \alpha(x w)$ for all
  $w \in A^{S^{\leq n}}$. For edges $(x, s)$ of $\gfX$, on whose
  extremities we have already defined $\eta(x) = P$, $\eta(x s) = P'$,
  note that by definition we have
  \[ P'(w) = \alpha(x s \cdot w) = \alpha(x \cdot s w) = P(s w)\] for
  all $w \in S^{\leq n-1}$, so $\gfF$ has an edge labelled $s$ from
  $P$ to $P'$. We elt $\eta((x, s))$ be this edge. By definition of
  $\gfF$, we have indeed defined a graph morphism
  $\eta \in \Hom(\gfX, \gfF)$.

  Secondly, consider $\eta \in \Hom(\gfX, \gfF)$, and construct a
  configuration $\alpha \in A^X$ as follows. Set
  $\alpha(x)\coloneqq\eta(x)(1)$, namely look at the pattern $\eta(x)$
  and extract its symbol at the identity. We claim
  $\alpha \in \Omega$. To see this, define $P_x\in A^{S^{\le n}}$ by
  $P_x(w) \coloneqq \alpha(x w)$ for all $w \in S^{\leq n}$. To prove
  that all $P_x$ belong to $\Pi$, it suffices to show $P_x = \eta(x)$,
  since then $P_x \in V(\gfF) = \Pi$. We show this simultaneously for
  all $x \in X$, considering all $w\in S^{\le n}$ in order of
  increasing length. If $|w|=0$, this is true by definition, since
  $P_x(1) = \alpha(x) = \eta(x)(1)$. Supposing the claim is true for
  $w$, consider a word $s w$ with $s \in S$. Unwrapping the definitions we
  have
  \[ P_x(s w) = \alpha(x \cdot s w) = \alpha(x s\cdot w) = P_{x s}(w) = \eta(x s)(w),
  \]
  so we are led to show $\eta(x s)(w) = \eta(x)(s w)$. The edge
  $(x, s)\in E(\gfX)$ has label $s$ and extremities
  $(x, s)^+ = x s, (x, s)^- = x$, so it we write
  $\eta(x, s) = (P, s, P') \in E(\gfF)$ then $P'(w) = P(s w)$, that is,
  $\eta(x s)(w) = P'(w) = P(s w) = \eta(x)(s w)$ as required.
  
  The constructions are clearly inverses of each other and are given by block maps.
\end{proof}

In the seeded case, we fix an ``origin'' $o$ in the $G$-set $X$, and
consider the Schreier graph of $X$ with generating set $S$, with the
sunny-side-up labelling where $o$ is mapped to the vitellus, and all
others to albumen (edge labellings are uniquely determined). The DHS
on this graph are \emph{seeded-DHSs}. (This is a slight generalization
of Example~\ref{ex:ssu}.) We define similarly a seeded variant of
SFTs.

\begin{defn}\label{defn:SeededVertexSFT}
  A \emph{seeded SFT} $\Omega$ on $(X, o)$ is given by a finite set
  $A$ called the \emph{alphabet}, an integer $n$ called the
  \emph{radius}, and a subset $\Pi$ of
  $(A \times \{0,1\})^{S^{\le n}}$ called the \emph{allowed
    patterns}. Writing $\pi_1, \pi_2$ respectively for the pointwise
  projections to the first and second coordinate of the alphabet, the
  SFT $\Omega$ is defined as
  \begin{align*}
  \Omega=\{\pi_1(\alpha) \;:\; & \alpha \in (A \times \{0,1\})^X \wedge (\pi_2(\alpha)_x = 1 \iff x = o) \wedge \\
   & \forall x \in X: \exists P_x \in \Pi: P_x(w)=\alpha(x w)\text{ for all }w\in S^{\le n} \}.\qedhere
  \end{align*}
\end{defn}

In other words, the allowed patterns see the marking at the origin, but this is
erased in the actual configurations.

We can define isomorphisms on seeded SFTs and on seeded DHSs by block maps,
similarly as in the unseeded case. The only difference is that the block map
is allowed to behave differently when near the seed, which can be implemented
as in Definition~\ref{defn:SeededVertexSFT} (allowing them to see the seed
position). The proof of the following proposition is similar to that of
Proposition~\ref{prop:sft=vsft}, and is omitted.

\begin{prop}\label{prop:ssft=svsft}
  Let $G$ be a monoid acting on a set $X$. Seeded SFTs on $X$ are equivalent
  by block map isomorphisms to weakly resolving seeded DHSs.
\end{prop}

\section{The lamplighter group}\label{ss:ll}

Our main result applies to a specific example of group, the
``lamplighter group''. Write $\Z/2$ for the two-element group. The group
$L$ may be defined in various manners: it is the wreath product
$L=\Z/2\wr\Z$, namely the extension of
$\{f\colon\Z\to\Z/2\text{ a finitely supported set map}\}$ by $\Z$
acting by shifts.

Writing $a$ for the generator of $\Z$ and $d$ for
the delta-function $f\colon\Z\to\Z/2$ taking value $1$ at $0$ and $0$
elsewhere, we have the presentation
\[L=\langle a,d\mid d^2,[d,d^{a^n}]\text{ for all }n>0\rangle.\]
We shall prefer the more symmetric presentation, setting $b=da$,
\[L=\langle a,b\mid (a^n b^{-n})^2\text{ for all }n>0\rangle. \]

The reason $L$ is called the ``lamplighter group'' is the
following. Picture a two-way-infinite street, with a house at every
integer, and a lamp between any two neighbouring houses. The
``lamplighter'' starts at house $0$, and has a schedule to follow:
turn on some specified lamps, and stop at a given house. This schedule
is an element of $L$. It may be expressed as a word in elementary
operations: ``move to the next house'' ($a$ or $a^{-1}$, depending on
the direction), and ``move to the next house, flipping the state of
the lamp along the way'' ($b$ or $b^{-1}$).

This description, where the lamps are between houses, avoids the issue of
whether the lamplighter flips the lamp before or after moving. We find this convenient,
and thus from now on consider the lamps to be on $\Z+1/2$. We use the notation
$\ZO=\Z+1/2$ for the half-integers.

We shall make use of two representations for elements of $L$: on the
one hand, words over $\{a^{\pm1},b^{\pm1}\}$ as in the first
paragraph of this section; and also as a global description of lamp configurations and
final position, as follows: if at the end of its schedule the
lamplighter is at position $n\le0$ and for all $i\in\ZO$ the lamp
at position $i$ is in state $s_i$, then the corresponding element of
$L$ is written
\[s_{-N}\cdots s_{n-1}\markout s_n\cdots s_{-1/2}\markin s_{1/2}\cdots s_M,
\]
with $N,M$ minimal such that $s_{-N},s_M$ are non-zero. If $n\ge0$ the
same notation is used, but with now the `$\markin$' to the left of the
`$\markout$'. Thus the expressions for the generators are respectively
\[a=\markin0\markout,\qquad b=\markin1\markout,\qquad a^{-1}=\markout0\markin,\qquad
  b^{-1}=\markout1\markin.
\]
Global descriptions may be multiplied as follows: align the `$\markout$'
of the first with the `$\markin$' of the second, and add bitwise the strings
of $0$ and $1$. The `$\markin$' and `$\markout$' of the result are respectively
the `$\markin$' of the first and the `$\markout$' of the second operand.

From a description $u\markout v\markin w$ or $u\markin v\markout w$
one easily reads the final position $n$ of the lamplighter, and the
states $(s_i)_{i\in\Z+1/2}$ of the lamps. In that notation, the
product of $(r,m)$ and $(s,n)$ is $(t,m+n)$ with $t_i=r_i+s_{i-m}$.

Sometimes, the origin in a global description is unimportant, and is
omitted; so we may consider partial descriptions of the form
`$u\markout v$'. Such partial descriptions may be acted upon by $L$,
by right multiplication. They are naturally identified with the
homogeneous space $\langle a\rangle\backslash L$.

\subsection{The Cayley graph of \boldmath $L$}\label{ss:cayleyLL}
The Cayley graph of $L$, in the generating set $\{a,b\}^{\pm1}$, is a
special case of \emph{horocyclic product},
see~\cite{bartholdi-n-w:horo}. Let first $\mathscr T_1,\mathscr T_2$
be two $3$-regular trees, and choose on each of them a infinite ray
$\xi_i\colon(-\N)\to \mathscr T_i$. (These rays define points at
infinity $\omega_i$ in the respective trees). The corresponding
\emph{Busemann functions} are $h_i\colon \mathscr T_i\to\Z$ defined by
$h_i(v)=\lim_{n\to-\infty}n+d(v,\xi_i(n))$; the points close to
$\omega_i$ have very negative $h_i$, and points with same Busemann
function value form horocycles with respect to the boundary points
$\omega_i$. Now the \emph{horocyclic product} of these trees, with
respect to these Busemann functions, is
\[\gfL=\{(v_1,v_2)\in \mathscr T_1\times \mathscr T_2:h_1(v_1)+h_2(v_2)=0\}.
\]
Formally speaking, we have defined the vertex set of $\gfL$ above;
there is then an edge between $(v_1,v_2)$ and $(w_1,w_2)$ whenever
there are edges in $\mathscr T_i$ between $v_i$ and $w_i$ for all
$i=1,2$. One could also say that $h_i$ is extended linearly to edges,
and take the definition of $\gfL$ above at face value.

Yet equivalently, the Busemann function $h_1$ defines a graph morphism
$\mathscr T_1\to\gfC(\Z,\{\pm1\})$, and $-h_2$ defines likewise a graph morphism $\mathscr T_2\to\gfC(\Z,\{\pm1\})$. Then $\gfL$ is the pullback
\[\gfL=\mathscr T_1\times_{\gfC(\Z,\{\pm1\})}\mathscr T_2.\]

\begin{prop}[\cite{woess:dl}*{\S2}]
  The Cayley graph $\gfC(L,\{a,b\}^{\pm1})$ is the horocyclic product
  $\gfL$ defined above.
\end{prop}
The proof is in fact straightforward: picture $\mathscr T_1$ as having
its boundary point $\omega_1$ at the bottom and $\mathscr T_2$ as
having its boundary point $\omega_2$ at the top. With our choice of
orientation, a point in $\gfL$ is then a pair of points $(v_1,v_2)$ of
same height.

Label all edges of $\mathscr T_i$ by $\{0,1\}$ with the condition that
the edges on the rays $\xi_i$ are all $0$. Then a vertex
$(v_1,v_2)\in\gfL$ may be uniquely identified by the following data: a
height $n\in\Z$, a finite string $u\in\{0,1\}^*$ expressing the labels
on the geodesic from $v_1$ to $\xi_1$, and a finite string
$v\in\{0,1\}^*$ expressing the labels on the geodesic from $v_2$ to
$\xi_2$. The corresponding element of $L$ is
`$\operatorname{reverse}(u)\markout v$', with the extra `$\markin$'
inserted $n$ places to the left of the `$\markout$'.

Consider now a finite subgraph of $\gfL$ as follows: choose $H\in\N$
and vertices $v_i\in\mathscr T_i$ with $h(v_1)+h(v_2)=H$. There are
height-$H$ binary trees in $\mathscr T_i$ consisting of all vertices
$w_i$ with $h(w_i)=d(v_i,w_i)\le H$, and their product, in $\gfL$,
gives a height-$H$ tetrahedron. Such a tetrahedron is displayed in
Figure~\ref{fig:ll} for $H=4$.

These tetrahedra in $\gfL$ are naturally nested: every vertex belongs
to increasing sequences of tetrahedra, and every height-$H$
tetrahedron is naturally part of two height-$(H+1)$ tetrahedra, one
extending above it and one below it.

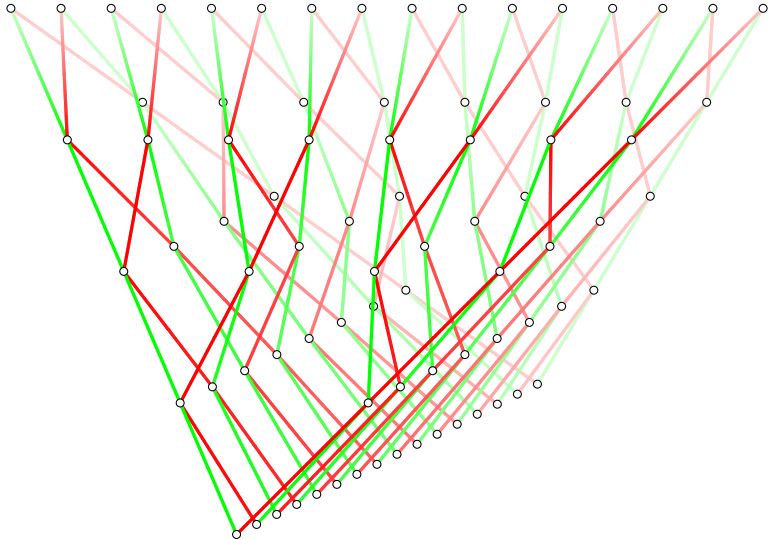
\begin{figure}
  \centerline{\begin{tikzpicture}[x={(5cm,0cm)},y={(-2cm,-1cm)},z={(0cm,6cm)}]
    \pgfmathsetmacro\h{4}
    \pgfmathsetlengthmacro\rad{3mm*pow(2,-\h)}
    \pgfmathsetmacro\hmo{\h-1}
    \foreach\k in {0,...,\h} {
      \pgfmathsetmacro\z{\k/\h}
      \pgfmathsetmacro\cz{1-\z}
      \pgfmathparse{pow(2,\h-\k)-1}
      \foreach\m in {0,...,\pgfmathresult} {
        \pgfmathsetmacro\y{-1+\m*2/(pow(2,\h-\k)-0.9999)}
        \pgfmathparse{pow(2,\k)-1}
        \foreach\l in {0,...,\pgfmathresult} {
          \pgfmathsetmacro\x{1-\l*2/(pow(2,\k)-0.9999)}
          \ifnum\k=\h\else
          \foreach\color/\pos in {green/1,red/0} {
            \pgfmathsetmacro\nx{1-(2*\l+(\pos ? mod(\m,2) : 1-mod(\m,2)))*2/(pow(2,\k+1)-1)}
            \pgfmathsetmacro\ny{-1+floor(\m/2)*2/(pow(2,\h-(\k+1))-0.9999)}
            \pgfmathsetmacro\nz{(\k+1)/\h}
            \pgfmathsetmacro\ncz{1-\nz}
            \pgfmathsetmacro\shade{60+40*\y}
            \pgfmathsetmacro\nshade{60+40*\ny}
            \shadepath{\color!\shade}{\color!\nshade}{\z*\x,\cz*\y,\z}{\nz*\nx,\ncz*\ny,\nz}
          }
          \fi
          \filldraw[fill=white] (\z*\x,\cz*\y,\z) circle (1.5pt);
        }
      }
    }
\end{tikzpicture}}
  \caption{(A portion of) the Cayley graph $\gfL$ of the lamplighter group, with generators $a$ in green and $b$ in red}\label{fig:ll}
\end{figure}

\subsection{The geometry of \boldmath $L$}
We describe some geometric aspects of the Cayley graph of $L$; these
will not be used elsewhere in the text, and serve as an illustration
of the relevance of $L$ to the tiling problem.

Firstly, $\gfL$ does not contain any embedded plane, so is a good
test case for Conjecture~\ref{conj:bs}. Indeed, assume there were
an injective, Lipschitz map $\Z^2\to\gfL$. In particular, the
elementary relation $[x,y]=1$ in $\Z^2$ would map to relations of
bounded length in $L$. Thus we may equivalently ask whether there
exists an embedded plane in the finitely presented group
\[L_N=\langle a,d\mid d^2,[d,d^{a^n}]\text{ whenever }0<n<N\rangle.\]
(Note that we use, for more convenience, the presentation of $L$ on
generators $\{a,d=a b^{-1}\}$). Now free groups do not contain
embedded planes, and
\begin{lem}
  The group $L_N$ is virtually free.
\end{lem}
\begin{proof}
  Consider the subgroup $K=\langle a,[d,a^N]\rangle$ of $L_N$. On the
  one hand, $K$ has index $2^N$: every element of $L_N$ may be written
  in the form $d^{a^{n_1}}\cdots d^{a^{n_k}}a^\ell$, and $K$ contains
  every such expression in which $\#\{i:n_i\equiv m\pmod 2\}$ is even
  for all $m=0,\dots,N-1$. On the other hand, one can check using
  the Reidemeister-Schreier procedure that $K$ is free on its
  generators~\cite{lyndon-s:cgt}*{\S II.4}.
\end{proof}

The group $L$ is ``amenable''; there are numerous equivalent
definitions of this property, but the simplest to state is probably
the graph-theoretical one: \emph{For every $\epsilon>0$, there exists
  a finite subset $F\subset L$ whose boundary
  $F\cdot\{a,b\}\setminus F$ has cardinality at most $\epsilon\#F$}.
These subsets may simply be taken to be the tetrahedra mentioned
above: a tetrahedron of height $H$ contains $(H+1)\cdot 2^H$ vertices,
and its boundary consists only of the upper and lower strips, so
contains $2\cdot 2^H$ vertices.



\newcommand\tikzllw[5]{
  \draw (0,-1) -- (-1,0) -- (0,1) -- (1,0) -- (0,-1) -- (0,1)
  (-1,0) -- (1,0) -- (0,0);
  \begin{scope}[scale=0.3]
    \node at (-1,1) {#5$#1$};
    \node at (1,1) {#5$#2$};
    \node at (1,-1) {#5$#3$};
    \node at (-1,-1) {#5$#4$};
  \end{scope}}

\newcommand\bigllw[4]{\begin{tikzpicture}[scale=0.7,baseline=-0.5ex]
    \tikzllw{#1}{#2}{#3}{#4}{\relax}
  \end{tikzpicture}}

\newcommand\smallllw[4]{\begin{tikzpicture}[scale=0.35,baseline=-2]
    \tikzllw{#1}{#2}{#3}{#4}{\tiny}
  \end{tikzpicture}}

\newcommand\tab{\bigllw{a}{b}{a}{d}}
\newcommand\tbb{\bigllw{t}{b}{s}{b}}
\newcommand\tsb{\bigllw{s}{o}{s}{o}}
\newcommand\tdb{\bigllw{s}{d}{r}{d}}
\newcommand\ttb{\bigllw{t}{o}{t}{o}}
\newcommand\trb{\bigllw{r}{o}{r}{o}}

\newcommand\tas{\smallllw{a}{b}{a}{d}}
\newcommand\tbs{\smallllw{t}{b}{s}{b}}
\newcommand\tss{\smallllw{s}{o}{s}{o}}
\newcommand\tds{\smallllw{s}{d}{r}{d}}
\newcommand\tts{\smallllw{t}{o}{t}{o}}
\newcommand\trs{\smallllw{r}{o}{r}{o}}

\subsection{Dominos on \boldmath $L$}

We now finally begin tiling the lamplighter group. As our application is to the seeded tiling problem, as shown in Proposition~\ref{prop:ssft=svsft} instead of the sunny-side-up labelling we can simply work with allowed patterns, and specify a different tiling rule near the origin. In practice, we work with tilings of $L$, and in the end pose an additional restriction at the origin, in fact only at the identity element.

We have written the theory of tilings using the DHSs, and in the previous section we explained how to convert between these and the SFTs in the sense of Definition~\ref{def:SFT}, which we think of as the most general setting. In the case of the lamplighter group, we introduce two more presentations specific to this group: tetrahedron tilings, and Wang tilings for a particular generating set. Tetrahedron tilings are a special case of SFTs, using the (non-symmetric) generating set $\{e,a,b,ab^{-1}\}$, and they are the most convenient way to present the rules for the sea level construction in Section~\ref{ss:sea}. Wang tiles can be seen as a special case of DHS, and they are the most convenient way to present the rules for the comb construction in Section~\ref{ss:comb}.

For Wang tiles we always use the generating set $L=\langle a^{\pm1},b^{\pm1}\rangle$. A set of \emph{Wang tiles} consists of a finite set $A$ of edge colours, and a subset $\Pi\subseteq A^{\{a^{\pm1},b^{\pm1}\}}$. Vertices are coloured by elements of $\Pi$, and we check that the $A$-colourings match. We visualize each pattern $\pi\in\Pi$ as a diamond:
\[\pi=\bigllw pqrs\text{ means }\pi(a)=p,\pi(b)=q,\pi(a^{-1})=r,\pi(b^{-1})=s.\]
The resulting SFT is a subset $\Omega\subseteq\Pi^L$, namely
\[\Omega=\{\eta\in\Pi^L:\forall g\in L:\eta(g)(a)=\eta(g a)(a^{-1}),\;\eta(g)(b)=\eta(g b)(b^{-1})\}.\]
This can be seen as a special case of the $\Hom$-presentation.

Equivalently, we may specify a colour on every vertex of $\gfL$, and
impose constraints on the edges, or on small subgraphs. We found it
most convenient to impose constraints on small, height-$1$ tetrahedra,
as follows. We fix a finite set $A$ of vertex colours, and a subset
$\Theta\subseteq A^{\{1,a b^{-1},a,b\}}\cong A^4$.
The resulting \emph{tetrahedron tiling system} is a subset $\Omega\subseteq A^L$, namely
\[\Omega=\{\eta\in A^L:\forall g\in L:(\eta(g),\eta(g a b^{-1}),\eta(g a),\eta(g b))\in\Theta\}.
\]
Note that we may, and do, always assume that $\Theta$ is invariant under
the permutation $(1,2)(3,4)$ of its coordinates, because the condition applied at
$g a b^{-1}$ is precisely
$(\eta(g a b^{-1}),\eta(g),\eta(g b),\eta(g a))\in\Theta$.
For $\theta \in \Theta$ we denote this by $\bar \theta$.

We note, even though it is irrelevant to our construction, that the
``tetrahedra graph'' of $\gfL$, namely the graph whose vertices are
height-$1$ tetrahedra, and whose edges connect tetrahedra that share a
common vertex, is isomorphic to $\gfL$ (it corresponds to the
index-$2$ subgroup $\langle a,b^2a^{-1}\rangle$ of $L$, which is
isomorphic to $L$). We may thus equivalently label vertices or
tetrahedra by the given tiles.

We can convert between Wang tiles and tetrahedron tilings essentially
by the construction of the previous section, and we give the specialized formulas.

It is easy to convert a set of Wang tiles into a set of tetrahedron
tiles: assume $\Omega$ is given by the Wang tileset
$\Pi\subseteq A^{\{a^{\pm1},b^{\pm1}\}}$; then $\Omega\subseteq\Pi^L$
is given by the tetrahedra constraints
\[\Theta=\left\{(\alpha,\beta,\gamma,\delta)\in\Pi^{\{1,a b^{-1},a,b\}}:\begin{alignedat}{2}\alpha(a)&=\gamma(a^{-1}),&\beta(a)&=\delta(a^{-1}),\\\alpha(b)&=\delta(b^{-1}),&\beta(b)&=\gamma(b^{-1})\end{alignedat}\right\}.\]

Conversely, let $\Theta\subseteq A^{\{1,a b^{-1},a,b\}}$ be a collection
of tetrahedron tiles. The edge colours will be simply $C=\Theta$.
It is easy to see that tilings of the Wang tile set
\[\Pi=\left\{(\theta,\theta,\eta,\bar{\eta}):\theta_1=\eta_3\;[\in A]\right\}.\]
are in one-to-one corresponence (topological conjugacy) to tilings by $\Theta$.

After this section, we will mostly take a more relaxed approach with
terminology: ``SFT'' can refer to DHS, to SFT in the sense of
Definition~\ref{def:SFT}, or to one of the subclasses from this
section. This should always be clear from context, and we have given
the formulas for translating between these formalism in
Section~\ref{ss:cayley} and in the present section.

\section{The comb}\label{ss:comb}
We construct in $\gfL$ a geometric structure resembling a ``comb'': it
is an SFT $\Omega_c$ marking a bi-infinite line, the \emph{spine} of
the comb; rays exiting upwards and downwards from the spine, its
\emph{teeth} and \emph{antiteeth}; and some extra synchronizing
signals. We then show how, when coloured by a comb, the graph $\gfL$
simulates the plane $\Z^2$. This comb is defined by the following set
$\Pi_c$ of Wang tiles:
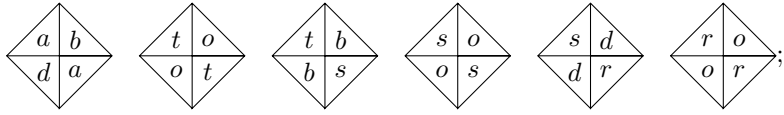
\[\tab\quad \ttb\quad \tbb\quad \tsb\quad \tdb\quad \trb;\]
as a graph SFT, it is $\Hom(\gfL,\gfF)$ for the following graph
$\gfF$, with generators $a$ in green and $b$ in two-headed red:
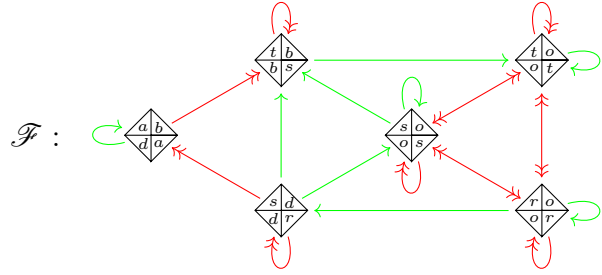
\[\begin{tikzpicture}[tile/.style={diamond,inner sep=-4pt}]
    \node at (-5,0) {\Large $\gfF:$};
    \node[tile] (ta) at (180:2*1.732) {$\tas$};
    \node[tile] (tb) at (150:2) {$\tbs$};
    \node[tile] (ts) at (0,0) {$\tss$};
    \node[tile] (td) at (210:2) {$\tds$};
    \node[tile] (tr) at (-30:2) {$\trs$};
    \node[tile] (tt) at (30:2) {$\tts$};

    \draw[green,->] (ta) edge[loop left] ()
    (tr) edge[loop right] () edge (td)
    (td) edge (tb) edge (ts)
    (ts) edge[loop above] () edge (tb)
    (tb) edge (tt)
    (tt) edge[loop right] ();
    \draw[red,->>] (td) edge (ta)
    (ta) edge (tb)
    (tt) edge[<<->>] (tr)
    (tr) edge[<<->>] (ts) to[loop below,->>] ()
    (ts) edge[<<->>] (tt);
    \draw[red,->>] (td) to[loop below] ();
    \draw[red,->>] (ts) to[loop below] ();
    \draw[red,->>] (tt) to[loop above] ();
    \draw[red,->>] (tb) to[loop above] ();
  \end{tikzpicture}
\]

Recall our notation for Wang tiles: the edge colour $a$ appears only
on the first tile $\tas$, and implies that each time a vertex $g\in L$
carries the tile $\tas$, its edges $(g,g a)$ and $(g,g a^{-1})$ carry
the colour $a$, and thus the whole coset $g\langle a\rangle$ carries
the tile $\tas$; this is the spine of the comb. Likewise, a ray of
$b$'s exits in the $b$ direction --- the teeth of the comb, and a ray
of $d$'s (thought of as reflected $b$'s) exists in the $b^{-1}$
direction --- the antiteeth. The tiles $\tbs$ and $\tds$ share a
signal $s$ which propagates via $\tss$, and are extended on their
other ends by respective signals $r$ and $t$. The remaining edges $a$
and $b$ edges may (but are not forced to) respectively be labelled $t$
and $o$. Let us write
\begin{equation}\label{eq:Z}
  Z = \{a^n b^m a^k : k,m,n\in\Z\},
\end{equation}
a union of $\langle a\rangle$-cosets.

\begin{lem}\label{lem:comb}
  If $\eta \in (\Pi_c)^L$ is a valid Wang tiling and $\eta_1 = \tas$, then
  \begin{align*}
    \forall n \in \Z:&\quad \eta_{a^n} = \tas,\\
    \forall n \in \Z, m \geq 1:&\quad \eta_{a^n b^m} = \tbs,\\
    \forall n \in \Z, m \geq 1:&\quad \eta_{a^n b^{-m}} = \tds,\\
    \forall n \in \Z, 1 \leq k < m:&\quad \eta_{a^n b^m a^{-k}} = \tss,\\
    \forall n \in \Z, 1 \leq k, m:&\quad \eta_{a^n b^{-m} a^{-k}} = \trs,\\
    \forall n \in \Z, 1 \leq k, m:&\quad \eta_{a^n b^m a^k} = \tts.
  \end{align*}
  Conversely, these formulas, together with $\eta_g = \tts$ whenever
  $g\notin Z$, define a valid Wang tiling of $L$.
\end{lem}
\begin{proof}
  Suppose that $\eta \in (\Pi_c)^L$ is a valid Wang tiling with
  $\eta_{e} = \tas$. We prove that $\eta$ satisfies the formulas by a
  series of Sudoku-style deductions, keeping track of possible values
  of cells.  Since the colour $a$ only appears in the tile $\tas$, we
  have
  \[ \forall n \in \Z: \eta_{a^n} = \tas.\]
  Since the $b$-colour of $\tas$ is $b$ and $b$ only appears as the
  $b^{-1}$-colour in the tile $\tbs$, we then have
  \[ \forall n \in \Z, m \geq 1: \eta_{a^n b^m} = \tbs, \]
  and by the same argument for $d$
  \[ \forall n \in \Z, m \geq 1: \eta_{a^n b^{-m}} = \tds. \]
  Since the $a^{-1}$-colour of $\tbs$ is $s$, we next have
  \[ \forall n, m \geq 1: \eta_{a^n b^m a^{-1}} \in \{\tss, \tds\}, \]
  and since the only tiles with $a$-colour in $\{r,s\}$ are
  $\{\tss, \tds, \trs\}$, and they have their $a^{-1}$-colours in
  $\{r,s\}$, we must have
  \[ \forall n, m \geq 1, k \geq 1: \eta_{a^n b^m a^{-k}} \in \{\tss,\tds,\trs\}.
  \]
  Since $\eta_{a^{n+m} b^{-m}} = \tds$ and
  $a^n b^m a^{-m} = a^{n+m} b^{-m}$, the $a$-colour of
  $\eta_{a^n b^m a^{-m}}$ is $s$, so the $a$-colour of
  $\eta_{a^n b^m a^{-k}}$ is $s$ for all $1 \leq k < m$ and thus
  $\eta_{a^n b^m a^{-k}}=\tss$ for all $1 \leq k < m$. Finally the
  $a$-colour $t$ of $\eta_{a^n b^m}$ forces $\eta_{a^n b^m a^k}=\tts$
  for all $k\ge1$, and the $a^{-1}$-colour $r$ of $\eta_{a^n b^{-m}}$
  forces $\eta_{a^n b^{-m} a^{-k}}=\trs$ for all $k\ge1$.

  Next, we show that the formulas of the lemma, together with
  $\eta_g = \tts$ for all $g \notin Z$, indeed define a valid
  tiling. This means, first, that the point $\eta$ given by the
  formulas is well-defined, i.e.\ exactly one value is given to each
  coordinate --- for this, it suffices to check that the sets
  $\{a^n : n \in \Z\}$, $\{a^n b^m : n \in \Z, m \geq 1\}$,
  $\{a^n b^{-m} : n \in \Z, m \geq 1\}$,
  $\{a^n b^m a^{-k} : n \in \Z, 1 \leq k < m \}$,
  $\{a^n b^m a^k : n \in \Z, 1 \leq k, m \}$ and
  $\{a^n b^{-m} a^{-k} : n \in \Z, 1 \leq k, m \}$, are disjoint. This
  also means that the neighbouring colours are correct whenever two
  tiles from these sets are adjacent, and that in every other
  $b$-direction, their colour is $o$ so that they match with the tile
  $\tts$ used to fill all other cells.

  To check these things, we consider the action of the lamplighter
  group on $\{0,1\}^\ZO\times\Z$, and the orbit of the configuration
  with the head at the origin and all edges with colour $o$ or
  $t$. The set $Z$ defined in~\eqref{eq:Z} corresponds to the
  configurations with a single run of $1$s (anywhere). Let us analyze
  the sets listed above, whose union is $Z$.

  The set $Z_1 = \{a^n : n \in \Z\}$ is the set of configurations
  where no $1$s have been written. The set
  $Z_2 = \{a^n b^m : n \in \Z, m \geq 1\}$ is the set where there is a
  run of $1$s and the head is exactly on the right border of the
  run. The set $Z_3 = \{a^n b^{-m} : n \in \Z, b \geq 1\}$ is the set
  of configurations where there is a run of $1$s and the head is
  exactly on the left border of the run. The set
  $Z_4 = \{a^n b^m a^{-k} : n \in \Z, m \geq 1, 1 \leq k < m\}$ is the
  set of configurations where there is a run of $1$s and the head is
  properly inside it. The set
  $Z_5 = \{a^n b^m a^k : n \in \Z, k,m \geq 1\}$ is the set of
  configurations where there is a run of $1$s and the head is strictly
  to its right. The set
  $Z_6 = \{a^n b^{-m} a^{-k} : n \in \Z, k,m \geq 1\}$ is the set of
  configurations where there is a run of $1$s and the head is strictly
  to its left.  Obviously these sets are disjoint. (The interpretation
  of the formula $a^n b^m a^{-m} = a^{n+m} b^{-m}$ used in the proof,
  in terms of this action, is that the head has either written a run
  of $1$s moving to the right and returned left, or has moved to the
  right and written a run of $1$s while returning.)

  Now, let us analyze neighbours of tiles in $Z$. It is straightforward
  to verify by a case analysis that whenever a tile in $Z$ has a
  non-$o$ edge colour in some direction, the corresponding neighbour is
  also in $Z$ and has the same colour in the opposite direction.

  We now study other edges, which must be in direction $b$ since $Z$
  is a union of $\langle a\rangle$-cosets.  The tile $\tss$ appears in
  $Z_4$ when the head is properly inside a run of $1$s. The $b$- and
  $b^{-1}$-edges have colour $o$. These neighbours are not in $Z$: in
  the action the head either leaves a nonempty run of $1$s to the left
  with a $0$ in between, or creates two runs of $1$s. The tiles $\tts$
  and $\trs$ in $Z_5$ and $Z_6$ are at positions at which the head is
  out of the run of $1$s, so their $b$- and $b^{-1}$-neighbours have
  two runs of $1$s and are also not in $Z$.
\end{proof}

Note that the Sudoku part of the argument could also be done ``bottom
up'' i.e. starting from $\eta_{a^{n+m} b^{-m}}$ towards
$\eta_{a^n b^m}$, in which case one would propagate the set
$\{\tbs, \tss, \trs\}$ instead.

\subsection{Simulating \boldmath $\Z^2$ on $\gfL$ marked by $\Omega_c$}
We consider the Cayley graph $\gfL$ of the lamplighter group as a
labelled graph, the labels given by the ``comb'' SFT $\Omega_c$ described
in the previous section. Our aim is to prove here
\begin{prop}\label{prop:combsimul}
  The graph $\gfL$, when labelled by any configuration from the SFT
  $\Omega_c$ with tile $\tas$ at the origin, simulates a graph
  containing the plane $\Z^2$, and the special configuration defined
  in Lemma~\ref{lem:comb} simulates precisely the plane. The same
  simulator can be used for each configuration.
\end{prop}
\begin{proof}
   Recall that the graph $\beta : \Z^2 \to \gfB$ has
  edge labels $E(\gfB) = \{\uparrow,\downarrow,\rightarrow,\leftarrow\}$, and vertex
  labels are trivial,  $V(\gfB) = \{\cdot\}$. The following figure
  gives the simulator. The split nodes have a vertex label from $\gfL$ and the trivial vertex label
  $\cdot$ from $\Z^2$, while the unsplit nodes correspond to vertices
  of $\gfB^*\setminus\gfB$ for the graph $\gfB$ defining the marking
  of $\Z^2$ (whose vertex label can be deduced from the edge labels):
  
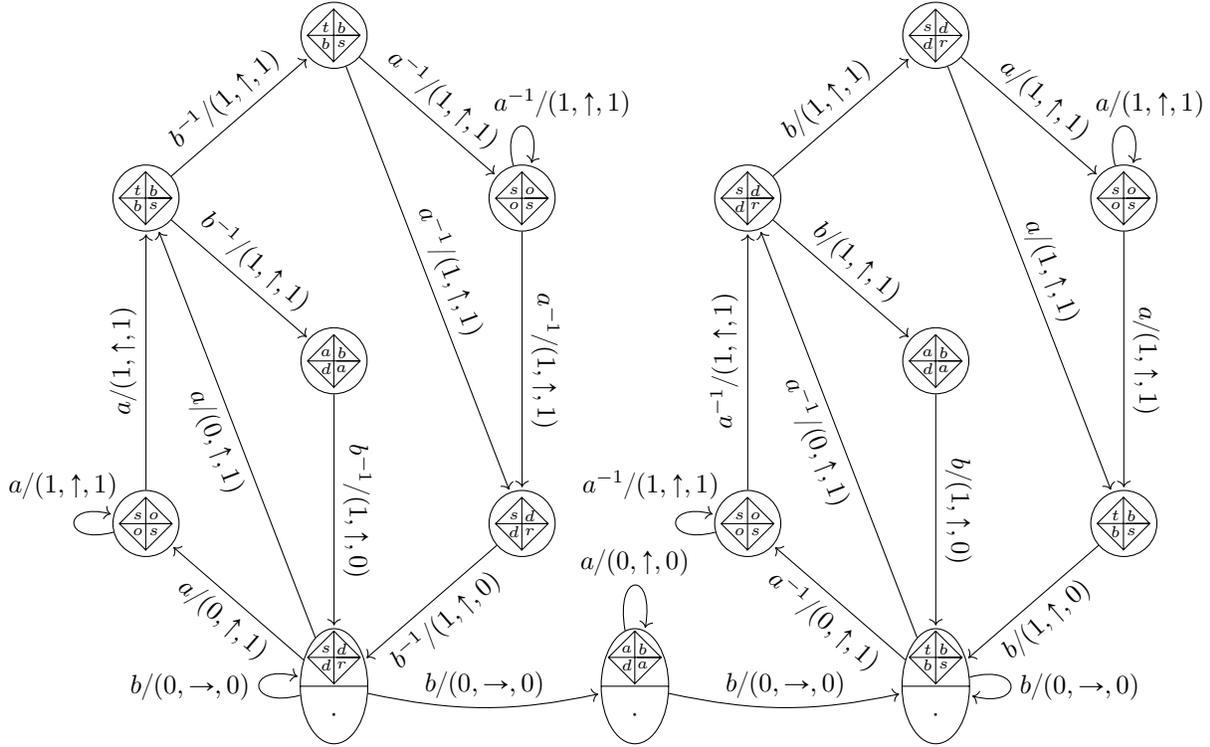
\[\hspace{-2cm}\begin{tikzpicture}[shorten >=1pt,node distance=1cm,auto,every state/.style={circle,inner sep=-6mm},x={(1,0)},y={(0,1)},z={(1,0.866)}]
      \ellipsesplitnode{a}{(0,0,0)}{4mm}{3mm}{$\tas$}{$\cdot$}
      \ellipsesplitnode{b}{(4,0,0)}{4mm}{3mm}{$\tbs$}{$\cdot$}
      \node[state] (c) at (-1,0,5) {$\tas$};
      \node[state] (d) at (-6,0,7.5) {$\tds$};
      \node[state] (e) at (-1,0,2.5) {$\tss$};
      \node[state] (f) at (-6,0,10) {$\tds$};
      \node[state] (g) at (-1,0,7.5) {$\tss$};
      \node[state] (h) at (4,0,2.5) {$\tbs$};
      \ellipsesplitnode{B}{(-4,0,0)}{4mm}{3mm}{$\tds$}{$\cdot$}
      \node[state] (C) at (-9,0,5) {$\tas$};
      \node[state] (D) at (-14,0,7.5) {$\tbs$};
      \node[state] (E) at (-9,0,2.5) {$\tss$};
      \node[state] (F) at (-14,0,10) {$\tbs$};
      \node[state] (G) at (-9,0,7.5) {$\tss$};
      \node[state] (H) at (-4,0,2.5) {$\tds$};
      \path[->]
    (a) edge[bend right=12] node[above] {$b/(0,{\rightarrow},0)$} (b)
        edge [out=100,in=80,looseness=8] node {$a/(0,{\uparrow},0)$} (a)
    (B) edge[bend right=12] node[above] {$b/(0,{\rightarrow},0)$} (a)
        edge [out=195,in=165,looseness=8] node {$b/(0,{\rightarrow},0)$} (B)
    (b) edge [out=15,in=-15,looseness=8] node {$b/(0,{\rightarrow},0)$} (b)
    (b) edge node[sloped,below] {$a^{-1}/(0,{\uparrow},1)$} (d)
        edge node[sloped,below] {$a^{-1}/(0,{\uparrow},1)$} (e)
    (B) edge node[sloped,below] {$a/(0,{\uparrow},1)$} (D)
        edge node[sloped,below] {$a/(0,{\uparrow},1)$} (E)
    (c) edge node[sloped,above] {$b/(1,{\uparrow},0)$} (b)
    (C) edge node[sloped,above] {$b^{-1}/(1,{\uparrow},0)$} (B)
    (d) edge node[sloped,above] {$b/(1,{\uparrow},1)$} (c)
        edge node[sloped,above] {$b/(1,{\uparrow},1)$} (f)
    (D) edge node[sloped,above] {$b^{-1}/(1,{\uparrow},1)$} (C)
        edge node[sloped,above] {$b^{-1}/(1,{\uparrow},1)$} (F)
    (e) edge node[sloped,above] {$a^{-1}/(1,{\uparrow},1)$} (d)
        edge [loop left] node[above left=2mm and -7mm] {$a^{-1}/(1,{\uparrow},1)$} ()
    (E) edge node[sloped,above] {$a/(1,{\uparrow},1)$} (D)
        edge [loop left] node[above left=2mm and -7mm] {$a/(1,{\uparrow},1)$} ()
    (f) edge node[sloped,above] {$a/(1,{\uparrow},1)$} (g)
        edge node[sloped,above] {$a/(1,{\uparrow},1)$} (h)
    (F) edge node[sloped,above] {$a^{-1}/(1,{\uparrow},1)$} (G)
        edge node[sloped,above] {$a^{-1}/(1,{\uparrow},1)$} (H)
    (g) edge [loop above] node[above right=0mm and -5mm] {$a/(1,{\uparrow},1)$} ()
        edge node[sloped,above] {$a/(1,{\uparrow},1)$} (h)
    (G) edge [loop above] node[above right=0mm and -5mm] {$a^{-1}/(1,{\uparrow},1)$} ()
        edge node[sloped,above] {$a^{-1}/(1,{\uparrow},1)$} (H)
    (h) edge node[sloped,below] {$b/(1,{\uparrow},0)$} (b)
    (H) edge node[sloped,below] {$b^{-1}/(1,{\uparrow},0)$} (B);
  \end{tikzpicture}
  \]
  
  The automaton is undirected: We only write half the edges in the
  diagram above, those whose label on the $\gfB^*$-side is $\uparrow$
  or $\rightarrow$; the missing $\downarrow$ and $\leftarrow$ edges
  are naturally recovered by applying the involutions of $\gfL$ and
  $\Z^2$, namely inverting the generators $a,b$ and reversing the
  direction of the arrows, switching the source and range tags.

  To prove that this simulator indeed produces $\Z^2$, let us explain
  what it does in terms of graph-walking automata, see
  Lemma~\ref{lem:gwa}. First, we recall that $\Z^2$ is simulated,
  within $\gfL$, as those vertices on the spine, teeth or antiteeth of
  the comb; namely, those vertices marked $\tas$, $\tbs$ or $\tds$. The
  bijection associates to $(m,n)\in\Z^2$ the element $a^n b^m$.

  The graph walking automaton for the generator `$\rightarrow$' of
  $\Z^2$ is simply ``if you're on the spine of the comb, move onto a
  tooth; if you're on a tooth, move further on the tooth; if you're on
  an antitooth, move towards the spine''. This is realized by the four
  arrows marked $b/(0,\rightarrow,0)$.

  The generator `$\uparrow$' is programmed as follows: ``if you're on
  the spine, move up the spine. If you're on a tooth, follow the `$s$'
  signal (using generator $a^{-1}$) till you reach an antitooth. Then
  do a step on that antitooth (using generator $b$), follow the '$s$'
  signal back up to a tooth (using generator $a$) and finally do a
  step on that tooth. If you're on antitooth, do the same, except you
  start by following the `$s$' signal using generator $a$ till you
  reach a tooth.''

  On the one hand, it is easy to see that this is what the above
  simulator does, following the big hexagon-shaped counterclockwise paths; on
  the other hand, let us convince ourselves that these operations
  indeed implement movement on $\Z^2$.

  The operation $\uparrow$ is $n\mapsto n+1$ on $\{(m,n)\in\Z^2\}$.
  The spine is the subset $Z_1=\{a^n:n\in\Z\}$ and corresponds to
  $\{0\}\times\Z$; so `$\uparrow$' is simply implemented by following
  the generator $a$. Consider now a point $(m,n)\in\Z^2$ with $m>1$,
  appearing on a tooth as $a^n b^m\in Z_2$. Then the path along the
  hexagon leads us successively (using $m$ times $a^{-1}$ along $Z_4$,
  following the `$s$' signal) to
  $a^n b^m a^{-m}=a^{n+m}b^{-m}\in Z_3$, then to $a^{n+m}b^{-(m-1)}\in Z_3 \cup Z_1$
  (we reach the spine $Z_1$ if $m = 1$), then (using $(m-1)$ times $a$ along
  $Z_4$ following the `$s$' signal) to $a^{n+m}b^{-(m-1)}a^{m-1}=a^{n+1}b^{m-1}$,
  and finally to $a^{n+1}b^m$. This corresponds to the point $(m,n+1)\in\Z^2$ as
  required. The same argument applies to the antitooth.
\end{proof}

\noindent This gives the first proof of Theorem~\ref{thm:main}:
\begin{cor}\label{cor:firstproof}
  The seeded tiling problem on the lamplighter group is undecidable.
\end{cor}

\begin{proof}
  We apply Theorem~\ref{thm:simulfamily} with $\Gamma$ the family of
  all graphs obtained from $\gfL$ by labelling it by any configuration
  from the SFT $\Omega_c$ with tile $\tas$ at the origin, and
  $\Delta = \{\Z^2\}$.  Let $\gfS$ be the simulator constructed in the
  previous lemma.

  The first condition of Theorem~\ref{thm:simulfamily},
  \begin{enumerate}
  \item every graph in $\Delta$ may be simulated: for every $\gfH\in\Delta$
  there is $\gfG\in\Gamma$ such that $\gfG\rtimes\gfS$ is \'etale and weakly
  maps to $\gfH$,
  \end{enumerate}
  holds, because the unique element $\Z^2 \in \Delta$ is even
  isomorphic to $\gfG \rtimes \gfS$, for $\gfG$ the special
  configuration defined in Lemma~\ref{lem:comb}; since $\Z^2$ is
  \'etale, so is the simulating graph.
  
The second condition of Theorem~\ref{thm:simulfamily},
  \begin{enumerate}
  \setcounter{enumi}{1}
  \item simulated graphs are images of $\Delta$: for every
    $\gfG\in\Gamma$ there is $\gfH\in\Delta$ that weakly maps to
    $\gfG\rtimes\gfS$,
  \end{enumerate}
  holds, because every $\gfG \in \Gamma$ contains the configuration
  simulating $\Z^2$ as a subgraph, so we even have
  $\Z^2 \subset \gfG \rtimes \gfS$.
\end{proof}

The simulator given above is simple enough that it may be directly
translated to a tileset on $\gfL$, and we do so here, incorporating some
ad hoc simplifications.

We construct, from a seeded Wang tileset $T$ on a
half-plane, a seeded Wang tileset $\Pi_T$ on $L$ which tiles if and
only if $T$ tiles. The construction could easily be extended to the
whole plane, at the cost of extra clutter.  Let us consider the
half-plane
\[\mathbb H=\{(m,n)\in\Z^2: m \ge n\}.
\]
For a colour set $C$, a Wang tileset is a subset
$T\subseteq C^{\{S,E,N,W\}}=C^4$ interpreted as follows:
$(i,j,k,\ell)\in T$ is a square with colours $i,j,k,\ell$ respectively
on the south, east, north, west sides. A valid tiling of the
half-plane $\mathbb H$ is an assignment $\zeta\colon\mathbb H\to T$
with $(\zeta_{(m,n-1)})_N=(\zeta_{(m,n)})_S$ and
$(\zeta_{(m,n)})_E=(\zeta_{(m+1,n)})_W$ for all $(m,n)\in\mathbb H$.

It is easy to show that the (un)decidability of the seeded tiling
problem on this rotated half-plane is equivalent to that on the
standard half-plane. The use of this rotated half-plane $\mathbb H$
simplifies somewhat the construction; we use the embedding of
$\mathbb H$ into $\gfL$ given by $(m,n)\mapsto a^n b^{m-n}$.

\begin{prop}
  Let $C$ be a finite set of colours, let
  $T\subseteq C^{\{S,E,N,W\}}=C^4$ be a Wang tileset for $\mathbb H$,
  and let $t_0\in T$ be a seed tile. Then a tileset $\Pi_T$ on $L$ and
  a seed $\pi_0\in\Pi_T$ may be algorithmically constructed, such that
  there exists a valid tiling of $\mathbb H$ by $T$ with value $t_0$
  at $(0,0)$ if and only if there exists a valid tiling of $L$ by
  $\Pi_T$ with value $\pi_0$ at $1$.
\end{prop}
\begin{proof}
  The tileset $\Pi_T$ will be given by product tiles, with in the
  first layer a tile from $\Pi_c$ and in the second layer a word of
  length $\in\{0,1,2\}$ over $C$. We denote by $\varepsilon$ the empty
  word (which we of course assume disjoint from $C$).

  For all tiles $(i,j,k,\ell) \in T$ we take the following product
  tiles in $\Pi_T$:
  \begin{gather*}
    \Bigg(\tab,\bigllw{\varepsilon}{j}{\varepsilon}{i}\Bigg), \qquad
    \Bigg(\ttb,\bigllw{\varepsilon}{\varepsilon}{\varepsilon}{\varepsilon}\Bigg),\qquad
    \Bigg(\trb,\bigllw{\varepsilon}{\varepsilon}{\varepsilon}{\varepsilon}\Bigg),\\
    \Bigg(\tbb,\bigllw{\varepsilon}{j}{ik}{\ell}\Bigg),\qquad
    \Bigg(\tsb,\bigllw{ik}{\varepsilon}{ik}{\varepsilon}\Bigg),\qquad
    \Bigg(\tdb,\bigllw{ik}{k}{\varepsilon}{i}\Bigg).
  \end{gather*}
  The seed tile is
  $\pi_0 = (\tas, \smallllw{\varepsilon}{j}{\varepsilon}{i})$ where
  $t_0 = (i,j,k,\ell)$.

  First, we prove that if there exists a valid tiling
  $\eta \in (\Pi_T)^L$ with $\eta_1 = \pi_0$ then $T$ admits a valid
  tiling with seed $t_0$. To see this, observe that we must have a
  valid instance of $\Pi_c$ on the first layer, so we can define the
  sets $Z=Z_1\sqcup Z_2\sqcup Z_3\sqcup Z_4\sqcup Z_5\sqcup Z_6$ as in
  the proof of Lemma~\ref{lem:comb} and when $g\in Z_i$ we have
  $\eta_g=(X_i,t)$ for some $t$, with the correspondence
  \[X_1=\tas,\quad X_2=\tbs,\quad X_3=\tds,\quad X_4=\tss,\quad X_5=\tts,\quad X_6=\trs.
  \]
  From $\eta$ we deduce a configuration $\zeta\in T^{\mathbb
    H}$ as follows: if
  $\eta_{a^n b^m} = (\tbs, \smallllw{\varepsilon}{j}{i
    k}{\ell})$ for $n \in \Z, m \geq 1$ then define
  $\zeta_{(m+n,n)} = (i,j,k,\ell)$, and if
  $\eta_{a^n} = (\tas,
  \smallllw{\varepsilon}{j}{\varepsilon}{i})$ then define
  $\zeta_{(n,n)} = (i,j,k,\ell)$ for any
  $k,\ell$ such that $(i,j,k,\ell) \in T$, and set
  $\zeta_{(0,0)} = t_0$.

  By construction we have $\zeta_{(0,0)} = t_0$, and we need to check
  $(\zeta_{(m,n)})_E = (\zeta_{(m+1,n)})_W$ and
  $(\zeta_{(m,n-1)})_N = (\zeta_{(m,n)})_S$.

  We have $(\zeta_{(m,n)})_E = (\zeta_{(m+1,n)})_W$ directly from the
  colouring rules of $\Pi_T$, since $(\zeta_{(m,n)})_E$ is the
  $b$-colour on the second layer of the tile at $\eta_{a^n b^{m-n}}$
  and $(\zeta_{(m+1,n)})_W$ is the $b^{-1}$-colour on the second layer
  of the tile at $\eta_{a^n b^{m+1-n}}$.

  We now check the formula $(\zeta_{(m,n-1)})_N = (\zeta_{(m,n)})_S$.
  If $m>n$ and $\zeta_{(m,n)} = (i,j,k,\ell)$ then the second layer of
  $\eta_{a^n b^{m-n}}$ is $\smallllw{\varepsilon}{j}{i k}{\ell}$ so
  its $a^{-1}$-colour is $i k$. It follows that in all the tiles
  $\eta_{a^n b^{m-n} a^{-p}}$ with $1 \leq p < m-n$ (namely, at
  positions in $Z_4$, where the first layer is $\tss$) the second
  layer contains the tile
  $\smallllw{i k}{\varepsilon}{i k}{\varepsilon}$. Therefore at
  $\eta_{a^n b^{m-n} a^{n-m}} = \eta_{a^m b^{n-m}}$ (at a position in
  $Z_3$, where the first layer is $\tds$), by the choice of tiles
  overlayed on $\tds$, the tile on the second layer must precisely be
  $\smallllw{i k}{k}{\varepsilon}{i}$. In this manner, we have proven
  the following property $(\star_{m,n})$ for $m>n$:
  \begin{align*}
    &\text{The $b$-colour of the second layer of $\eta_{a^m b^{n-m}}$ is $(\zeta_{(m,n)})_N$.}\\
    &\text{The $b^{-1}$-colour of the second layer of $\eta_{a^m b^{n-m}}$ is $(\zeta_{(m,n)})_S$.}
  \end{align*}

  By $(\star_{m,n-1})$ the $b$-colour at $\eta_{a^m b^{n-1-m}}$ is
  $(\zeta_{(m,n-1)})_N$. If $m>n$ then by $(\star_{m,n})$ the
  $b^{-1}$-colour at $\eta_{a^m b^{n-m}}$ is $(\zeta_{(m,n)})_S$;
  since $a^m b^{n-m}$ is the $b$-neighbour of $a^m b^{n-1-m}$, we have
  $(\zeta_{(m,n-1)})_N = (\zeta_{(m,n)})_S$ as required. If $m=n$ then
  the $b^{-1}$-colour at $\eta_{a^m}$ is $(\zeta_{(n,n)})_S$ and again
  we have $(\zeta_{(n,n-1)})_N = (\zeta_{(n,n)})_S$.  We have proven
  that a valid $\pi_0$-seeded tiling for $\Pi_T$ yields a valid
  $t_0$-seeded tiling of $\mathbb H$.
  
  Conversely, if there is a valid tiling $\zeta\in T^{\mathbb H}$ with
  $\zeta_{(0,0)} = t_0$, then the above proof shows rather directly
  how to construct a valid configuration $\eta \in (\Pi_T)^L$ with
  $\eta_1=\pi_0$: set
  \begin{align*}
    \forall n \in \Z: \zeta_{(n,n)} = (i,j,k,\ell) \implies& \eta_{a^n} = \Big(\tas,\smallllw{\varepsilon}{j}{\varepsilon}{i}\Big),\\
    \forall m>n \in \Z: \zeta_{(m,n)} = (i,j,k,\ell) \implies& \eta_{a^n b^{m-n}} = \Big(\tbs,\smallllw{\varepsilon}{j}{i k}{\ell}\Big),\\
    \forall m>n\in \Z: \zeta_{(m,n)} = (i,j,k,\ell) \implies& \eta_{a^m b^{n-m}} = \Big(\tds,\smallllw{i k}{k}{\varepsilon}{i}\Big),\\
    \forall m>p>n \in \Z: \zeta_{(m,n)} = (i,j,k,\ell) \implies& \eta_{a^n b^{m-n} a^{p-m}} = \Big(\tss,\smallllw{i k}{\varepsilon}{i k}{\varepsilon}\Big),\\
    \forall p>m>n \in \Z: \quad  & \eta_{a^n b^{m-n} a^{p-m}} = \Big(\tts,\smallllw{\varepsilon}{\varepsilon}{\varepsilon}{\varepsilon}\Big),\\
    \forall m>n>p \in \Z: \quad & \eta_{a^n b^{m-n} a^{p-m}} = \Big(\trs,\smallllw{\varepsilon}{\varepsilon}{\varepsilon}{\varepsilon}\Big),\\
    \forall g \notin Z:\quad & \eta_g = \Big(\tts,\smallllw{\varepsilon}{\varepsilon}{\varepsilon}{\varepsilon}\Big).\qedhere
  \end{align*}
\end{proof}

\section{The sea level}\label{ss:sea}

In this section we show, again, that $\gfL$ can simulate the plane 
$\Z^2$. We do this in steps: (1) using an SFT, we can mark trees in
$\gfL$, or more precisely subgraphs of the form
$\mathscr T\times_\Z\Z$ and of the form $\Z\times_\Z\mathscr T$ in the
pullback description of $\gfL$. They correspond to limits of faces of
tetrahedra. (2) using another SFT, we can mark vertices at height $0$
in $\gfL$. A vertex at height $0$ is naturally identified with a point
in the plane: from $u\markinout v$ we read $u,v$ as binary expansions
of integers $x,y$ respectively, and identify $u\markinout v$ with
$(x,y)$. (3) using a simulation, we show how the arithmetic operations
$(x,y)\mapsto(x\pm1,y)$ and $(x,y)\mapsto(x,y\pm1)$ can be described
in $\gfL$.

In fact, we first explain how the construction lets us simulate the
quadrant $\N^2$, and then explain which changes let us simulate a
whole plane. This last step is in fact unnecessary, since the quadrant
simulates the plane (Example~\ref{ex:simquadrant}) and simulation is
transitive (Lemma~\ref{lem:simultrans}).

In some sense, this simulation is more efficient than the ``comb''
from \S\ref{ss:comb}: the square $[1,2^n]^2\subset\Z^2$ is simulated
within a tetrahedron of height $2n$, and therefore of size
$(2n+1)\cdot 2^{2n}$.

\newcommand\true{\top}
\newcommand\false{\bot}
\subsection{Marking a ray \boldmath $\langle a\rangle$}
The SFTs in this section will be given by patterns on tetrahedra. We
choose as alphabet the Boolean algebra $\mathbb B=\{\false,\true\}$
meaning false and true, and define $\Pi_\leftarrow$ as the following
set of patterns:
\[\Pi_\leftarrow=\left\{(\alpha,\beta,\gamma,\delta)\in \mathbb B^4=\mathbb B^{\{1,a b^{-1},a,b\}}:
    \begin{aligned}
      \alpha\vee\beta&\Longrightarrow\gamma\wedge\delta\\
      \gamma\vee\delta&\Longrightarrow\alpha\neq\beta
    \end{aligned}\right\}.
\]

In the geometric model of the Cayley graph, that means that above
every $\true$ node both neighbours are $\true$, while below a $\true$
node precisely one of the neighbours is $\true$. We shall see that
$\Pi_\leftarrow$ forces a ray in $\gfL$ to be marked $\true$. In passing, we
observe another property of the subshift generated by $\Pi_\leftarrow$.

\begin{defn}
  Let $X$ be a compact metric space. A topological dynamical system
  $G \curvearrowright X$ is \emph{almost minimal} if there is a unique
  $G$-fixed point $0 \in X$, and
  $\forall x \neq 0: \overline{G x} = X$.
\end{defn}

We briefly recall and extend our notation for lamplighter group
elements that we introduced in~\S\ref{ss:ll}. Recall that the lamps in the lamplighter group are
at positions in the half-integers $\ZO = \Z + \frac12$. Every $g\in L$ may be written as $g=(s,n)$ with
$s\colon\ZO\to\Z/2$ and $n\in\Z$. We write $s_{>n}$ for the right
subword $s(n+1/2)s(n+3/2)\dots$ and $s_{<n}$ for the left subword
$\dots s(n-3/2)s(n-1/2)$.

There is a natural action of $L$ on $(\Z/2)^\ZO$, if we interpret
$(\Z/2)^\ZO$ as bi-infinite strings over $\{0,1\}$ with a $\markout$
marker at position $0$. The differences between descriptions of $L$
and $(\Z/2)^\ZO$ is that elements of $L$ also have a $\markin$ marker,
but on the other hand are almost everywhere $0$.

We shall write, here and throughout this section, $\Omega_\leftarrow$
for the subshift of $\mathbb B^L$ defined by $\Pi_\leftarrow$. (As the
astute reader may have guessed, there will soon be a $\Pi_\rightarrow$
and $\Omega_\rightarrow$.)

\begin{lem}\label{lem:XLCharacterization}
  For $t\in (\Z/2)^\ZO$ define $\eta(t)\in\mathbb B^L$ as follows:
  \[\text{for }g=(s,n)\in L: \eta(t)_g \iff s_{>n}=t_{>n}.
  \]
  Then the correspondence $t\mapsto\eta(t)$ is an $L$-equivariant,
  surjective map $(\Z/2)^\ZO\twoheadrightarrow\Omega_\leftarrow$. It
  collapses $\{t\in(\Z/2)^\ZO:t_{>0}\text{ is infinitely supported}\}$
  to the point $\false^L$, and is injective on its complement, so
  $t\mapsto\eta(t)$ presents $\Omega_\leftarrow$ as
  \[\Omega_\leftarrow=\{\eta(t):t\in(\Z/2)^\ZO\}=(\Z/2)^\ZO/(t\sim t'\text{ if both $t_{>0}$ and $t'_{>0}$ are infinitely supported}).\]
    The subshift $\Omega_\leftarrow$ is almost minimal.
\end{lem}

Note that the mapping $t\mapsto\eta(t)$ is not continuous for the
Cantor topology on $(\Z/2)^\ZO$, but it \emph{is} continuous when
$(\Z/2)^\ZO$ has the topology with basis
\[ \{y \in (\Z/2)^\ZO  : y_{>n} = x_{>n}\}\text{ for $n\in\Z$ and $x$ ranging over }(\Z/2)^\ZO.
\]
We may thus compute $\eta(t)$ as $\lim\eta(t_i)$ as long as the $t_i$ agree with $t$ on ever larger right-infinite intervals.

\begin{proof}
  We first show that indeed $\eta(t)$ belongs to $\Omega_\leftarrow$
  for every $t\in (\Z/2)^\ZO$. If $\eta_g$ then $g = (s,n)$ with
  $s_{>n} = t_{>n}$. There is then a unique $h \in \{a^{-1}, b^{-1}\}$
  such that $\eta_{g h}$, namely $h = a^{-1}$ if
  $s_{n-1/2} = t_{n-1/2}$ and $h=b^{-1}$ otherwise. On the other hand,
  $\eta_g$ implies $\eta_{g a}$ and $\eta_{g b}$ because
  $g a = (s,n+1)$ and $g b = (s',n+1)$ with
  $s'_{>n+1} = s_{>n+1} = t_{>n+1}$; so both equations defining
  $\Omega_\leftarrow$ are satisfied.

  We next prove that $t\mapsto\eta(t)$ has image
  $\Omega_\leftarrow$. Consider $\eta\in\Omega_\leftarrow$; if
  $\eta = \false^L$ choose any $t \in (\Z/2)^\ZO$ with infinitely many
  $1$s in its right tail. Otherwise, let $g\in L$ be in the support of
  $\eta$. By the rule `$\gamma\vee\delta\implies\alpha\neq\beta$',
  precisely one of $g a^{-1}$ and $gb^{-1}$ is in the support of
  $\eta$. Following this path from $g$, we get a sequence of symbols
  $w_1,w_2,\dots\in \{a^{-1}, b^{-1}\}$ such that $g w_1\cdots w_n$ is
  in the support of $\eta$ for all $n$. The sequence $g w_1 w_2\cdots$
  converges to a configuration $t\in(\Z/2)^\ZO$. Observe that the head
  only moves to the left, so $t_i=0$ for all $i$ large enough.

  We show that the $t$ we just constructed is indeed a preimage of
  $\eta$. Consider first some $h=(s,n)\in L$ such that
  $s_{>n} = t_{>n}$; we prove $\eta_h$. By definition of $t$ we have
  $\eta_{g w_1\cdots w_m}$ for some $w_i\in\{a^{-1},b^{-1}\}^*$ and
  $m$ arbitrarily large, with $g w_1\cdots w_m=(t',n')$ and
  $t'_{>n'} = t_{>n'}$; so $t'_{>n} = t_{>n}$ as soon as $n'\le n$. By
  the rule `$\alpha\vee\beta\implies\gamma\wedge\delta$' we have
  $\eta_{g w_1\cdots w_m u}$ for all $u \in \{a,b\}^*$. If furthermore
  $m$ is large enough that $n'$ is smaller than all elements in the
  support of $s$, then $h\in g w_1\cdots w_m\{a, b\}^*$ implying
  $\eta_h$.

  Consider next $h=(s,n)\in L$ such that $s_{>n} \neq t_{>n}$, so
  $s_i \neq t_i$ for some $i > n$. Define $s'\in(\Z/2)^\ZO$ by
  $s'_{<n}\coloneqq s_{<n}$ and $s'_{>n}\coloneqq t_{>n}$; so
  $s'_i\neq t_i$ as before. By the rule
  `$\alpha\vee\beta\implies \gamma\wedge\delta$' we have
  $\eta_{c(s',m)}$ for all $m \geq n$. Furthermore,
  $h\in(s',m)\{a^{-1}, b^{-1}\}^{m-n}$ for all $m$ large enough. Now
  by induction, using the rule
  `$\gamma\vee\delta\implies\alpha\neq\beta$', whenever $\eta_f$ there
  is for all $m \geq 0$ a unique $w\in\{a^{-1}, b^{-1}\}^m$ such that
  $\eta_{f w}$. Thus we cannot have simultaneously $\eta_{(s'n)}$ and
  $\eta_{(s,n)}$, so $\neg\eta_h$.

  Consider $t\in(\Z/2)^\ZO$. If $t_{>0}$ is infinitely supported, then
  $\eta(t)=\false^L$ since $t_{>n}$ can never agree with $s_{>n}$ for
  $(s,n)\in L$. Assume then that $t_{>0}$ is finitely
  supported. Defining $s\in(\Z/2)^\ZO$ by $s_{<0}=0$ and
  $s_{>0}=t_{>0}$ we get $\eta(t)_{s,0}$ so
  $\eta(t)\neq\false^L$. Consider next $t'\neq t$ such that $t'_{>0}$
  is also finitely supported, and let $n$ be such that
  $t'_{n+1/2}\neq t_{n+1/2}$. Define $(s,n)\in L$ by $s_{<n}=0$ and
  $s_{>n}=t_{>n}$; then $\eta(t)\neq\eta(t')$ because
  \[\eta(t)_{(s,n)}\implies s_{>n}=t_{>n}\implies s_{>n}\neq t'_{>n} \implies \neg\eta(t')_{(s,n)}.\]

The $L$-equivariance of $\eta$ is easily checked; it amounts to
checking $a\eta(t) = \eta(t')$ with $t'(i)=t(i-1)$ and
$a b^{-1}\eta(t) = \eta(t'')$ with $t''_i = t_i$ for $i \neq 1/2$ and
$t''_{1/2} = 1 - t_{1/2}$.

We finally prove that $\Omega_\leftarrow$ is almost minimal. If
$\eta(t), \eta(u)$ are different from $\false^L$, then for every
$m \in \N$ we can find $g \in L$ such that $g\eta(t) = \eta(t')$ with
$t'_{>-m}=u_{>-m}$. In this manner we can make an arbitrarily large
central portion of $\eta(t')$ equal to that of $\eta(u)$, so
$L\eta(t)$ approaches $\eta(u)$ arbitrarily closely. The fixed point
$0=\false^L$ is approached as a limit of $\eta(t)$ with $t_{>0}$
having support of size $\to\infty$.
\end{proof}

We symmetrically define the SFT $\Omega_\rightarrow$ by switching
the roles of left and right; so $\Omega_\rightarrow$ is also almost
minimal, and we have
\[\Omega_\rightarrow=\left\{\eta\in\mathbb B^L: \exists t \in (\Z/2)^\ZO:\;\forall g=(s,n) \in L: \eta_g \iff s_{<n} = t_{<n}\right\}.
\]
We next combine these two SFTs by a product construction:
\[\Omega_\leftrightarrow\subset\Omega_\leftarrow\times\Omega_\rightarrow\]
consists in configurations $\eta$ such that, writing
$(\eta_g,\eta_{g a b^{-1}},\eta_{g a},\eta_{g
  b})=(\alpha,\beta,\gamma,\delta)$, we have
$\alpha=(\true,\true)\iff\gamma=(\true,\true)$.

\begin{lem}\label{lem:aUnique}
  The SFT $\Omega_\leftrightarrow$ is the orbit closure of
  $\eta \in(\mathbb B\times\mathbb B)^L$ defined by
  \[\eta_{(s,n)} = (s_{>n}\equiv 0^{>n}, s_{<n}\equiv 0^{<n}).\]
  The configuration $\eta$ is the only configuration in
  $\Omega_\leftrightarrow$ satisfying $\eta_1 = (\true,\true)$.
\end{lem}
\begin{proof}
  The configuration $\eta$ belongs to $\Omega_\leftrightarrow$: its
  first projection is in $\Omega_\leftarrow$ by
  Lemma~\ref{lem:XLCharacterization} with $t = 0^\ZO$, and
  symmetrically its right projection is in $\Omega_\rightarrow$. On
  the other hand $\eta_g = (\true,\true)$ happens precisely when
  $g=(0^\ZO,n)$ for some $N\in\Z$, so $g\in\langle a\rangle$.

  Consider now an arbitrary configuration
  $\zeta\in\Omega_\leftrightarrow$. Suppose first
  $\zeta = (\false,\false)^L$; then $\zeta$ is in the orbit closure of
  $\eta$, since arbitrarily large $(\false,\false)$-balls are seen
  around elements of the form $(s,n)$ in which $s$ has many $1$s in
  its left and right tails. The set of configurations
  $\zeta \in \Omega_\leftrightarrow$ with second projection $\false^L$
  is precisely $\Omega_\leftarrow \times\{\false^L\}$. To reach these
  configurations in the orbit closure of $\eta$ it suffices to find
  only one of them, by the almost minimality of
  $\Omega_\leftarrow$. Now clearly if $s_{>0} = 0^{>0}$ and
  $s_{-m-1/2} \neq 0$ for some $m \geq 0$, then
  $((s,0)\cdot\eta)_1 = (\true,\false)$ and the second projection of
  $(s,0)\cdot\eta$ tends to $\false^L$ as $m\to\infty$. Thus indeed
  $\Omega_\leftarrow \times \{\false^L\}$ is contained in the orbit
  closure of $\eta$. Similarly, $\{\false^L\}\times\Omega_\rightarrow$
  is contained in the orbit closure of $\eta$.

  Consider then $\zeta \in \Omega_\leftrightarrow$ whose projections
  are both $\neq\false^L$. First, we have $\zeta_g = (\true,\true)$
  for some $g\in L$: indeed, suppose $\zeta_{g_1} = (\true,\false)$
  and $\zeta_{g_2} = (\false,\true)$. Then the rules force
  $\zeta_{g_1 v_1} = (\true,*)$ for all $v_1 \in \{a, b\}^*$ and
  $\zeta_{g_2 v_2}=(*,\true)$ for some $v_2 \in \{a, b\}^m$ and any
  $m \geq 0$; and symmetrically $\zeta_{g_1 v_1 w_1} = (\true,*)$ for
  some $w_1\in\{a^{-1},b^{-1}\}^n$ and all $n\ge0$ while
  $\zeta_{g_2 v_2 w_2}=(*,\true)$ for all
  $w_2\in\{a^{-1},b^{-1}\}^*$. Now every element of $L$, in particular
  $g_1^{-1}g_2$, may be written in the form $v_1 w_1 w_2^{-1}v_2^{-1}$
  with $v_2,w_1$ fixed, as soon as they are long enough (depending on
  the support of $g_1,g_2$); so we may set $g=g_1v_1w_1=g_2v_2w_2$ and
  note $\zeta_g=(\true,\true)$. By replacing $\zeta$ by a translate,
  we may assume $\zeta_1 = (\true,\true)$ and it now enough to prove
  $\zeta = \eta$. By the rule
  `$\alpha = (\true,\true) \iff \gamma = (\true,\true)$', both $\eta$
  and $\zeta$ contain $(\true,\true)$ only on the subgroup
  $\langle a \rangle$. Now any configuration in $\Omega_\leftarrow$
  (respectively $\Omega_\rightarrow$) is determined by a bi-infinite
  $\{a,b\}$-labelled path of $\true$s, by
  Lemma~\ref{lem:XLCharacterization}, so $\zeta = \eta$.

  The last claim holds because $a \cdot \eta = \eta$ and $\eta$
  contains $(\true,\true)$ only on the coset $\langle a \rangle$.
\end{proof}

Generalizing the ``sunny-side-up'' shift, it would be interesting to
understand \emph{which subgroups can be marked by a sofic shift}. By
this we mean the following: consider $G$ a countable group and
$H\le G$ a subgroup; then $G$ acts on the space $G/H=\{g H:g\in G\}$
of left cosets of $H$ by left translation, and this action extends to
the one-point compactification $G/H \cup \{\infty\}$ giving it the
structure of an expansive zero-dimensional topological dynamical
system.  Thus it is abstractly a subshift, which we call the
\emph{$H$-coset subshift}.  When $H$ is the trivial group, the
$H$-coset subshift is the sunny-side-up; it is still not understood
for which groups $G$ it is sofic. Now mapping $(\true,\true)$ to $1$
and everything else to $0$ produces a subshift of $\{0,1\}^L$, for
which we have the
\begin{prop}\label{prop:aSofic}
  The $\langle a \rangle$-coset subshift on the lamplighter group is
  sofic.\qed
\end{prop}

It is straightforward to superpose any sofic $\Z$-shift on the
$\langle a \rangle$-coset, giving also the following corollary:
\begin{prop}\label{prop:SSUSofic}
  The sunny-side-up subshift on the lamplighter group is sofic.\qed
\end{prop}

\begin{figure}
  \centerline{\begin{tikzpicture}[x={(5cm,0cm)},y={(-2cm,-1cm)},z={(0cm,6cm)}]
    \pgfmathsetmacro\h{4}
    \pgfmathsetlengthmacro\rad{3mm*pow(2,-\h)}
    \pgfmathsetmacro\hmo{\h-1}
    \foreach\k in {0,...,\h} {
      \pgfmathsetmacro\z{\k/\h}
      \pgfmathsetmacro\cz{1-\z}
      \pgfmathsetmacro\maxl{pow(2,\k)-1}
      \pgfmathsetmacro\maxm{pow(2,\h-\k)-1}
      \ifnum\k=2
        \foreach\x in {-0.5,-0.166,0.166,0.5} {
          \draw[blue,very thick] (\x,-0.6,\z) -- (\x,0.6,\z);
          \draw[blue,very thick] (-0.6,\x,\z) -- (0.6,\x,\z);
        }
      \fi
      \foreach\m in {0,...,\maxm} {
        \pgfmathsetmacro\y{-1+\m*2/(pow(2,\h-\k)-0.9999)}
        \foreach\l in {0,...,\maxl} {
          \pgfmathsetmacro\x{1-\l*2/(pow(2,\k)-0.9999)}
          \ifnum\k=\h\else
          \foreach\color/\pos in {green/1,red/0} {
            \pgfmathsetmacro\nx{1-(2*\l+(\pos ? mod(\m,2) : 1-mod(\m,2)))*2/(pow(2,\k+1)-1)}
            \pgfmathsetmacro\ny{-1+floor(\m/2)*2/(pow(2,\h-(\k+1))-0.9999)}
            \pgfmathsetmacro\nz{(\k+1)/\h}
            \pgfmathsetmacro\ncz{1-\nz}
            \pgfmathsetmacro\shade{30+20*\y}\def\shadecolor{\color!\shade}
            \pgfmathsetmacro\nshade{30+20*\ny}\def\nshadecolor{\color!\nshade}
            \ifnum\m=\maxm\ifnum\l=\maxl\ifnum\pos=1
            \def\shadecolor{green!80!black}
            \def\nshadecolor{green!80!black}
            \fi\fi\fi
            \shadepath{\shadecolor}{\nshadecolor}{\z*\x,\cz*\y,\z}{\nz*\nx,\ncz*\ny,\nz}
          }
          \fi
          \filldraw[fill=white] (\z*\x,\cz*\y,\z) circle (1.5pt);
        }
      }
    }
  \end{tikzpicture}}
  \caption{A tetrahedron in $\gfL$, with in blue the ``sea level'' grid and in heavy green the marked ray $\langle a\rangle$}\label{fig:sealevel}
\end{figure}
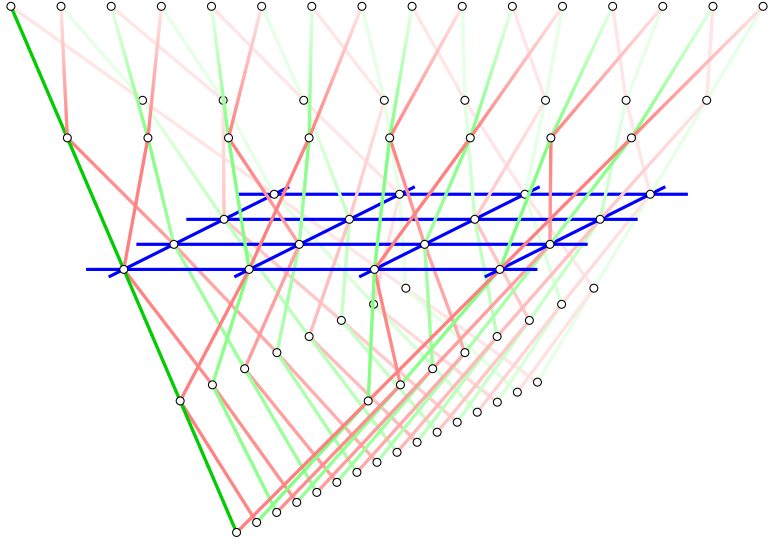

\newcommand{\LLr}[3]{{#1\markin#2\markout#3}}
\newcommand{\LLc}[2]{{#1\markinout#2}}
\newcommand{\LLl}[3]{{#1\markout#2\markin#3}}
\subsection{Unsynchronized binary trees over the sea surface}
We define in this section a ``sea level'' SFT: it will mark one level
of $\gfL$, namely all $(s,n)\in L$ for some fixed $n\in\Z$, by a
``sea level'' symbol $\aquarius$, and  mark by ``above (respectively below) sea level'' symbols all elements $(s,n')$ with $n'>n$ (respectively $n'<n$).
The ``sea level'' appears as a grid in
Figure~\ref{fig:sealevel}.

Additionally, the SFT will mark some binary
trees in the ``above sea level'' portion, that connect columns of the
grid together, as well as binary trees in the ``below sea level''
portion connecting rows of the grid. The SFT will be defined by
allowed tetrahedra; we introduce
\[U = \{\nwarrow, \uparrow, \nearrow\},\qquad D = \{\swarrow,
  \downarrow, \searrow\},\qquad A = U \cup \{\aquarius\} \cup D,\]
define $\phi\colon A \to\{1,0,-1\}$ by $\phi(U)=\{1\}$,
$\phi(\aquarius) = 0$, $\phi(D) = \{-1\}$, and define
\[\Omega_{\aquarius}\subset A^L\]
as those $\eta\in A^L$ such that, for all $g\in L$, the tetrahedron
$(\eta_g,\eta_{g a b^{-1}},\eta_{g a},\eta_{g
  b})=(\alpha,\beta,\gamma,\delta)$ satisfies
\begin{gather}
  \phi(\alpha) = \phi(\beta)\text{ and }\phi(\gamma) = \phi(\delta)\tag{\aquarius.1}\\
  \phi(\gamma) - \phi(\alpha)\in\{0,1\}\tag{\aquarius.2}\\
  \alpha = \aquarius \implies \{\gamma, \delta\} = \{\nwarrow, \nearrow\}\tag{\aquarius.3}\\
  \gamma = \aquarius \implies \{\alpha, \beta\} = \{\swarrow, \searrow\}\tag{\aquarius.4}\\
  \alpha\in U\implies \beta=\alpha\wedge\{\gamma, \delta\} = \{\uparrow,\alpha\}\tag{\aquarius.5}\\
  \gamma\in D\implies \delta=\gamma\wedge\{\alpha, \beta\} = \{\downarrow,\gamma\}\tag{\aquarius.6}
\end{gather}

For a finite word $v$ and an infinite (right or left) word $u$, we
write $v\sqsubset u$ if $v$ is at the extremity (prefix, suffix) of
$u$. We recall our notation `$u\markin v\markout w$' for elements of
the lamplighter group, introduced in~\S\ref{ss:ll}; in particular the
identity is written `$\markinout$', and `$u\markinout v$' corresponds
to elements $(s,0)$. This will be the most convenient notation to
describe the structure of $\Omega_{\aquarius}$.

The following result shows that, in every tiling
respecting~(\aquarius.1--\aquarius.6), and containing the symbol
$\aquarius$, there is a ``sea level'', namely an infinite grid of
$\aquarius$'s; and some binary trees attached above and below it, in
directions specified by some rays $s_{u,\nwarrow}$, $s_{u,\nearrow}$,
$s_{u,\swarrow}$, $s_{u,\searrow}$:
\begin{lem}\label{lem:StarGivesBinaryTrees}
  Consider $\eta \in A^L$ and suppose $\eta_{\LLc{}{}} = \aquarius$. Then
  $\eta \in \Omega_{\aquarius}$ if and only if the following holds:
  \begin{itemize}
  \item $\eta_{\LLc{u}{v}} = \aquarius$ for all $u, v \in (\Z/2)^*$;
  \item for all $u \in (\Z/2)^*$ there exist $s_{u,\nwarrow}$ and
    $s_{u,\nearrow}$ in $(\Z/2)^\N$ with
    $(s_{u,\nwarrow})_0 \neq (s_{u,\nearrow})_0$, and $s_{u,\swarrow}$
    and $s_{u,\searrow}$ in $(\Z/2)^{-\N}$ with
    $(s_{u,\swarrow})_0 \neq (s_{u,\searrow})_0$, such that
    \begin{align*}
      \eta_{\LLr{u}{v}{w}} = {\nwarrow} &\text{ if $|v| > 0$ and }v\sqsubset s_{u,\nwarrow},\\
      \eta_{\LLr{u}{v}{w}} = {\nearrow} &\text{ if $|v| > 0$ and }v\sqsubset s_{u,\nearrow},\\
      \eta_{\LLr{u}{v}{w}} = {\uparrow} &\text{ if $|v| > 0$ and the previous two cases do not apply,}\\
      \eta_{\LLl{w}{v}{u}} = {\swarrow} &\text{ if $|v| > 0$ and }v\sqsubset s_{u,\swarrow},\\
      \eta_{\LLl{w}{v}{u}} = {\searrow} &\text{ if $|v| > 0$ and }v\sqsubset s_{u,\searrow},\\
      \eta_{\LLl{w}{v}{u}} = {\downarrow} &\text{ if $|v| > 0$ and the previous two cases do not apply.}
    \end{align*}
  \end{itemize}
  A configuration $\eta$ with $\eta_{\LLc{}{}} = \aquarius$ is fully
  determined by the collection, for all $u\in(\Z/2)^*$, of the words
  $s_{u,\nwarrow},s_{u,\nearrow} \in (\Z/2)^\N$ and
  $s_{u,\swarrow},s_{u,\searrow} \in (\Z/2)^{-\N}$.
\end{lem}
\begin{proof}
  It is straightforward to verify that a configuration $\eta$
  satisfying the above for some choices
  $s_{u,\nwarrow},s_{u,\nearrow} \in (\Z/2)^\N$ and
  $s_{u,\swarrow},s_{u,\searrow} \in (\Z/2)^{-\N}$ belongs to
  $\Omega_{\aquarius}$.

  Suppose now $\eta_{\LLc{u}{}} = \aquarius$ for some
  $u \in (\Z/2)^*$. Rules~(\aquarius.3) and~(\aquarius.5) imply, by
  induction on $n$, that for some choice $c \in\Z/2$ the pattern
  $\eta\restriction_{\LLr{u}{c(\Z/2)^n}{}}$ contains exactly one
  $\nwarrow$ and otherwise only $\uparrow$'s, and
  $\eta\restriction_{\LLr{u}{(1-c)(\Z/2)^n}{}}$ contains exactly one
  $\nearrow$ and otherwise only $\uparrow$'s.

  We then show by induction on $k$ that
  $\eta\restriction_{\LLr{u}{c(\Z/2)^{n-k}}{w}}$ also contains exactly
  one $\nwarrow$ and otherwise only $\uparrow$ for all
  $w \in (\Z/2)^k$. First, observe that $\aquarius$ cannot appear in
  this set, as it would imply a $\nearrow$ in
  $\eta\restriction_{\LLr{u}{c(\Z/2)^n}{}}$. Thus, by~(\aquarius.2),
  no elements of $D$ can appear either. But as long as all symbols in
  $\eta\restriction_{\LLr{u}{c(\Z/2)^{n-k}}{w}}$ are in $U$, it is
  clear from~(\aquarius.5) that in fact the symbol at
  $\eta_{\LLr{u}{c v}{w}}$ for $v \in (\Z/2)^{n-k}$ is uniquely
  determined by the counts of symbols $\nwarrow, \uparrow, \nearrow$
  in
  $\eta\restriction_{\LLr{u}{c v}{w} \cdot \{a, b\}^k} =
  \eta\restriction_{\LLr{u}{c v(\Z/2)^k}{}}$. We conclude that the
  symbol $\eta_{\LLr{u}{c v}{w}}$ is independent of $w$, and a
  symmetric claim holds for $\eta\restriction_{\LLr{u}{(1-c)v}{w}}$.

  As we observed two paragraphs above, $\eta_{\LLr{u}{c(\Z/2)^n}{}}$
  contains for all $n\in\N$ exactly one $\nwarrow$ and otherwise only
  $\uparrow$'s. More precisely, (\aquarius.5) implies that there is a
  unique path $s_{u,\nwarrow}$ such that
  $\eta_{\LLr{u}{c v}{}} = {\nwarrow}$ precisely when
  $v\sqsubset s_{u,\nwarrow}$. The previous paragraph then shows for
  all $w\in(\Z/2)^*$ that $\eta_{\LLr{u}{cv}{w}} = {\nwarrow}$ holds
  precisely when $v\sqsubset s_{u,\nwarrow}$, with
  $\eta_{\LLr{u}{cv}{w}} = {\uparrow}$ otherwise. Symmetrically, there
  is a unique path $s_{u,\nearrow}$ such that
  $\eta_{\LLr{u}{(1-c)v}{w}} = {\nearrow}$ if and only if
  $v \sqsubset s_{u,\nearrow}$, with
  $\eta_{\LLr{u}{cv}{w}} = {\uparrow}$ otherwise. Clearly
  only~(\aquarius.3) can apply at $\LLc{u}{w}$, and we get
  $\eta_{\LLc{u}{w}} = \aquarius$ for all $w \in (\Z/2)^*$.

  The rules for $\swarrow, \downarrow, \searrow$ are symmetric to
  those for $\nwarrow, \uparrow, \nearrow$, so if
  $\eta_{\LLc{}{u}} = \aquarius$ then there exist
  $s_{u,\swarrow}, s_{u,\searrow}$ differing at $0$ such that
  $\eta_{\LLl{w}{v}{u}} = \searrow$ if and only if
  $v \sqsubset s_{u,\searrow}$ and $\eta_{\LLl{w}{v}{u}} = {\swarrow}$
  if and only if $v \sqsubset s_{u,\swarrow}$, with
  $\eta_{\LLl{w}{v}{u}} = {\downarrow}$ for all other non-empty $v$,
  and thus $\eta_{\LLl{w}{}{u}} = \aquarius$ for all $w \in (\Z/2)^*$.

  Now suppose $\eta_{\LLc{}{}} = \aquarius$. From the above we obtain
  $\eta_{\LLc{u}{v}} = \aquarius$ for all $u,v \in (\Z/2)^*$; and the
  previous analysis applied to $\eta_{\LLc{u}{}} = \aquarius$ and
  $\eta_{\LLc{}{v}} = \aquarius$ proves that $\eta$ has the claimed
  form.

  Finally, the claim that $\eta$ is fully determined by
  $s_{u,\nwarrow}, s_{u,\nearrow} \in (\Z/2)^\N$ and
  $s_{u,\swarrow}, s_{u,\searrow} \in (\Z/2)^{-\N}$ directly follows
  from the given construction of $\eta$'s values.
\end{proof}

Just as for Lemma~\ref{lem:aUnique}, the previous lemma can be stated
in terms of coset subshifts. Indeed, mapping $\aquarius$ to $1$ and
everything else to $0$, we obtain the following
\begin{prop}\label{prop:sealevelSofic}
  Let $\phi\colon L \to \Z$ be the homomorphism given by
  $\phi(a) = \phi(b) = 1$.  Then the $\ker \phi$-coset subshift on the
  lamplighter group is sofic.\qed
\end{prop}

More generally, let $G$ be a group, let $\phi\colon G \to H$ be a
homomorphism, and let $X\subset A^H$ be an $H$-subshift. The
\emph{pullback} of $X$ along $\phi$ is the subshift
\[\phi^*(X) = \{ y \in A^G : \exists x \in X:\; \forall g \in G: y_g = x_{\phi(g)} \}.\]
Rephrasing Proposition~\ref{prop:sealevelSofic} in these terms, we get
the following:
\begin{prop}\label{prop:sealevelSofic2}
  Let $\phi\colon L \to \Z$ be the homomorphism given by
  $\phi(a) = \phi(b) = 1$.  Then the pullback of the sunny-side-up
  subshift from $\Z$ along $\phi$ is sofic.\qed
\end{prop}

We do not know whether the pullback of a full shift on $\Z$ is sofic.

\subsection{Synchronizing the trees}
We now impose some extra conditions on $\Omega_{\aquarius}$ to
synchronize the marked directions, namely to force the binary trees
above and below the sea level to lie in specific directions. This is
done by combining $\Omega_\leftrightarrow$ with $\Omega_{\aquarius}$.
Define thus
\[\Omega\subset\Omega_\leftrightarrow \times \Omega_{\aquarius}\]
as those $\eta\in\Omega_\leftrightarrow \times \Omega_{\aquarius}$
such that, for all $g \in L$, the tetrahedron
$(\eta_g, \eta_{g a b^{-1}}, \eta_{g a}, \eta_{g b}) =
(\alpha,\beta,\gamma,\delta)$ satisfies
\begin{align*}
  \alpha=((\true,*),\aquarius) &\implies \gamma=(*,\nwarrow)\wedge\delta=(*,\nearrow),\\
  \alpha=((\true,*),\nwarrow) &\implies \gamma=(*,\nwarrow),\\
  \alpha=((\true,*),\nearrow) &\implies \delta=(*,\nearrow),\\
  \gamma=((*,\true),\aquarius) &\implies \alpha=(*,\swarrow)\wedge\beta=(*,\searrow),\\
  \gamma=((*,\true),\swarrow) &\implies \alpha=(*,\swarrow),\\
  \gamma=((*,\true),\searrow) &\implies \beta=(*,\searrow).
\end{align*}

This combination of $\Omega_\leftrightarrow$ and $\Omega_{\aquarius}$
vastly reduces the size of the SFT $\Omega$: more precisely,
\begin{lem}\label{lem:omega}
  There is a unique configuration $\eta\in\Omega$ satisfying
  $\eta_{\LLc{}{}} = ((\true,\true), \aquarius)$. The projection to
  $\Omega_\leftrightarrow$ of $\eta$ is the one described by
  Lemma~\ref{lem:aUnique}, and its projection to $\Omega_{\aquarius}$
  is given by the choices $s_{u,\nwarrow}=0^\N$,
  $s_{u,\nearrow}=1^\N$, $s_{u,\swarrow}=0^{-\N}$ and
  $s_{u,\searrow}=1^{-\N}$, namely we have
  \begin{align*}
    \eta_{\LLr uvw} &= \begin{cases}
      (*,\nwarrow) & \text{if }v\in0^+,\\
      (*,\nearrow) & \text{if }v\in1^+,\\
      (*,\uparrow) & \text{in all remaining cases},
    \end{cases}\\
    \eta_{\LLc uv} &= ((\true,\true),\aquarius),\\
    \eta_{\LLl wvu} &= \begin{cases}
      (*,\swarrow) & \text{if }v\in0^+,\\
      (*,\searrow) & \text{if }v\in1^+,\\
      (*,\downarrow) & \text{in all remaining cases}.
    \end{cases}
  \end{align*}
\end{lem}
\begin{proof}
  We first show that the given $\eta$ belongs to $\Omega$. Our choices
  $s_{u,\nwarrow} = 0^\N, s_{u,\nearrow} = 1^\N, s_{u,\swarrow} =
  0^{-\N}, s_{u,\searrow} = 1^{-\N}$ for all applicable $u$ show that
  the second projection of $\eta$ satisfies the characterization of
  Lemma~\ref{lem:StarGivesBinaryTrees}. With
  $(\eta_g, \eta_{g a b^{-1}}, \eta_{g a}, \eta_{g b}) =
  (\alpha,\beta,\gamma,\delta)$, we have $\alpha=((\true,*),*)$ if and
  only if $g = \LLr{u}{v}{}$ or $g = \LLl{u}{0^n}{}$ for some $u, v, n$,
  since the first projection of $\eta$ is the configuration described
  by Lemma~\ref{lem:aUnique}.

  We now show that the rules joining $\Omega_\leftrightarrow$ and
  $\Omega_{\aquarius}$ are satisfied by $\eta$. If
  $\alpha=(*,\aquarius)$, then $g = \LLc{u}{v}$ for some $u, v$; force
  $|v| > 0$ by appending $0$ if needed. If furthermore
  $\alpha=((\true,*),\aquarius)$, then by Lemma~\ref{lem:aUnique} we
  may write $v=0v'$ since $v = 0^{|v|}$, and then indeed we have a
  `$\nwarrow$' at $(\LLc{u}{v}) \cdot a = \LLr{u}{0}{v'}$ and a
  `$\nearrow$' at $(\LLc{u}{v}) \cdot b = \LLr{u}{1}{v'}$, as
  required. The verifications when $\alpha=(*,\nwarrow)$ or
  $(*,\nearrow)$ are similar: by Lemma~\ref{lem:aUnique} if
  $\alpha=((\true,*),*)$ then the bit after the head in $g$ must be
  $0$, and therefore a `$\nearrow$' (respectively a `$\nwarrow$')
  appear in the correct neighbour. The verifications for $\searrow$
  and $\swarrow$ are symmetric, because $\alpha=((*,\true),*)$ implies
  that the bit to the left of the head is $0$.

  Next, we show that $\eta$ is the unique configuration in $\Omega$
  seeded at $\LLc{}{}$. Let $\zeta \in \Omega$ be any configuration
  with $\zeta_{\LLc{}{}} = ((\true,\true),\aquarius)$. By definition
  of $\Omega$, the first projection of $\zeta$ is the unique
  configuration of $\Omega_\leftrightarrow$ described in
  Lemma~\ref{lem:aUnique}, and its second projection is one of the
  configurations of $\Omega_{\aquarius}$ described by
  Lemma~\ref{lem:StarGivesBinaryTrees}, say by words
  $s_{u,\nwarrow},s_{u,\nearrow} \in (\Z/2)^\N$ and
  $s_{u,\swarrow}, s_{u,\searrow} \in (\Z/2)^{-\N}$. We claim that
  these words must be the ones definining $\eta$, which will imply
  $\zeta = \eta$ because by Lemma~\ref{lem:StarGivesBinaryTrees} these
  words fully determine the configuration.

  Suppose for example $s_{u,\nwarrow} \neq 0^\N$ for some
  $u \in (\Z/2)^*$. This means $\zeta_{\LLr{u}{v}{}} = {\nwarrow}$ for
  some $|v| > 0$ and $v \not\sqsubset 0^{\N}$. Let $v$ be of minimal
  length with this property, so $\zeta_{\LLr{u}{0^k1}{}} = {\nwarrow}$
  for some $k \geq 0$. We then have
  $\zeta_{\LLr{u}{0^k}{}} \in \{\aquarius, \nwarrow\}$, so in its
  first projection $\zeta_{\LLr{u}{0^k}{}} = ((\true,*),*)$; then the
  additional rules of $\Omega$ require
  $\zeta_{\LLr{u}{0^k0}{}} = (*,\nwarrow)$. Now this contradicts the
  defining rules of $\Omega_{\aquarius}$ at $\LLr{u}{0^k}{}$ since
  neither $\aquarius$ nor $\nwarrow$ can have $\nwarrow$ at both its
  $a$- and the $b$-neighbour. Thus $s_{u,\nwarrow} = 0^\N$ is the only
  possibility. The verifications for $s_{u,\nearrow} = 1^\N$,
  $s_{u,\swarrow} = 0^{-\N}$ and $s_{u,\searrow} = 1^{-\N}$ are
  symmetric.
\end{proof}

\subsection{Simulating \boldmath $\N\times\N$ on $\gfL$ marked by $\Omega$}
\label{ss:N2Simu}
Consider the graph $\gfL$ marked by the SFT $\Omega$ defined in the
previous sections; more precisely, imposing the seed constraint that
the origin in $\gfL$ is labelled $((\true,\true),\aquarius)$, we have
by Lemma~\ref{lem:omega} uniquely specified a vertex labelling of
$\gfL$. Its edges retain the Cayley graph labelling by
$\{a^{\pm1},b^{\pm1}\}$. We claim:
\begin{prop}\label{prop:seasimul}
  The graph $\gfL$ labelled by the configuration $\eta\in\Omega$ from
  Lemma~\ref{lem:omega} simulates $\N\times\N$.
\end{prop}
\begin{proof}
  It suffices to produce a simulator. Here it is, in simplified form:
  
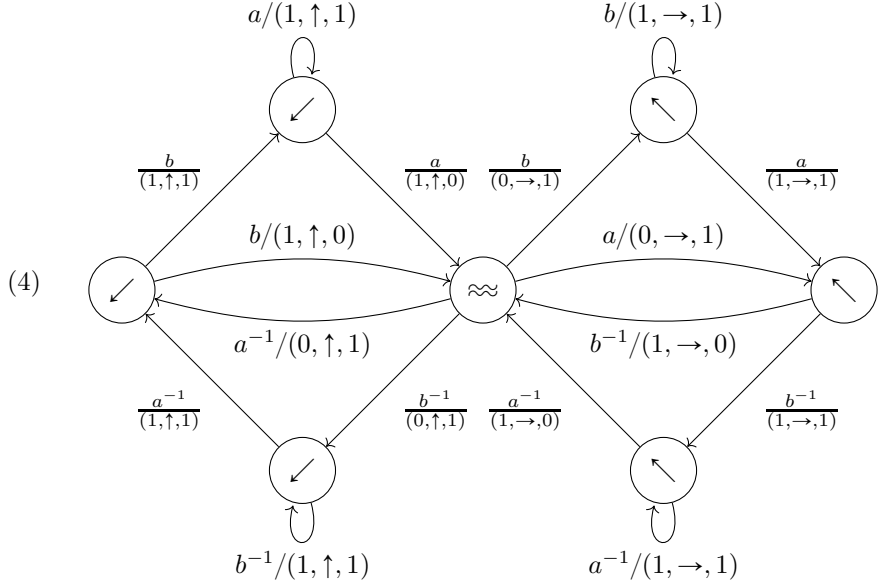
\begin{equation}\label{eq:seasim}
    \begin{tikzpicture}[scale=1.2,auto,every state/.style={inner sep=0mm},baseline=0]
      \node[state] (sea) at (0,0) {\aquarius};
      \node[state] (up0) at (2,2) {$\nwarrow$};
      \node[state] (up1) at (4,0) {$\nwarrow$};
      \node[state] (up2) at (2,-2) {$\nwarrow$};
      \node[state] (down0) at (-2,-2) {$\swarrow$};
      \node[state] (down1) at (-4,0) {$\swarrow$};
      \node[state] (down2) at (-2,2) {$\swarrow$};
      \path[->] (sea) edge[bend left=15] node[above] {$a/(0,\rightarrow,1)$} (up1)
      (up1) edge[bend left=15] node[below] {$b^{-1}/(1,\rightarrow,0)$} (sea)
      (sea) edge node {$\frac{b}{(0,\rightarrow,1)}$} (up0)
      (up0) edge[loop above] node {$b/(1,\rightarrow,1)$} ()
      (up0) edge node {$\frac{a}{(1,\rightarrow,1)}$} (up1)
      (up1) edge node {$\frac{b^{-1}}{(1,\rightarrow,1)}$} (up2)
      (up2) edge[loop below] node {$a^{-1}/(1,\rightarrow,1)$} ()
      (up2) edge node {$\frac{a^{-1}}{(1,\rightarrow,0)}$} (sea);
      \path[->] (sea) edge[bend left=15] node[below] {$a^{-1}/(0,\uparrow,1)$} (down1)
      (down1) edge[bend left=15] node[above] {$b/(1,\uparrow,0)$} (sea)
      (sea) edge node {$\frac{b^{-1}}{(0,\uparrow,1)}$} (down0)
      (down0) edge[loop below] node {$b^{-1}/(1,\uparrow,1)$} ()
      (down0) edge node {$\frac{a^{-1}}{(1,\uparrow,1)}$} (down1)
      (down1) edge node {$\frac{b}{(1,\uparrow,1)}$} (down2)
      (down2) edge[loop above] node {$a/(1,\uparrow,1)$} ()
      (down2) edge node {$\frac{a}{(1,\uparrow,0)}$} (sea);
    \end{tikzpicture}
  \end{equation}
  The grid $\N^2$ is represented by elements of $\gfL$ at sea level:
  given natural numbers $m,n$, write $m=m_k\ldots m_0$ and
  $n=n_\ell\ldots n_0$ in base $2$; then the grid point $(m,n)\in\N^2$
  is represented by the element
  $n_\ell\ldots n_0\markinout m_0\ldots m_k\in L$.

  The transformation $(m,n)\mapsto(m+1,n)$, corresponding to the
  generator `$\rightarrow$' of $\N^2$, is thus realized by adding $1$
  with carry to the word on the right of the `$\markinout$' mark, and
  similarly $(m,n)\mapsto(m,n+1)$ corresponds to the generator
  `$\uparrow$' and to adding one with carry to the word on the left of
  the `$\markinout$' mark.

  The markings imposed by $\Omega$ on $\gfL$, see
  Lemma~\ref{lem:omega}, force every row of the grid to have a
  distinguished binary tree marked by the symbol `$\nwarrow$' in the
  region above it, and force every column of the grid to have a binary
  tree marked `$\swarrow$' in the region below it. Furthermore, the
  trees are ``synchronized'' by $\Omega_\leftrightarrow$ in such a way
  that the path starting from a vertex $(m,n)$ in the grid and going
  upwards while remaining in the `$\nwarrow$' marked region follows a
  sequence in $\{a,b\}$ reading the binary expansion of $m$, and
  similarly the path going downwards while following `$\swarrow$'
  follows a sequence in $\{a^{-1},b^{-1}\}$ reading the binary
  expansion of $n$.

  The operation of adding one with carry on the word to the right of
  the mark is therefore realized by following the regular expression
  $b^* a b^{-1}(a^{-1})^*$, which is precisely the loop followed on
  the right half of~\eqref{eq:seasim}; and symmetrically the operation
  of adding one to the left of the mark is realized by the regular
  expression $(b^{-1})^*a^{-1}b a^*$, which is the loop followed on
  the left half of~\eqref{eq:seasim}.
  
  To obtain a \emph{bona fide} simulator for $\N^2$, three
  modifications are necessary: firstly, we have only written the
  operations $\uparrow$ and $\rightarrow$; adding reverse edges gives
  the operations $\downarrow$ and $\leftarrow$.

  Secondly, \eqref{eq:seasim} ignores the first two symbols (in
  $\{\true,\false\}$) given to $\gfL$ by $\Omega$; so we should take
  four copies of the diagram, for all possibilities of elements in
  $\{\true,\false\}^2$, and connect them by complete bipartite graphs
  (namely replace every edge $s\to t$ by sixteen edges for all choices
  of $\{\true,\false\}^2$ at source and range).

  Thirdly, \eqref{eq:seasim} does not recognize the vertex markings of
  $\N^2$; so again we should take four copies of the simulator, one
  for each of $\{\NE,\NES,\NEW,\NESW\}$, and connect them
  appropriately: four copies of the `$\uparrow$' circuit on the right
  of~\eqref{eq:seasim} go from $\NE$ to $\NES$, from $\NEW$ to
  $\NESW$, and loop at $\NES$ and $\NESW$, while four copies of the
  `$\rightarrow$' circuit on the left of~\eqref{eq:seasim} go from
  $\NE$ to $\NEW$, from $\NES$ to $\NESW$, and loop at $\NEW$ and
  $\NESW$. This has the effect that the `$\downarrow$' and/or
  `$\leftarrow$' operations are unavailable at $\NES,\NEW,\NE$; and
  indeed the regular expression $a^*ba^{-1}(b^{-1})^*$ does not match
  at $(0,n)$, since it remains stuck on the $a^*$ loop which it reads
  infinitely. See the next section for another remedy.

  It would be tedious to draw the complete simulator, but we have made
  it available in ancillary computer files, in the language
  \textsf{Julia}, using which SFTs on the lamplighter group can be
  explored; see~\S\ref{ss:computer}.
\end{proof}

\subsection{From the quadrant to the plane}
\label{ss:QuadToPlane}
The simulation above implements $\N^2$ inside $\gfL$; negative
coordinates cannot be reached, because the operation $\leftarrow$,
implemented by the regular expression $a^* b a^{-1}(a^{-1})^*$, reads
an infinite string of $a$'s without coming to completion. The cause of
this is that we represented integers in binary, with automata
implementing addition with carry, and in this notation passing from
$0$ to $-1$ causes an infinite sequence of carries.

It is of course possible to simulate $\Z^2$ in $\gfL$ using the fact
that $\N^2$ simulates $\Z^2$ if given suitable markings, see
Example~\ref{ex:simquadrant}. However, a simple change lets us
directly simulate $\Z^2$ in $\gfL$.

It suffices indeed to represent $\Z$ differently than in usual binary:
consider all infinite sequences over $\{0,1\}$ that are confinal to
$(01)^\infty$, namely all sequences $t_0t_1\cdots$ with $t_n=n\bmod 2$
for all $n$ large enough. Then the operations $+1$ and $-1$ may be
performed on such expressions, with the usual rules for carrying and
borrowing bits; and one never encounters an infinite sequence of
carries or borrows. (Abstractly, we are working on the coset $\Z+1/3$
in $\Z_2$.)

This may be realized by making the sequencs $s_{u,\nwarrow}$ etc.\ slightly more complicated: we choose
\[s_{u,\nwarrow}=(01)^\N,\quad s_{u,\nearrow}=(10)^\N,\quad s_{u,\swarrow}=(01)^{-\N},\quad s_{u,\searrow}=(10)^{-\N}.
\]
We omit the details of the construction of corresponding
$\Omega_\leftrightarrow'$ and
$\Omega'\subset\Omega_\leftrightarrow'\times\Omega_{\aquarius}$, which
is only slightly more complicated.

\section{Diestel-Leader graphs}\label{ss:DL}
The lamplighter group can be seen as a special case of a \emph{Diestel-Leader graph}.
As in~\S\ref{ss:cayleyLL}, consider $p,q\ge2$ and two trees
$\mathscr T_p$ and $\mathscr T_q$, respectively $(p+1)$-regular and
$(q+1)$-regular, and endow each with a Busemann function. Let $\DL$
denote their horocyclic product. This is a $(p+q)$-regular graph,
endowed with a graph morphism $h\colon\DL\to\Z$; each vertex
$v=(v_1,v_2)$ at height $n$ has $p$ neighbours $(v'_1,v'_2)$ at height
$n+1$ with $v'_1$ one of the $p$ successors of $v_1$ in $\mathscr T_p$
and $v'_2$ the unique ancestor of $v_2$ in $\mathscr T_q$; and
symmetrically $q$ neighbours at height $n-1$. The remarkable discovery
of Diestel and Leader is that when $p \neq q$, this is a vertex-transitive
graph but is not a Cayley graph: the automorphism group of $\DL$ acts
transitively on vertices, but does not contain a subgroup acting simply
transitively (= with trivial stabilizers).

Vertices of $\DL$ may also be described by sequences as
in~\S\ref{ss:ll}: these are sequences with a marker at an integer
position $n$, elements of $\{0,\dots,q-1\}$ at all half-integer
positions $<n$, almost all $0$, and elements of $\{0,\dots,p-1\}$ at
all half-integer positions $>n$, also almost all $0$.

There is also a notion of tetrahedron for these graphs:
for choices of sequences $u\in\{0,\dots,q-1\}^*$ and
$v\in\{0,\dots,p-1\}^*$, the associated tetrahedron has $p+q$ vertices
`$u\markout i v$' and `$u j\markout v$' for all $i\in\{0,\dots,p-1\}$
and all $j\in\{0,\dots,q-1\}$, and has $p q$ edges between them in a
complete bipartite graph.

We consider $\DL$ with a natural labelling: Every edge in $\DL$
is labelled by $\{0,\dots,p-1\}\times\{0,\dots,q-1\}$: it joins two sequences in
which the marker positions differ by $1$, say `$\cdots\markout i\cdots$'
and `$\cdots j\markout\cdots$', and has label $(i,j)$. There are $p q$
different kinds of vertices, depending on the symbols in the sequence
immediately left and right of the marker; so there are $p q$ different
kinds of immediate neighbourhoods that should be specified in a vertex
SFT.

We note that while this labelling is natural, this is not the
labelling used for the lamplighter group, even when $p = q = 2$.
Indeed, the labelling is not vertex-transitive.



All the constructions from this paper work \emph{mutatis mutandis}
for subshifts of finite type on these labelled Diestel-Leader graphs.

The comb, ray, and sea level SFTs adapt easily to labelled $\DL$
graphs: for instance, since $\DL$ contains $\DL[2,2]$ as the subgraph
spanned by the edges $\{0,1\}\times\{0,1\}$: the tiles can be extended
from $\\DL[2,2]$ to $\DL$ by simply ignoring the colours on the extra
edges. On the other hand, from tiling systems on the lamplighter group
we obtain ones on labelled $\\DL[2,2]$ by observing that mapping edge
labels by $(0,0),(1,1) \mapsto a$ and $(0,1),(1,0) \mapsto b$ maps
$\\DL[2,2]$ onto the lamplighter group.

In the case of the sea level, there is also a natural direct
construction on the Diestel Leader graph $\DL$: There is no
`$\langle a\rangle$' subgroup, but it may be represented by the ray
marked $(0,0)$ through any given seed vertex and corresponds to the
all-off lamp configurations with arbitrary marker position. There is
no `$\ker(\phi)$' sea level subgroup, but it corresponds to the lamp
configurations with marker at $0$. For example, the ``ray subshift''
$\Pi_\leftarrow$ becomes
\begin{equation}
\Pi_\leftarrow=\left\{(\alpha_1,\dots,\alpha_p,\beta_1,\dots,\beta_q)\in \mathbb B^{p+q}:
    \begin{aligned}
      \bigvee\alpha_i&\Longrightarrow\bigwedge\beta_j\\
      \bigvee\beta_j&\Longrightarrow\exists!i:\alpha_i
    \end{aligned}\right\}.
    \label{eq:DLRay}
\end{equation}

\begin{thm}\label{thm:DL}
  For all $p,q \geq 2$, labelled $\DL$ has undecidable seeded tiling
  problem.
\end{thm}

One may wonder if the rigidity of the labelling of $\DL$ makes its
tiling problems easy (to prove undecidable), but this does not seem to
be the case: the labelling is highly recurrent, so one cannot use the
non-rigidity of the tiling to force a (unique) seed to appear with
local rules, and we have not been able to solve the decidability of
unseeded tiling problem on these graphs either.

Note that, if $\DL$ is considered with the trivial labelling (we write
it $\gfD$ from now on to avoid confusion), then its seeded
tiling problem is decidable for uninteresting reasons in the
$\Hom$-formalism. To see this, let $\gfB=\{1,2\}\sqcup\{1,2\}^2$ be
the complete graph with self-loops on two vertices $1,2$, and let
$\beta\colon\gfD\to \gfB$ be the sunny-side-up labelling marking the
origin with $1$. We claim that, given a finite graph
$\gfF \in \Graph_{/\gfB}$, it is decidable whether
$\Graph_{/\gfB}(\gfD, \gfF) = \emptyset$.  Indeed, let $\gfH$ be the
graph with vertex set $\{1,2,3\}$, edges both ways between $i$ and
$i+1$, and labelling induced by $\beta(1) = 1$,
$\beta(2) = \beta(3) = 2$. Then
$\Graph_{/\gfB}(\gfD, \gfF) = \emptyset$ if and only if
$\Graph_{/\gfB}(\gfH, \gfF) = \emptyset$: given
$\phi\colon\gfH\to\gfF$, lift it to $\gfD$ by mapping the origin to
$1$, all other vertices at even height to $3$, and vertices at odd
height to $2$. Conversely, $\gfH$ is a subgraph of $\gfD$.

A better question is the following. Let $G=\operatorname{Aut}(\gfD)$
denote the automorphism group of the unlabelled graph $\gfD =
\DL$. For a finite alphabet $A$, consider the set $A^\gfD$ of maps
$V(\gfD) \sqcup E(\gfD) \to A$, with the natural action of $G$ by
precomposition. A subset of $A^\gfD$ is called an \emph{SFT} if it is
of the form
$\{\eta \in A^\gfD:\; \forall g \in G: \eta\circ g \in C\}$ for some
clopen subset $C \subset A^\gfD$. One defines seeded SFTs as in
Definition~\ref{defn:SeededVertexSFT}, by conditioning on a sunny-side-up.

\begin{conj}\label{conj:DLSeeded}
  The unlabelled graph $\DL$ has undecidable seeded tiling problem.
\end{conj}

\section{Electronic resources to manipulate SFTs on the lamplighter group}\label{ss:computer}
It is quite entertaining to experiment with SFTs on the lamplighter
group; we have written some simple code to help in such experiments.

The Julia module \texttt{LL.jl} should be loaded with
`\verb+include("LL.jl")+' in a recent Julia distribution, including
the packages \texttt{Makie} (for 3D visualization) and
\texttt{CryptoMiniSat} or \texttt{PicoSAT} (to compute tilings of
tetrahedra using a SAT solver). Elements of $L$ are displayed as
sequences over $\{0,1\}$, with an underline or overline between the
origin and the marker position: `$u\markin v\markout w$' is
represented as \texttt{u\underline v\relax w} and `$u\markout v\markin w$'
is represented as \makeatletter\texttt{u$\overline{\hbox{\texttt
      v}}\m@th$w}\makeatother. A sample run could be
\begin{Verbatim}[commandchars=\\\{\}]
julia> include("LL.jl")
julia> root = LL.Element(0,0,3)
\underline{000}
julia> seadict = LL.solve(LL.graph(6),sea,seed=[root=>1]);
julia> LL.walk(root,seaeast,seadict)
\underline{000}1
julia> LL.walk(ans,seanorth,seadict)
\underline{001}1
julia> LL.walk(ans,seawest,seadict)
\underline{001}
julia> LL.walk(ans,seasouth,seadict)
\underline{000}
\end{Verbatim}


\end{document}